\documentclass[12pt]{amsart}
\usepackage[top=72pt,bottom=96pt,left=72pt,right=72pt]{geometry}
\usepackage{amsmath}
\usepackage{mathtools}
\usepackage{amssymb}
\usepackage{mathrsfs}
\usepackage{bbm}
\usepackage{extarrows}

\def\G{\mathcal{G}}
\def\H{\mathcal{H}}
\def\C{\mathbb{C}}

\def\N{\mathbb{N}}
\def\dvol{\mathrm{dvol}}
\def\YM{\mathrm{YM}}

\def\SM{\mathrm{SM}}

\def\YMSM{\mathrm{YMSM}}

\def\qS{\mathscr{S}}

\def\R{\mathbb{R}}
\def\F{\mathfrak{F}}
\def\f{\mathfrak{f}}
\def\qGG{\mathfrak{qGG}}
\def\Hor{\mathrm{Hor}}
\def\Ad{\mathrm{Ad}}
\def\ad{\mathrm{ad}}
\def\id{\mathrm{id}}

\def\Ker{\mathrm{Ker}}
\def\Im{\mathrm{Im}}

\def\H{\mathrm{Hor}}
\def\V{\mathrm{V}}
\def\Mor{\textsc{Mor}}
\def\N{\mathbb{N}}
\def\Obj{\textsc{Obj}}

\def\Rep{\mathbf{Rep}}
\def\triv{\mathrm{triv}}

\def\l{\mathrm{L}}

\def\End{\mathrm{End}}

\def\r{\mathrm{R}}
\def\Vert{\mathrm{Ver}}

\def\T{\mathcal{T}}

\newtheorem{Lemma}{Lemma}[section]
\newtheorem{Remark}[Lemma]{Remark}

\newtheorem{Theorem}[Lemma]{Theorem}
\newtheorem{Corollary}[Lemma]{Corollary}
\newtheorem{Definition}[Lemma]{Definition}
\newtheorem{Proposition}[Lemma]{Proposition}

\begin{document}
\date{\today}
\title{Quantum Principal Bundles and Yang--Mills Scalar Matter Fields}
\author{Gustavo Amilcar Salda\~na Moncada}
\address{Gustavo Amilcar Salda\~na Moncada\\
Mathematics Research Center, CIMAT}
\email{gustavo.saldana@cimat.mx, gamilcar@ciencias.unam.mx}
\begin{abstract}
The aim of this paper is to develop a \emph{non--commutative geometrical} version of the theory of Yang--Mills and space--time scalar matter fields for Riemannian manifolds. To achieve this, we \emph{dualize the geometrical formulation} of the \emph{classical} theory, in which principal $G$--bundles, principal connections, and linear representations play a central role. In addition, we introduce the \emph{non--commutative geometrical action} of the system together with the corresponding \emph{non--commutative geometrical field equations}. Finally, we conclude by presenting two illustrative examples.

 \begin{center}
  \parbox{300pt}{\textit{MSC 2010:}\ 46L87, 58B99.}
  \\[5pt]
  \parbox{300pt}{\textit{Keywords:}\ Yang--Mills Theory, Scalar Matter fields, Quantum Bundles, Quantum  Connections.}
 \end{center}
\end{abstract}
\maketitle 
\section{Introduction}

The Standard Model of Elementary Particles is one of the most successful and important theoretical achievements in modern physics. From a philosophical and mathematical point of view, it provides yet another example of the intrinsic relationship and deep interplay between physics and differential geometry, which in this case is realized through the geometrical framework of principal bundles, their connections, and the associated structures.

Despite these successes, the Standard Model presents some basic and fundamental problems that it cannot address. One prominent example is the lack of a consistent and coherent description of space--time at the Planck scale. The need for further investigation is therefore evident. 

Non--commutative geometry, also known as \emph{quantum geometry}, emerges as an algebraic and physical generalization of \emph{classical} geometrical concepts \cite{con,prug,woro0}. There are several reasons to believe that this branch of mathematics may contribute to resolving some of the fundamental problems of the Standard Model. The reader is encouraged to consult references \cite{con,wit,liz,dmin2} and subsequent works related to models of space--time at the Planck scale formulated in terms of non--commutative algebras.

The purpose of this paper is to develop a \emph{non--commutative geometrical version of the theory of Yang--Mills scalar matter fields} for Riemannian manifolds in the framework of M. Durdevich's formulation of quantum principal bundles and in agreement with references \cite{qvbH,lrz,z,l}. To this end, we \emph{dualize the geometrical formulation} of the \emph{classical} theory, in which principal $G$--bundles, principal connections, and linear representations play a central role. Concretely, we introduce a \emph{non--commutative geometrical action} for the system, together with the corresponding \emph{non--commutative geometrical field equations}. Two illustrative examples are presented at the end of the paper and it is worth mentioning that this work continues the development of the theory initiated in \cite{sald2}.

The importance of this paper lies not only in its geometrical approach, but also in the generality of the theory, which can be applied to a wide class of quantum principal bundles  (\cite{sald2,bdh,micho2}). Furthermore, this work opens the door to a \emph{geometrical} formulation of the Standard Model within the framework of quantum principal bundles and its associated structures, such as spin geometry, as well as to the investigation of Grand Unification Theories, as the reader can see in Section 5.

This work is organized into six sections. After this introduction, Section 2 presents the basic concepts of quantum groups, quantum principal bundles, and \emph{left/right} associated quantum vector bundles. We have deliberately taken the time to provide the reader with a clear and coherent context for Durdevich's formulation of quantum principal bundles, since it is not widely known and the reader will need to become familiar with its essential aspects in order to follow the theory developed in the subsequent sections.

In Section~3, the  novel part of the paper begins. In this section, we develop the theory of the \emph{left/right} quantum Hodge $\star$ operator, as well as the corresponding \emph{left/right} quantum codifferential for a \emph{quantum} Riemannian metric. Moreover, by considering left/right associated quantum vector bundles, we introduce the \emph{non--commutative geometrical} counterparts of the formal adjoint operators of the exterior covariant derivatives.

Building on Sections 2 and 3, Section 4 is the core of this paper, since in this section we develop the theory of Yang--Mills fields and space--time scalar matter fields, which is the purpose of this paper. We begin with \emph{pure} Yang--Mills fields and then we treat space--time scalar matter fields coupled to Yang--Mills fields. In Section~5, in order to keep the length of the paper reasonable, we present two representative classes of quantum principal bundles to which our theory applies: quantum principal bundles over \emph{classical spaces} (manifolds) and trivial quantum principal bundles over the Moyal--Weyl algebra. The last section is devoted to some concluding remarks. It is worth mentioning that concrete examples that illustrate the theory developed here can be found in \cite{sald3,sald4,sald5,sald6}.

Appendix A contains the explicit proof of a statement from Section 3.2, which is used in Section 4.1 and Appendix B collects several propositions concerning the formal adjointability of $A$--linear operators between finitely generated projective $A$--modules. Finally, in Appendix~C we show that the \emph{non--commutative geometrical field equations} of the theory developed in Section~4 coincide with their counterparts in differential geometry. Moreover, Appendix~C also serves to demonstrate the \emph{naturality} of our construction (which is given solely by taking the \emph{pull--back}), showing how the present framework extends and generalizes the \emph{classical} case.

It is worth mentioning that, by requiring \emph{quantum} pseudo--Riemannian metrics, all our results remain valid. References~\cite{sald5,sald6} provide concrete examples of this.

Other viewpoints on quantum bundles can be found in the literature; see, for example, \cite{bm,bu,pl}. All these formulations are intrinsically related through the theory of \emph{Hopf--Galois extensions} \cite{kt}. Moreover, there exist alternative approaches to Yang--Mills theory in non--commutative geometry, such as those developed in \cite{con,a,Dj,ch}, where the authors work directly with \emph{quantum vector bundles} and the framework of \emph{spectral triples}.

We have chosen to work with quantum principal bundles in this paper because we believe that a Yang--Mills scalar matter theory in non--commutative geometry should be formulated in terms of principal bundles and corepresentations, just as the \emph{dualization via the pull--back} of the the \emph{classical} case indicates. Furthermore, we adopt Durdevich’s formulation of quantum principal bundles due to its purely geometrical--algebraic nature, a feature that will become evident throughout this work. 

\section{Preliminaries}

In this section, we are going to show the highlights of the theory of $\ast$--Hopf algebras, quantum principal bundles and associated quantum vector bundles within Durdevich's formulation of quantum principal bundles \cite{sald2,micho2,woro1,woro2,micho1,micho3,stheve,micho7,sald1}. In other words, this section merely provides the necessary context so that, in a first reading, the reader does not need to consult other works in order to understand the theory developed in Sections 3 and 4. However, since Durdevich's formulation is not widely known, we do not recommend skipping this section.

 It is worth mentioning that, throughout this work, we use Sweedler’s notation and retain the notation for Durdevich's formulation introduced in~\cite{sald2}, since 
 \begin{enumerate}
    \item This paper is a direct continuation of reference \cite{sald2}.
    \item The notation used in~\cite{sald2} for Durdevich's formulation is closely related to that of the Brzezi\'nski--Majid formulation \cite{libro} (which is the framework more commonly used in the literature), while at the same time emphasizing the mathematical differences between the two formulations.
\end{enumerate}
In particular, all our quantum space will be represented by unital $\ast$--algebras; so in general, we will omit the word \emph{unital}.

\subsection{$\ast$--Hopf Algebras and their Quantum Differential Forms}

This subsection is brief summary of references \cite{micho2,woro1,woro2,micho1,stheve}. Consider $(H,\cdot,\mathbbm{1},\ast)$ a $\ast$--algebra.  We say that  it is a $\ast$--Hopf algebra if there exist $\ast$--algebra morphisms $$\Delta:H\longrightarrow H\otimes H,\qquad \epsilon:H \longrightarrow \C $$ called {\it the coproduct} and {\it the counit}, respectively, and there exists a linear map
    $$S:H\longrightarrow H$$ called {\it the antipode}
    such that 
\begin{equation*}
\begin{aligned}
    (\id_H\otimes \Delta)\circ  \Delta &=(\Delta \otimes \id_H)\circ \Delta,\\  (\epsilon \otimes \id_H)\circ \Delta = \id_H,&\qquad
     (\id_H\otimes \epsilon)\circ \Delta=\id_H,\\
     m\circ (S\otimes \id_H)\circ \Delta=\eta \circ \epsilon \;\;\;\;\; &\mbox{ and } \;\;\;\;\;m\circ (\id_H\otimes S )\circ \Delta=\eta \circ \epsilon,
\end{aligned}
\end{equation*}
where $\eta: \C\longrightarrow H$ is the linear map defined by $\eta(\lambda)=\lambda \mathbbm{1}$. A $\ast$--Hopf algebra will be represented by 
\begin{equation}
    \label{2.f0}
    H^\infty=(H,\cdot,\mathbbm{1},\Delta,\epsilon,S,\ast).
\end{equation}

\begin{Definition}
    \label{corepdef}
    Let $H^\infty$ be a $\ast$--Hopf algebra and let $V$ be a $\C$--vector space $V$. A linear map $$\delta^V : V \longrightarrow V \otimes H$$ is a (right) $H$--corepresentation on $V$ (or a coaction) if
    \begin{equation}
\label{2.f1}
(\id_V \otimes \epsilon)\circ \delta^V = \id_V, 
\qquad 
(\id_V \otimes \Delta)\circ \delta^V = (\delta^V \otimes \id_H)\circ \delta^V.
\end{equation}
\noindent
We say that $\delta^V$ is finite--dimensional if $\dim_{\C}(V)$ is finite.  
\end{Definition}

A $H$--corepresentation $\delta^V$ is called reducible if there exists a non--trivial subspace $L$ such that $\delta^V(L)\subseteq L\otimes H$.

On the other hand, according to \cite{woro1}, the Definition \ref{corepdef} is equivalent to the following: a $H$--corepresentation is an invertible element $\delta^V$ of $B(V)\otimes H$ such that $$(\mathbbm{1}\otimes \Delta)\delta^V=\delta^V_{12}\,\delta^V_{13},$$ where we have used leg--numbering notation, and $B(V)$ is the space of all linear endomorphism of $V$. In this way, a corepresentation $\delta^V$ is called unitary if there exists an inner product $\langle -|- \rangle$ on $V$ (not necessarily unique) such that $\delta^V$ is a unitary element of $B(V)\otimes H$. In reference \cite{woro1} it is proved that every finite--dimensional $H$--corepresentation on $V$ is unitary for some (not necessarily unique) inner product. In this way, in the rest of this paper, we will always assume that every finite--dimensional $H$--corepresentation is unitary.

\begin{Definition}
    \label{morrepdef}
    Given two $H$--corepresentations $\delta^V$ and $\delta^W$, a corepresentation morphism is a linear map
\begin{equation}
\label{2.f3}
T:V\longrightarrow W 
\qquad \text{such that} \qquad 
(T\otimes \id_H)\circ \delta^V=\delta^W \circ T .
\end{equation}
The set of all corepresentation morphisms between $\delta^V$ and $\delta^W$ will be denoted by
\begin{equation}
\label{2.f4}
\Mor(\delta^V,\delta^W),
\end{equation}
and the set of all finite--dimensional $H$--corepresentations will be denoted by
\begin{equation}
\label{2.f6}
\Obj(\Rep_{H}).
\end{equation}
\end{Definition}

In this paper, compact matrix quantum groups, or simply quantum groups~(\cite{woro1}), will be denoted by $\mathcal{G}$. In accordance with~\cite{woro1}, every quantum group $\mathcal{G}$ possesses a dense $\ast$--Hopf (sub)algebra $H^\infty$. This $\ast$--Hopf algebra will be interpreted as the algebra of \emph{$\C$--valued polynomial functions} on $\mathcal{G}$. 

\begin{Remark}
\label{groupsandalgebras}
From now on, we will work exclusively with $\ast$--Hopf algebras $H^\infty$ arising from quantum groups $\mathcal{G}$. 
\end{Remark}

The assumption of Remark \ref{groupsandalgebras} is taken due to the fact that the following theorem holds for $\ast$--Hopf algebras of this kind. The reader can find a proof of this theorem in reference \cite{woro1}.

\begin{Theorem}
\label{rep}
Let $\T$ be a complete set of mutually non--equivalent irreducible (necessarily finite--dimensional) 
$H$--corepresentations with $\delta^\C_{\triv}\in\T$ (the trivial corepresentation on $\C$).  
For any $\delta^V\in\T$ acting on $V$,
\begin{equation}
\label{2.f8}
\delta^V(e_j)=\sum^{n_{V}}_{i=1} e_i\otimes g^{V}_{ij},
\end{equation}
where $\{e_i\}_{i=1}^{n_V}$ is an orthonormal basis of $V$ and 
$\{g^{V}_{ij}\}_{i,j=1}^{n_V}\subseteq H$.  
Then, the collection $\{g^{V}_{ij}\}_{\delta^V,i,j}$ is a linear basis of $H$, where $\delta^V$ runs over $\T$ and 
$i,j$ run from $1$ to $n_V=\dim_{\C}(V)$.
\end{Theorem}

Now, consider a first--order differential $\ast$--calculus  over $H$ (abbreviated ``$\ast$--FODC'') $$(\Gamma,d),\qquad d:H\longrightarrow \Gamma.$$ This is a standard notion in non--commutative geometry; however, for readers unfamiliar with the concept, we recommend~\cite{stheve} as a basic reference.  We say that $(\Gamma,d)$ is right covariant if there exists a linear map
\begin{equation}
    \label{2.f9.1}
    {_\Gamma}\Phi:\Gamma \longrightarrow \Gamma\otimes H  
\end{equation}
such that
\begin{enumerate}
    \item  ${_\Gamma}\Phi$ preserves the $\ast$--structure and ${_\Gamma}\Phi(\vartheta\,g)={_\Gamma}\Phi(\vartheta)\Delta(g)$ for all $\vartheta$ $\in$ $\Gamma$ and all $g$ $\in$ $H$. Here, the $\ast$--structure of $\Gamma\otimes H$ is given by $(\vartheta\otimes g)^\ast:=\vartheta^\ast\otimes g^\ast$.
    \item ${_\Gamma}\Phi$ is a (right) $H$--corepresentation on $\Gamma$.
    \item ${_\Gamma}\Phi\circ d=(d\otimes \id_H)\circ \Delta$.
\end{enumerate}
In a similar manner, using a linear map 
\begin{equation}
    \label{2.f9.2}
    \Phi{_\Gamma}:\Gamma \longrightarrow H\otimes \Gamma,  
\end{equation}
we can define left covariant $\ast$--FODC's (\cite{stheve}). Finally, we say that a $\ast$--FODC $(\Gamma, d)$ is \emph{bicovariant} 
if it is both left covariant and right covariant. 

The reader can check a proof of the following proposition in Section 6 of reference \cite{stheve}.
\begin{Proposition}
    \label{0U}
    Let  $\mathcal{R}$ $\subseteq$ $\Ker(\epsilon)$ be a right $H$--ideal such that $S(\mathcal{R})^\ast\subseteq \mathcal{R}$ and $\Ad(\mathcal{R})\subseteq \mathcal{R}\otimes H$, where
\begin{equation}
    \label{2.f9.5}
    \Ad:H\longrightarrow H\otimes H,\qquad g\longmapsto g^{(2)}\otimes S(g^{(1)})\,g^{(3)}
\end{equation}
is the (right) adjoint coaction. Then 
    \begin{equation}
    \label{2.f9.4}
    (\Gamma:=H\otimes {\Ker(\epsilon)\over \mathcal{R}},d),
\end{equation}
defines a $\ast$--FODC over $H$ for some  differential $d$. Moreover, this $\ast$--FODC is bicovariant for some   linear maps ${_{\Gamma}}\Phi$, $\Phi_{{\Gamma}}$.

Reciprocally, every bicovariant $\ast$--FODC $(\Gamma,d)$ on $H$ is isomorphic to the one in equation (\ref{2.f9.4}) for some   right $H$--ideal $\mathcal{R}$ $\subseteq$ $\Ker(\epsilon)$ such that $S(\mathcal{R})^\ast\subseteq \mathcal{R}$ and $\Ad(\mathcal{R})=\mathcal{R}\otimes H$.
\end{Proposition}

Let $(\Gamma,d)$ be a bicovariant $\ast$--FODC over $H$ and consider the $\C$--vector space given by
\begin{equation}
\label{2.f13}
\mathfrak{qg}^\#:=\{\theta \in \Gamma \mid \Phi_\Gamma(\theta)=\mathbbm{1}\otimes \theta  \}   =\mathbbm{1}\otimes {\Ker(\epsilon)\over \mathcal{R}}\cong {\Ker(\epsilon)\over \mathcal{R}}.
\end{equation}

\begin{Definition}
    \label{addef}
    We define the adjoint $H$--corepresentation on $\mathfrak{qg}^\#$ as the linear map
    \begin{equation}
\label{2.f15}
\ad: \mathfrak{qg}^\#\longrightarrow \mathfrak{qg}^\#\otimes H
\end{equation} given by
    \begin{equation}
\label{2.f15.1}
\ad\circ \pi= (\pi\otimes \id_H)\circ \Ad.
\end{equation}
\end{Definition}
In the previous definition, the linear map
\begin{equation}
\label{2.f14}
\pi:H \longrightarrow \mathfrak{qg}^\#,\qquad
g \longmapsto S(g^{(1)})dg^{(2)}
\end{equation}
is the quantum germs map (see Section 6.4 of reference \cite{stheve}). It is worth mentioning that $\pi$ has several useful properties, for example, $\pi|_{\Ker(\epsilon)}$ is surjective, and $$ \Ker(\pi)=\C\mathbbm{1}\oplus \mathcal{R}, \qquad \pi(g)^\ast=-\pi(S(g)^\ast)\;\;\;\mbox{ and }\;\;\; dg=g^{(1)}\pi(g^{(2)})\;\; \mbox{ for all }\;\ g\in H.$$ There is a well--defined right $H$--module structure on $\mathfrak{qg}^\#$ given by 
\begin{equation}
\label{2.f16}
\theta  \diamondsuit  g:=\pi(hg-\epsilon(h)g)=S(g^{(1)})\theta g^{(2)}
\end{equation}
for every $\theta=\pi(h)$ $\in$ $\mathfrak{qg}^\#$, $g$ $\in$ $H$ (\cite{stheve}).

Let $H^\infty$ be a $\ast$--Hopf algebra and consider  $(\Gamma,d)$ a $\ast$--FODC over $H$. Now, we can consider its \emph{universal differential envelope $\ast$--calculus} (\cite{micho1,stheve})
  \begin{equation}
    \label{2.f10.1}
     (\Gamma^\wedge,d,\ast).
\end{equation}
This space is a graded differential $\ast$--algebra generated by its space of $0$--degree elements $\Gamma^{\wedge\,0}=H$. In addition, its space of $1$--degree elements is $\Gamma^{\wedge\,1}=\Gamma$. The reader can check the explicit  construction of $(\Gamma^\wedge,d,\ast)$ in references \cite{micho1,stheve}; here we will focus only in present the more relevant aspects of this space for our purposes. 

For example, the space $(\Gamma^\wedge,d,\ast)$ is actually the maximal prolongation of $(\Gamma,d)$. That is, $(\Gamma^\wedge,d,\ast)$ is the biggest graded differential $\ast$--algebra generated by $H$ (the space of $0$--degree elements) with $\Gamma$ as the space of $1$--degree elements. Moreover, the following proposition holds; a proof of this proposition may be found in~\cite{micho1,stheve}.

\begin{Proposition}
\label{1U}
Suppose $(\Omega,d_{\Omega},\ast)$ is a graded differential $\ast$--algebra generated by $\Omega^0=H$ and $(\Gamma,d)$ is a $\ast$--FODC over $H$. Let $\varphi^{0}:\Gamma^{\wedge\,0}=H\longrightarrow  \Omega^0=H$ be a $\ast$--algebra morphism and $\varphi^{1}:\Gamma\longrightarrow \Omega $ be a linear map such that $\varphi^{1}(g\,dh)=\varphi^{0}(g)\,d_{\Omega}(\varphi^{0}(h))$ for all $g$, $h$ $\in$ $H$. Then, there exists a unique family of linear maps $\varphi^{k}:\Gamma^{\wedge k}\longrightarrow \Omega$ such that $$\varphi:=\displaystyle\bigoplus_k \varphi^{k}: \Gamma^{\wedge}\longrightarrow \Omega $$ is a graded differential $\ast$--algebra morphism.
\end{Proposition}

We are particularly interested in the universal differential envelope $\ast$--calculus $(\Gamma^\wedge,d,\ast)\cong (H\otimes \mathfrak{qg}^{\#\wedge},d,\ast)$ of a bicovariant $\ast$--FODC $(\Gamma,d)$. In this situation, the structure of $\ast$--Hopf of $H$ $$\Delta:H\longrightarrow H\otimes H,\qquad \epsilon:H\longrightarrow \C, \qquad S:H\longrightarrow H $$ can be extended to
     \begin{equation}
\label{2.f3.5}
\Delta: \Gamma^\wedge \longrightarrow \Gamma^\wedge \otimes  \Gamma^\wedge \qquad \epsilon:\Gamma^\wedge\longrightarrow \C, \qquad S:\Gamma^\wedge\longrightarrow \Gamma^\wedge
\end{equation}
(the tensor product of $\Gamma^\wedge \otimes  \Gamma^\wedge$ is the tensor product of graded differential $\ast$--algebras) in such a way that 
\begin{equation}
\label{2.f3.8}
\Gamma^{\wedge\,\infty}:=(\Gamma^\wedge,\cdot,\mathbbm{1},\Delta,\epsilon, S,d,\ast)
\end{equation}
is a graded differential $\ast$--Hopf algebra.  Full details can be found in Appendix $B$ of reference \cite{micho1}. For example, the following formula holds for every $\theta$ $\in$ $\mathfrak{qg}^\#$ (\cite{micho1,stheve}) 
\begin{equation}
    \label{loquesea3}
    \Delta(\theta)=\mathbbm{1}\otimes \theta+\ad(\theta).
\end{equation}

It is worth mentioning that the right $H$--module structure of $\mathfrak{qg}^\#$ (see equation (\ref{2.f16})) can be extended to $\mathfrak{qg}^{\#\wedge}$ by means of 
\begin{equation}
    \label{2.f12.4}
    \mathbbm{1}\diamondsuit g=\epsilon(g),\quad (\theta_1\theta_2)\diamondsuit g=(\theta_1\circ g^{(1)})(\theta_2\diamondsuit g^{(2)}).
\end{equation}
for all $g$ $\in$ $H$, $\theta_1$, $\theta_2$ $\in$ $\mathfrak{qg}^\#$. In addition, since  $\Gamma^{\wedge}$ has structure of $\ast$--Hopf algebra, it has a (right) adjoint corepresentation
\begin{equation}
\label{2.f11}
\Ad:\Gamma^\wedge \longrightarrow \Gamma^\wedge \otimes \Gamma^\wedge
\end{equation}
such that $$\Ad(t)=(-1)^{\partial t^{(1)}\partial t^{(2)}} t^{(2)}\otimes S(t^{(1)})t^{(3)},$$ where $\partial x$ denotes the grade of $x$ and $(\id_{\Gamma^\wedge}\otimes \Delta)\Delta(t)=(\Delta\otimes \id_{\Gamma^\wedge})\Delta(t)=t^{(1)}\otimes t^{(2)}\otimes t^{(3)}.$ Clearly, $\Ad$ extends the adjoint coaction $\Ad:H\longrightarrow H\otimes H$.

Furthermore, for a bicovariant $\ast$--FODC $(\Gamma,d)$ over $H$, the following isomorphism holds (\cite{micho1,tomas})
\begin{equation}
    \label{2.f12.3}
    \Gamma^{\wedge}\cong H\otimes \mathfrak{qg}^{\#\wedge},
\end{equation} 
where
\begin{equation}
    \label{2.f12.1}
    \begin{aligned}
          \mathfrak{qg}^{\#\wedge}=\otimes^\bullet \mathfrak{qg}^\#/A^\wedge,&\qquad\otimes^\bullet\mathfrak{qg}^\#:=\bigoplus_k (\otimes^k\mathfrak{qg}^\#),\qquad
\otimes^k\mathfrak{qg}^\#:=\underbrace{\mathfrak{qg}^\#\otimes\cdots\otimes \mathfrak{qg}^\#}_{k\; times},\\
 &\otimes^1\mathfrak{qg}^\#=\mathfrak{qg}^\#, \qquad \otimes^0\mathfrak{qg}^\#=\C\mathbbm{1}
    \end{aligned}
\end{equation}
 with $A^\wedge$ the two--side ideal of $\otimes^\bullet\mathfrak{qg}^\#$ generated by elements 
\begin{equation}
    \label{2.f12.11}
    \pi(g^{(1)})\otimes \pi(g^{(2)})\qquad \mbox{ for all }\qquad g \,\in\, \mathcal{R}.
\end{equation}
According to \cite{micho1}, the Maurer--Cartan equation holds
 \begin{equation}
    \label{2.f12.2}
    d\pi(g)=-\pi(g^{(1)})\pi(g^{(2)})
\end{equation}
(the product in $\mathfrak{qg}^{\#\wedge}$ is denoted by juxtaposition of elements) for all $g$ $\in$ $H$.

For a given quantum group $\G$ and a bicovariant $\ast$--FODC $(\Gamma,d)$ on $H$, the triple $$(\Gamma^\wedge,d,\ast)$$ will be interpreted as the algebra of {\it quantum differential forms} of $H$. Additionally,  the space 
$$\mathfrak{qg}^\#={\Ker(\epsilon)\over \mathcal{R}}$$ will play the role of the {\it quantum dual Lie algebra} and the  $H$--corepresentation $\ad$ of Definition \ref{addef} will play the role of the {\it dualization} of the right adjoint action of a Lie group on its Lie algebra. Reference~\cite{appendix} shows that these concepts constitute genuine \emph{non--commutative geometrical} generalizations of the corresponding \emph{classical} constructions associated with compact matrix Lie groups.

\subsection{Quantum Principal Bundles}

As mentioned before, this work is developed within M. Durdevich’s formulation of quantum principal bundles. Of course, there are mathematical reasons to adopt this approach (see Section 6) rather than Brzezi\'nski–Majid formulation \cite{libro}. Since Durdevich’s framework is not widely known, we will take some time to explain in some detail the basic constructions and motivations of the parts of  Durdevich’s theory used for this paper; while commenting on some differences and similarities with the Brzezi\'nski -Majid framework. The reader is encouraged to consult the original work \cite{micho2,micho1,micho3} for further details.

According to Section 12 of reference \cite{stheve}, the following definition comes from \emph{dualizing via the pull--back} the definition of a principal $G$--bundle in differential geometry.
 
\begin{Definition}
    \label{qpbdef}
    Let $(P,\cdot,\mathbbm{1},\ast)$ be a quantum space and let $\G$ be a quantum group with its associated $\ast$--Hopf algebra $H^\infty=(H,\cdot,\mathbbm{1},\ast,\Delta,\epsilon,S)$. A {\it quantum principal $\G$--bundle}  (abbreviated ``qpb") is a triple 
\begin{equation*}
\zeta=(P,H,\Delta_P),
\end{equation*}
where $(P,\cdot,\mathbbm{1},\ast)$ is called the {\it quantum total space}, and  $$\Delta_P:P \longrightarrow P\otimes H$$ is a $\ast$--algebra morphism that satisfies
\begin{enumerate}
\item $\Delta_P$ is a $H$--corepresentation.
\item The linear map $\beta:P\otimes P\longrightarrow P\otimes H$ given by $$\beta(x\otimes y):=x\cdot \Delta_P(y):=(x\otimes \mathbbm{1})\cdot \Delta_P(y) $$ is surjective. 
\end{enumerate}
In this situation, 
\begin{equation}
    \label{basespacedef}
    B=\{x\in P \mid \Delta_P(x)=x\otimes \mathbbm{1} \}
\end{equation}
is a $\ast$--(sub)algebra which receives the name of quantum base space. It is common to say that $\zeta$ is a qpb over $B$.
\end{Definition}

In general, Durdevich’s formulation does not require working with $\ast$--Hopf algebras arising from quantum groups (\cite{stheve}), but the previous definition is appropriated for the purpose of this paper. Furthermore, it is worth mentioning that under Definition \ref{qpbdef}, every qpb is a Hopf--Galois extension (\cite{libro}).

In this paper, the quantum base space $(B,\cdot,\mathbbm{1},\ast)$ plays a fundamental role, as it represents the \emph{non--commutative geometrical} analog of space--time in theoretical physics. Thus, in order to emphasize its importance, we will denote quantum principal bundles as   
\begin{equation}
\label{2.f17}
\zeta=(P,B,\Delta_P).
\end{equation}

In differential geometry, given a\footnote{$GM$ is the total space, $M$ is the base space, $G$ is a Lie group and $\pi$ is the bundle projection. For this paper, we will consider that $GM$ and $M$ are compact manifold, and $G$ is a compact matrix Lie group.} principal $G$--bundle $\pi:GM\longrightarrow M$, it is well--known that\footnote{$TGM$ is the tangent bundle of the total space, $TM$ is the tangent bundle of the base space, $TG$ is the tangent bundle of $G$ and $d\pi$ is the differential of $\pi$.} $d\pi:TGM\longrightarrow TM$ is principal $TG$--bundle (\cite{nodg}). This fact motivates the following definition in Durdevich's formulation of qpb's. 

\begin{Definition}
    \label{difcaldef}
    Let $\zeta=(P,B,\Delta_P)$ be a qpb  over $B$. A differential calculus on $\zeta$ is a triple $$ (\Omega^\bullet(P),\Gamma^\wedge,\Delta_{\Omega^\bullet(P)}),$$ where 
 \begin{enumerate}
 \item The space $(\Omega^\bullet(P),d,\ast)$ is a graded differential $\ast$--algebra generated by $\Omega^0(P)=P$ ({\it quantum differential forms of $P$}).
 \item The space $(\Gamma^\wedge,d,\ast)$ is the universal differential envelope $\ast$--calculus for some fixed bicovariant $\ast$--FODC $(\Gamma,d)$ over $H$.
 \item The map  $$\Delta_{\Omega^\bullet(P)}:\Omega^\bullet(P)\longrightarrow \Omega^\bullet(P)\otimes \Gamma^{\wedge}$$ is a graded differential $\ast$--algebra morphism such that $\Delta_{\Omega^\bullet(P)}|_P=\Delta_P$.  Here, we have considered that $\otimes$ is the tensor product of graded differential $\ast$--algebras.
 \end{enumerate}
 In this situation, 
 \begin{equation}
\label{2.f20}
\Omega^\bullet(B):=\{\mu \in \Omega^\bullet(P)\mid \Delta_{\Omega^\bullet(P)}(\mu)=\mu\otimes \mathbbm{1}\}
\end{equation}
is a graded differential $\ast$--(sub)algebra which receives the name of space of base forms.
\end{Definition}

Notice that the purpose of the previous definition is to ensure that $$(\Omega^\bullet(P),\Omega^\bullet(B),\Delta_{\Omega^\bullet(P)})$$ is a qpb, as the reader can verify in~\cite{micho5}, where the author follows this approach and also works with the quantum translation map in higher degrees. Moreover, it is worth mentioning that in general, the space of base forms is not generated by $\Omega^0(B)=B$. An explicit example of this fact can be found in reference \cite{appendix}. Notice that in Brzezi\'nski--Majid formulation of qpb’s, the space of base forms is defined only in degree $1$ elements and is required to be generated by $B$.

\begin{Definition}
    \label{horverdef}
    Let $\zeta$ be a qpb over $B$ with a differential calculus. We define the space of horizontal forms as \begin{equation}
\label{2.f18}
\Hor^\bullet P\,:=\{\varphi \in \Omega^\bullet(P)\mid \Delta_{\Omega^\bullet(P)}(\varphi)\, \in \, \Omega^\bullet(P)\otimes H \},
\end{equation}
and it is a graded  $\ast$--subalgebra of $\Omega^\bullet(P)$  (\cite{stheve}). 
\end{Definition}

Since $\Delta_{\Omega^\bullet(P)}(\Hor^\bullet P)\subseteq \Hor^\bullet P\otimes H,$ the map 
\begin{equation}
\label{2.f19}
\Delta_\Hor:=\Delta_{\Omega^\bullet(P)}|_{\Hor^\bullet P}: \Hor^\bullet P \longrightarrow \Hor^\bullet P\otimes H
\end{equation}
is a $H$--corepresentation on $\Hor^\bullet P$ \cite{stheve}. 

In Brzezi\'nski--Majid formulation of qpb's, the horizontal space is defined only in degree $1$ elements by the $P$--bimodule $P\,\Omega^1(B)\,P$ (\cite{libro}) and clearly, this space is a subspace of $\Hor^1 P$. However, in principle, this inclusion can be proper \cite{tomas}.

\begin{Definition}
    \label{verdef}
    The space of vertical forms is introduced as the graded vector space  
\begin{equation}
    \label{2.fver1}
    \Vert^\bullet P:= P\otimes \mathfrak{qg}^{\#\wedge} 
\end{equation}
with the operations (see equation (\ref{2.f12.1}))
\begin{equation}
    \label{2.fver2}
    \begin{aligned}
        (x\otimes\theta)(y\otimes \vartheta):=xy^{(0)}\otimes (\theta\diamondsuit y^{(1)})\vartheta,\\
    (x\otimes\theta)^{\ast}:=x^{(0)\ast}\otimes(\theta^{\ast}\diamondsuit x^{(1)\ast}),\\
        d_v(x\otimes \theta):=x\otimes d\theta+x^{(0)}\otimes\pi(x^{(1)})\theta,
    \end{aligned}
\end{equation}
where $x$, $y$ $\in$ $P$, $\theta$, $\vartheta$ $\in$ $\mathfrak{qg}^{\#\wedge}$, $\Delta_P(x)=x^{(0)}\otimes x^{(1)}$, $\Delta_P(y)=y^{(0)}\otimes y^{(1)}$ and $\theta \diamondsuit g$  is defined in equations (\ref{2.f16}), (\ref{2.f12.4}). With these operations, $\Vert^\bullet P$ is a graded differential $\ast$--algebra.
\end{Definition}

Notice that in Brzezi\'nski--Majid formulation of qpb's, the vertical space is defined only in degree $1$ elements by the $P$--bimodule $P\otimes \mathfrak{qg}^\#$ (\cite{libro}); without the other operations.

According to Lemma 3.1 of reference \cite{micho2}, $\Vert^\bullet P$ is generated by its degree $0$ elements $\Vert^0 P=P$. Moreover, in accordance with Lemma 3.2 of reference \cite{micho2}, the map 
$$\Delta_\Vert: \Vert^\bullet P: \longrightarrow \Vert^\bullet P\otimes \Gamma^\wedge \quad \mbox{ such that }\quad  \Delta_\Vert(x \otimes \theta)=x^{(0)}\otimes \theta^{(1)} \otimes x^{(0)}\theta^{(2)}$$ (with $\Delta(\theta)= \theta^{(1)} \otimes \theta^{(2)})$ is the unique graded differential $\ast$--algebra morphism that is also a $\Gamma^\wedge$--corepresentation and is $\Delta_P$ in the degree $0$ case.

The reader can find proofs of the two following propositions in Proposition 3.6 and Lemma 3.7 of reference \cite{micho2}. 

\begin{Proposition}
\label{piv}
The map $$\pi_\V:\Omega^\bullet(P)\longrightarrow \Vert^\bullet P $$ given by $$\pi_\V=(\id_{P}\otimes (\epsilon \otimes \id_{\mathfrak{qg}^{\#\wedge}}))\circ (\id_{\Omega^\bullet(P)}\otimes \rho_k) \circ \Delta_{\Omega^\bullet(P)}$$  is the unique graded differential $\ast$--algebra morphism  such that  $\pi_\V=\id_{P}$ in degree $0$, and $$\Delta_\Vert\circ \pi_\V  =(\pi_\V\otimes\id_{\Gamma^\wedge})\circ \Delta_{\Omega^\bullet(P)}.$$ Moreover, $\pi_\V$ is surjective. Here, $\rho_k:\Gamma^\wedge\longrightarrow \Gamma^{\wedge\, k}$ is the canonical projection onto the degree $k$ elements and we have considered that $\Gamma^\wedge=H \otimes \mathfrak{qg}^{\#\wedge}$ (see equation (\ref{2.f12.3})).
\end{Proposition}

\begin{Proposition}
    \label{seq}
    The following sequence of $\ast$--$P$--bimodules
\begin{equation}
\label{3.f1.4}
0\longrightarrow  \Hor^1 P \lhook\joinrel\relbar\joinrel\rightarrow  \Omega^1(P) \xlongrightarrow{\pi_\V} \Vert^1 P \longrightarrow 0
\end{equation}
is always exact. This sequence is called the Atiyah sequence.
\end{Proposition}

It is worth mentioning that in Brzezi\'nski--Majid formulation of qpb's, the exactness of the Atiyah sequence is a condition of the theory (\cite{libro}), while in Durdevich's formulation it is a result of the theory.

By {\it dualizing via the pull--back} the notion of principal connections in differential geometry (\cite{stheve}), we have the following definition in Durdevich's formulation of qpb's.
\begin{Definition}
    \label{qpc's}
    Let $\zeta$ be a qpb with a differential calculus. A quantum principal connection (abbreviated ``qpc") on $\zeta$ is a linear 
    $$\omega:\mathfrak{qg}^\#\longrightarrow \Omega^{1}(P)$$ such that for all $\theta$ $\in$ $\mathfrak{qg}^\#$
    \begin{equation}
    \label{qpc1}
    \Delta_{\Omega^\bullet(P)}(\omega(\theta))=(\omega\otimes \id_H)\ad(\theta)+\mathbbm{1}\otimes\theta \qquad \mbox{ and }\qquad
   \omega(\theta^\ast)=\omega(\theta)^\ast,
\end{equation}
where $\ad$ is the $H$--corepresentation given in equation (\ref{2.f15}). Equivalently (\cite{micho2}), a quantum principal connection is a linear map $$\omega:\mathfrak{qg}^\#\longrightarrow \Omega^1(P)$$ such that for all $\theta$ $\in$ $\mathfrak{qg}^\#$
\begin{equation}
    \label{qpc3}
    \omega(\theta)\,\in\, \Mor(\ad,\Delta_{\Omega^\bullet(P)}),\quad  (\pi_\V\circ \omega)(\theta)=\mathbbm{1}\otimes \theta 
\quad \mbox{ and }\quad
         \omega(\theta^\ast)=\omega(\theta)^\ast,
\end{equation}
where
\begin{eqnarray*}
    \Mor(\ad,\Delta_{\Omega^\bullet(P)})=\{\psi:\mathfrak{qg}^\# \longrightarrow \Omega^\bullet(P) &\mid& \psi \mbox{ is linear and } \\ & &(\psi\otimes \id_{H})\circ \ad=(\id_{\Omega^\bullet(P)}\otimes \rho_0)\circ \Delta_{\Omega^\bullet(P)} \circ \psi\},
\end{eqnarray*}
with $\rho_0:\Gamma^\wedge\longrightarrow H$ the canonical projection map. 
\end{Definition}

In analogy with the {\it classical} case, it can be proved that the set 
\begin{equation}
\label{2.f24.1}
\mathfrak{qpc}(\zeta):=\{\omega:\mathfrak{qg}^\#\longrightarrow \Omega^1(P)\mid \omega \mbox{ is a qpc on }\zeta \}
\end{equation}
is not empty, and it is an affine space modeled by the $\R$--vector space of {\it connection displacements} (\cite{micho2,stheve}) 
\begin{equation}
 \label{2.f24.2}
\begin{aligned}
    \overrightarrow{\mathfrak{qpc}(\zeta)}:=\{\lambda:\mathfrak{qg}^\# \longrightarrow \Hor^1 P \mid & \,\lambda \mbox{ is linear and }\\
    &(\lambda\otimes \id_H)\circ \ad=\Delta_\Hor \circ \lambda,\;\; \ast\circ\lambda\circ \ast= \lambda \}.
\end{aligned}
\end{equation}

A qpc is called {\it regular} if for all $\varphi$ $\in$ $\Hor^{k}P$ and $\theta$ $\in$ $\mathfrak{qg}^\#$, we have 
\begin{equation}
\label{2.f25}
\omega(\theta)\,\varphi=(-1)^{k}\varphi^{(0)}\omega(\theta\diamondsuit \varphi^{(1)}), 
\end{equation}
where $\Delta_\Hor(\varphi)=\varphi^{(0)}\otimes\varphi^{(1)}.$ A qpc $\omega$ is called {\it multiplicative} if 
\begin{equation}
\label{2.f26}
\omega(\pi(g^{(1)}))\omega(\pi(g^{(2)}))=0
\end{equation}
for all $g$ $\in$ $\mathcal{R}$. 

In light of Theorem 12.8 of reference \cite{stheve}, qpc's are in bijection with left $P$--module splittings of the Atiyah sequence. We recommend to the interested reader to check this reference for more details.

Let $\pi:GM\longrightarrow M$ be a {\it classical} principal $G$--bundle with a principal connection $\omega_\mathrm{class}$. Consider $\Delta_\Hor$  the pull--back of the canonical $G$--action on the complexification of the horizontal bundle of $GM$. Then, the covariant derivative of $\omega_\mathrm{class}$ acting on horizontal $\C$--valued differential $k$--forms of $GM$ can be written in the form
\begin{equation}
    \label{dif0}
    \begin{aligned}
        D^{\omega_\mathrm{class}}(\eta)=d\eta-(-1)^k\,\eta^{(0)}\wedge \omega^\#_{\mathrm{class}}(\pi_{\mathrm{class}}(\eta^{(1)})),
    \end{aligned}
\end{equation}
 where $\wedge$ is the usual wedge product of $\C$--valued differential forms, $\Delta_\Hor(\eta)=\eta^{(0)}\otimes \eta^{(1)}$, $$\omega^\#_{\mathrm{class}}:\mathfrak{g}^\#_\C\longrightarrow \Omega^\bullet_\C(GM)$$  is the complex--extension of the pull--back of $\omega_{\mathrm{class}}$, and $$\pi_{\mathrm{class}}:C^\infty_{\mathrm{pol}}(G)\longrightarrow \mathfrak{g}^\#_\C,\qquad g\longmapsto S(g^{(1)})dg^{(2)}=(dg)_e $$ is the the corresponding \emph{classical} germs map, with $(dg)_e$ the differential at $e$ (the neutral element of $G$)  of an element of $C^\infty_{\mathrm{pol}}(G)$ (the space of polynomical $\C$--valued functions of $G$) \cite{appendix}. Here, $$\mathfrak{g}^\#_\C={\Ker(\epsilon)\over \Ker^2(\epsilon)} \qquad \mbox{ with }\qquad \Ker^2(\epsilon)=\left\{\sum_i g_i\,h_i\mid g_i,\,h_i\in \Ker(\epsilon) \right\}$$ is the dual space of the complexification $\mathfrak{g}_\C$ of the Lie algebra $\mathfrak{g}$ of $G$.  

Equation~(\ref{dif0}) motivates the following definition in both Durdevich’s and the Brzezi\'nski--Majid formulations of qpb’s (although, in the Brzezi\'nski--Majid formulation, the definition is given only for $0$--degree elements). It is worth mentioning that $\omega^\#_{\mathrm{class}}$ is a regular and multiplicative qpc \cite{micho1}.

\begin{Definition}
    \label{covdifdef}
   Let $\zeta=(P,B,\Delta_P)$ be a qpb with a differential calculus. For a given qpc $\omega$, we define its {\it covariant derivative} as the first--order linear map (\cite{micho3})
\begin{equation}
\label{2.f30}
D^{\omega}: \Hor^\bullet P \longrightarrow \Hor^\bullet P
\end{equation}
such that for every $\varphi$ $\in$ $\Hor^k P$ $$
D^{\omega}(\varphi)=  d\varphi-(-1)^{k}\varphi^{(0)}\omega(\pi(\varphi^{(1)})).$$ 
On the other hand, the first--order linear map 
\begin{equation}
\label{2.f30.1}
\widehat{D}^{\omega}:=\ast\circ D^{\omega}\circ\ast
\end{equation}
is called the {\it dual covariant derivative} of $\omega$. Explicitly, for every $\varphi$ $\in$ $\Hor^k P$, we have (\cite{micho3}) 
\begin{equation}
    \label{2.f30.5.1}
    \widehat{D}^{\omega}(\varphi)=d\varphi +\omega (\pi(S^{-1}(\varphi^{(1)})))\varphi^{(0)}.
\end{equation}
\end{Definition}

According to  \cite{micho3}, the following equation holds
\begin{equation}
    \label{covd.1}
    \widehat{D}^{\omega}(\varphi)=D^{\omega}(\varphi)+\ell^{\omega}(\pi(S^{-1}(\varphi^{(1)})),\varphi^{(0)}),
\end{equation}
where 
\begin{equation}
\label{something}
\begin{aligned}
\ell^{\omega}:\mathfrak{qg}^\#\times \Hor^\bullet P \longrightarrow \Hor^\bullet P, \qquad
(\theta,\varphi)\longmapsto \omega(\theta)\varphi-(-1)^k \varphi^{(0)}\omega(\theta\diamondsuit \varphi^{(1)}).
\end{aligned}
\end{equation}
The map $\ell^{\omega}$ {\it measures the degree of non--regularity} of $\omega$, in the sense of $\ell^{\omega}=0$ if and only if $\omega$ is regular. In this way, for regular qpc's we obtain $D^\omega=\widehat{D}^{\omega}$, which is the situation for qpc's of the form $\omega^\#_\mathrm{class}$ (\cite{micho1}). In other words, $D^\omega$ and $\widehat{D}^\omega$ are two different horizontal operators that generalize the covariant derivative of a principal connection in differential geometry. In the next section, we will work with both operators. 

Direct calculations prove that (\cite{micho2,micho3})
\begin{equation}
\label{2.f31}
D^{\omega}, \;\; \widehat{D}^{\omega}\, \in\,\Mor(\Delta_\Hor,\Delta_\Hor),
\end{equation}
among other properties \cite{micho3}.

Let $\pi:GM\longrightarrow M$ be a {\it classical} principal $G$--bundle with a principal connection $\omega_\mathrm{class}$. Then, the  complex--extension $\Omega^{\omega_\mathrm{class}}$ of  the curvature of $\omega_\mathrm{class}$ is defined as the $\mathfrak{g}_\C$--valued differential $2$--form of $GM$ given by (\cite{nodg})
\begin{equation}
    \label{classcur}
    \Omega^{\omega_\mathrm{class}}=d\omega_\mathrm{class} + {1\over 2}[\omega_\mathrm{class}\wedge \omega_\mathrm{class}].
\end{equation}
Since $\Omega^{\omega_\mathrm{class}}$ is $\mathfrak{g}_\C$--valued, it is natural to expect that, in the \emph{non--commutative geometrical} setting, the curvature of a qpc should be defined in $\mathfrak{qg}^\#$, as suggested by the \emph{dualization via the pull--back} of the \emph{classical} case. This is why in Durdevich's formulation of qpb's we do not define the curvature as in Brzezi\'nski--Majid formulation (\cite{libro}):
\begin{equation}
    \label{majidcurv}
    r^\omega: \Ker(\epsilon)\longrightarrow \Omega^2(P), \qquad g\longmapsto d\omega(\pi(g))+\omega(\pi(g^{(1)}))\,\omega(\pi(g^{(2)})).
\end{equation}

The pull--back of  $\Omega^{\omega_\mathrm{class}}$ is given by (\cite{micho1})
\begin{equation}
    \label{classcur2}
    d\omega^\#_\mathrm{class}+{1\over 2} \,m_{GM}(\omega^\#_\mathrm{class}\otimes \omega^\#_\mathrm{class})\circ c^T, 
\end{equation}
where  $m_{GM}$ is the product map of $\C$--valued differential forms of $GM$, and 
\begin{equation}
    \label{trans2}
    \mathrm{c}^T:= (\id_{\mathfrak{g}^\#_\C}\otimes \pi_\mathrm{class})\circ \ad^\#_\mathrm{class} : \mathfrak{g}^\#_\C\longrightarrow \mathfrak{g}^\#_\C\otimes \mathfrak{g}^\#_\C 
\end{equation}
is the complex--extension of the pull--back of the Lie commutator of $\mathfrak{g}$  (\cite{appendix}). Here, $\ad^\#_\mathrm{class}$ is the pull--back of the complex--extension of the right adjoint action of $G$ on $\mathfrak{g}$, and $\pi_{\mathrm{class}}$ is the classical germs map.

According to \cite{micho1}, we have that 
\begin{enumerate}
    \item $\displaystyle -{1\over 2}c^T \in \Mor(\ad^\#_\mathrm{class},\ad^\#_\mathrm{class}\otimes \ad^\#_\mathrm{class})$, where we consider that $\ad^\#$ is a  $C^\infty_{\mathrm{pol}}(G)$--corepresentation and the symbol $\otimes$ denotes the tensor product of corepresentations (\cite{woro1}).
    \item The Maurer--Cartan formula holds: for every $\theta$ $\in$ $\mathfrak{g}^\#_\C$, we get 
    \begin{equation*}
        d\theta=m_G(-{1\over 2}c^T(\theta)),
    \end{equation*}
    where $m_G$ is the product map of $\C$--valued differential forms of $G$.
\end{enumerate}
Notice that the Maurer--Cartan formula establishes a \emph{compatibility} between the Lie algebra structure of $\mathfrak{g}^\#_\C$ and its differential structure.

The linear map of equation (\ref{trans2}) can be easily generalized to the \emph{non--commutative geometrical} setting. In fact, for any bicovariant $\ast$--FODC $(\Gamma,d)$ over $H$ with quantum dual Lie algebra $$\mathfrak{qg}^\#={\Ker(\epsilon)\over\mathcal{R}},$$ we define the \emph{transpose commutator} as
\begin{equation}
    \label{trans1}
    \mathrm{c}^T:= (\id_{\mathfrak{qg}^\#}\otimes \pi)\circ \ad : \mathfrak{qg}^\#\longrightarrow \mathfrak{qg}^\#\otimes \mathfrak{qg}^\#, 
\end{equation}
where $$\pi: H\longrightarrow \mathfrak{qg}^\#$$ is the corresponding quantum germs map (see equation (\ref{2.f14})), and $\ad$ is the $H$--corepresentation of Definition \ref{addef}. In this way, given a qpb $\zeta=(P,B,\Delta_P   )$ with a differential calculus and following equation (\ref{classcur2}), one could take 
\begin{equation}
    \label{falsecur}
    d\omega+{1\over 2}\,m_\Omega(\omega\otimes \omega)\circ c^T
\end{equation}
as the definition of the curvature for a qpc $\omega$, with $m_\Omega: \Omega^\bullet(P)\otimes \Omega^\bullet(P)\longrightarrow \Omega^\bullet(P)$ the product map of $\Omega^\bullet(P)$.

However, in the \emph{non--commutative geometrical} setting, we have that (\cite{stheve}) $$\displaystyle -{1\over 2}c^T\; \in\; \Mor(\ad,\ad\otimes \ad),$$ but in general 
\begin{equation}
    \label{something1}
    d\theta \not=m_{\Gamma^\wedge}(\displaystyle -{1\over 2}c^T(\theta))
\end{equation}
for all $\theta$ $\in$ $\mathfrak{qg}^\#$, where $m_{\Gamma^\wedge}:\Gamma^\wedge\otimes \Gamma^\wedge\longrightarrow \Gamma^\wedge$ is the product map of $\Gamma^\wedge$. Section 2.3 of reference \cite{saldcon} presents a concrete example of the last statement since for that bicovariant $\ast$--FODC, we have $c^T=0$ and $d\not=0$.

Equation (\ref{something1}) implies that in general, the map $\displaystyle -{1\over 2}c^T$ does not provide  \emph{compatibility} between the quantum Lie algebra structure of $\mathfrak{qg}^\#$ and its differential structure and hence, equation (\ref{falsecur}) cannot be a proper general definition of the curvature of a qpc. This motivates the introduction of the following auxiliary map.

\begin{Definition}
    \label{embdifdef}
    An embedded differential is a linear map 
\begin{equation}
    \label{ec.algo0}
    \Theta:\mathfrak{qg}^\#\longrightarrow \mathfrak{qg}^\#\otimes \mathfrak{qg}^\# 
\end{equation} 
such that
\begin{enumerate}
    \item  $\Theta$ $\in$ $\Mor(\ad,\ad\otimes \ad)$.
    \item The Maurer--Cartan formula holds for $\Theta$:  for every $\theta$ $\in$ $\mathfrak{qg}^\#$, we get 
    \begin{equation*}
        d\theta=m_{\Gamma^\wedge}(\Theta(\theta)).
    \end{equation*}
\end{enumerate}
\end{Definition}

In this sense, an embedded differential can be regarded as a \emph{mechanism to correctly combine} the quantum Lie algebra structure of $\mathfrak{qg}^\#$ with its differential structure. According to \cite{micho1}, in general, an embedded differential can be constructed by choosing a $\ast$--$S$--invariant $\Ad$--invariant complement $L \subset \Ker(\epsilon)$ of $\mathcal{R}$ and taking  
\begin{equation}
    \label{emdiff}
    \Theta=-(\pi\otimes \pi)\circ \Delta\circ \pi^{-1}|_{L}.
\end{equation}
For example, if $\beta=\{\theta_i=\pi(g_i) \}$ is a linear basis of $\mathfrak{qg}^\#$ such that $\theta^\ast_i$ $\in$ $\beta$, then the linear space $L$ generated by $\{ g_i-\epsilon(g_i)\mathbbm{1}\}$ defines an embedded differential by means of equation (\ref{emdiff}). In particular, this implies that embedded differentials are not unique in general.

Now, the motivation for the following definition in Durdevich’s formulation of qpb's should be clear.

\begin{Definition}
    \label{curdef}
   Let $\zeta=(P,B,\Delta_P)$ be a qpb with a differential calculus and fix any such embedded differential $\Theta$. We define
the {\it curvature} of a qpc $\omega$ as the linear map
\begin{equation}
\label{2.f28}
R^{\omega}:=d\omega-\langle \omega,\omega\rangle : \mathfrak{qg}^\#\longrightarrow  \Omega^2(P),
\end{equation}
where $$\langle \omega,\omega\rangle:=m_\Omega\circ(\omega\otimes \omega)\circ\Theta: \mathfrak{qg}^\#\longrightarrow \Omega^2(P).$$ A qpc $\omega$ is called flat if $R^\omega=0$.
\end{Definition}

By the properties of $\Theta$, it can be proven that  $\Im(R^{\omega})\subseteq \Hor^2 P$ and
\begin{equation}
\label{2.f29}
R^{\omega}\,\in\, \Mor^2(\ad,\Delta_\Hor)
\end{equation}
for every qpc $\omega$ (see Theorem 12.11 in reference \cite{stheve}), where \begin{equation}
 \label{2.f24.252}
\begin{aligned}
    \Mor^2(\ad,\Delta_\Hor)=\{\lambda:\mathfrak{qg}^\# \longrightarrow \Hor^2 P\mid \lambda \mbox{ is} &\mbox{ linear and} \\
    & (\lambda\otimes \id_H)\circ \ad=\Delta_\Hor \circ \lambda\}.
\end{aligned}
\end{equation}

According to references \cite{micho2, stheve}, if $\omega$ is multiplicative, then $R^\omega$ does not depend on the choice of $\Theta$.  This is the case, for example, of qpc of the form $\omega^\#_{\mathrm{class}}$ (qpc's that come from the \emph{dualization via the pull--back} of \emph{classical} principal connections).  It is worth mentioning that, for multiplicative qpc’s, one has (\cite{micho2}) $$r^\omega=R^\omega\circ \pi.$$ In other words, only for multiplicative qpc's, the notions of curvature in the Brzezi\'nski--Majid and Durdevich formulations are directly related. However, in general, there exist qpc’s that are not multiplicative (\cite{micho2}); so in general, $r^\omega$ and $R^\omega$ are not--related operators.

It is worth noticing that according to \cite{micho2}, for every qpc $\omega$, the domain of its covariant derivatives $D^\omega$, $\widehat{D}^\omega$ can be naturally extended to $\Omega^\bullet(P)$ and we obtain $$D^\omega\omega=\widehat{D}^\omega \omega=R^\omega,$$ as in the \emph{classical} case (\cite{nodg}). The last equation provides another clear justification for Definition~\ref{curdef}; however, for the purposes of this paper, it is not necessary to pursue this part of Durdevich’s theory further. Interested readers may consult~\cite{micho2} for details.

In general, for any $\ast$--algebra $X$ and linear maps $$T_1,\,\,T_2:\mathfrak{qg}^\# \longrightarrow X,$$ we define 
\begin{equation}
\label{trans}
[ T_1,T_2]:=m_X \circ (T_1\otimes T_2)\circ \mathrm{c}^T:\mathfrak{qg}^\# \longrightarrow X,
\end{equation}
and
\begin{equation}
\label{trans0.1}
\langle T_1,T_2\rangle:=m_X \circ (T_1\otimes T_2)\circ \Theta:\mathfrak{qg}^\# \longrightarrow X,
\end{equation}
for an embedded differential $\Theta$, where  $m_X:X\otimes X\longrightarrow X$ is the product map of $X$.

The $\ast$ operation on $\mathfrak{qg}^\#$  satisfies  $$(\ast\otimes \ast) \circ \ad=\ad \circ \ast$$ and motivates the following operation.
 
\begin{Definition}
    \label{astdef}
    We define the wedge antilinear involution map on $\Mor(\ad,\Delta_{\Omega^\bullet(P)})$ (see Definition \ref{qpc's} and references \cite{micho2,stheve}) by 
    \begin{equation}
    \label{ad.f3}
    \begin{aligned}
        \wedge: \Mor(\ad,\Delta_{\Omega^\bullet(P)})\longrightarrow \Mor(\ad,\Delta_{\Omega^\bullet(P)}),\qquad
        \psi  \longmapsto \widehat{\psi}:=\ast \circ \psi \circ \ast.
    \end{aligned} 
\end{equation}
\end{Definition}

It is worth mentioning that for every $\psi$ $\in$ $\Mor(\ad,\Delta_\Hor)$, we have $\widehat{\psi}$ $\in$ $\Mor(\ad,\Delta_\Hor)$. In  Proposition 12.13 of reference \cite{stheve}, the reader can find the proof of the following statement.
\begin{Proposition}
\label{util}
     For every $\psi$, $\phi$ $\in$ $\Mor(\ad,\Delta_{\Omega^\bullet(P)})$ with $\Im(\psi)\subseteq \Omega^k(P)$, $\Im(\psi)\subseteq \Omega^l(P)$, we have $$[\psi,\phi],\; \langle \psi,\phi\rangle \,\in\, \Mor(\ad,\Delta_{\Omega^\bullet(P)})$$ with $\Im([ \psi,\phi]),\;\Im(\langle \psi,\phi\rangle)\subseteq \Omega^{k+j}(P)$. Furthermore, $$\widehat{\langle\psi,\phi\rangle}:=\ast\circ \langle\psi,\phi\rangle\circ \ast=-(-1)^{kl} \langle\widehat{\phi},\widehat{\psi}\rangle.$$  There are analogous results for $\Mor(\ad,\Delta_{\Hor})$.
\end{Proposition}
It follows from the  last proposition that  $\widehat{\langle\omega,\omega\rangle}=\langle \omega,\omega\rangle$ (because $\widehat{\omega}=\omega$) and hence
\begin{equation}
    \label{2.f29.1}
    \widehat{R^{\omega}}=R^\omega. 
\end{equation}

The proof of the following proposition can be found in Proposition 4.7 of reference \cite{micho2}. 

\begin{Proposition}
    \label{prop0.1}
    Let $\omega$ be a qpc.  By equation (\ref{2.f31}), we can consider 
    \begin{equation*}
        \begin{aligned}
            D^\omega:\Mor(\ad,\Delta_\H)\longrightarrow \Mor(\ad,\Delta_\H),\qquad
            \tau  \longmapsto D^\omega(\tau)
        \end{aligned}
    \end{equation*}
    given by $$D^\omega(\tau)(\theta)=D^\omega(\tau(\theta)),$$ for all $\theta$ $\in$ $\mathfrak{qg}^\#$. Then 
    $$D^\omega(\tau)=d\tau-(-1)^k[\tau,\omega],$$ if $\Im(\tau)$ $\subseteq$ $\Hor^k P$.    
\end{Proposition}

\noindent Notice that the last proposition is exactly the {\it dualization} of the well--known result in differential geometry about the action of the covariant derivative on basic differential forms of type $\ad$ (\cite{nodg}).  

Let $\delta^V$ be a finite--dimensional $H$--corepresentation. By equation (\ref{2.f31}), we can always consider  
\begin{equation}
\label{new1}
        \begin{aligned}
        D^\omega:\Mor(\delta^V,\Delta_\H)\longrightarrow \Mor(\delta^V,\Delta_\H)\qquad
            \tau  \longmapsto D^\omega(\tau),
        \end{aligned}
    \end{equation}
\begin{equation}
\label{new2}
        \begin{aligned}
    \widehat{D}^\omega:\Mor(\delta^V,\Delta_\H)\longrightarrow \Mor(\ad,\Delta_\H)\qquad
            \tau  \longmapsto \widehat{D}^\omega(\tau),
        \end{aligned}
    \end{equation}
given by $$D^\omega(\tau)(v)=D^\omega(\tau(v))\quad \mbox{ and }\quad   \widehat{D}^\omega(\tau)(v)=\widehat{D}^\omega(\tau(v))=(D^\omega(\tau(v)^\ast))^\ast$$
for all $v$ $\in$ $V$. 

In general,  there is no way to define a $\ast$ operation on $V$ such that $(\ast\otimes \ast)\circ \delta^V=\delta^V\circ \ast;$ so  we cannot define the $\wedge$ operation of Definition \ref{astdef} for an arbitrary $\delta^V$. However, for $\tau$ $\in$ $\Mor(\ad,\Delta_\Hor)$, we have
  $$\widehat{D}^\omega(\tau)=(\wedge \circ D^\omega\circ \wedge)(\tau).$$

\begin{Definition}
\label{koperator}
    Let $\delta^V$ be a finite--dimensional $H$--corepresentation and $\lambda$ $\in$ $\overrightarrow{\mathfrak{qpc}(\zeta)}$ (see equation (\ref{2.f24.2})). We define the operator
    \begin{equation*}
        \begin{aligned}
          K^{\lambda}: \Mor(\delta^V,\Delta_\Hor)\longrightarrow  \Mor(\delta^V,\Delta_\Hor),\qquad
          \tau  \longmapsto K^{\lambda}(\tau)
        \end{aligned}
    \end{equation*}
    given by $$K^{\lambda}(\tau)(v)= -(-1)^k (\tau^{(0)}(v))\,\lambda(\pi(\tau^{(1)}(v))) $$ for all $v$ $\in$ $V$, if $\Im(\tau)$ $\subseteq$ $\Hor^k P$ and $\Delta_\H(\tau(v))=\tau^{(0)}(v)\otimes \tau^{(1)}(v)$. In the same way, we define  the dual $K^{\lambda}$ operator  
    \begin{equation*}
        \begin{aligned}
          \widehat{K}^{\lambda}: \Mor(\delta^V,\Delta_\Hor)\longrightarrow  \Mor(\delta^V,\Delta_\Hor),\qquad
          \tau & \longmapsto \widehat{K}^{\lambda}(\tau)
        \end{aligned}
    \end{equation*}
    as
    $$(\widehat{K}^\lambda(\tau))(v)=(K^\lambda(\tau(v)^\ast))^\ast$$ for all $v$ $\in$ $V$.
\end{Definition}

\noindent Since $\mathfrak{qpc}(\zeta)$ is an affine space modeled by $\overrightarrow{\mathfrak{qpc}(\zeta)}$, for every $\omega$ $\in$ $\mathfrak{qpc}(\zeta)$ and every $\lambda$ $\in$ $\overrightarrow{\mathfrak{qpc}(\zeta)}$, we have $\omega+\lambda$ $\in$ $\mathfrak{qpc}(\zeta)$ and by equations (\ref{2.f30}), (\ref{2.f30.1}) we get
\begin{equation}
    \label{2.f30.1.1}
    D^{\omega+\lambda}=D^\omega+K^\lambda \quad \mbox{ and }\quad
 \widehat{D}^{\omega+\lambda}=\widehat{D}^\omega+\widehat{K}^\lambda.
\end{equation}

\begin{Definition}
    \label{Soperator}
    Let $\omega$ be a qpc. We define the operator 
    $$S^{\omega}: \Mor(\ad,\Delta_\H) \longrightarrow \Mor(\ad,\Delta_\H)$$
   given by
    \begin{equation*}
S^{\omega}(\tau):=\langle \omega,\tau\rangle-(-1)^k\langle\tau,\omega\rangle-(-1)^k[\tau,\omega]
\end{equation*}
for every $\tau$ $\in$ $\Mor(\ad,\Delta_\H)$ with $\Im(\tau)$ $\subseteq$ $\Hor^k P$. Similarly, we define the dual $S^\omega$ operator as $$\widehat{S}^\omega:=\wedge\circ S^\omega\circ \wedge.$$ 
\end{Definition}
In reference \cite{micho2}, the operator $S^{\omega}$ is denoted by $q_\omega$ and the reader is encouraged to consult this reference for more details on the operator $S^{\omega}$. For example, $S^{\omega}=0$ when $\omega$ is regular.

\begin{Definition}
    \label{twisted}
    Let $\zeta=(P,B,\Delta_P)$ be a qpb with a differential calculus. We define the twisted covariant derivative of a qpc $\omega$ as the operator $$DS^\omega:=D^\omega-S^\omega: \Mor(\ad,\Delta_\H) \longrightarrow \Mor(\ad,\Delta_\H).$$ Explicitly, by Proposition \ref{prop0.1}, for every $\tau$ $\in$ $\Mor(\ad,\Delta_\H)$ with $\Im(\tau)$ $\subseteq$ $\Hor^k P$, we have 
    \begin{equation}
        \label{twcoder}
        DS^\omega(\tau)=d\tau-\langle \omega,\tau\rangle+(-1)^k\langle\tau,\omega\rangle.
    \end{equation}
    In the same way, we define the dual twisted covariant derivative as the operator $$\widehat{DS^\omega}=\wedge \circ DS^\omega\circ \wedge.$$
\end{Definition}

In accordance with the motivation behind of Definition \ref{embdifdef} and Proposition \ref{prop0.1}, the operators $DS^\omega$, $\widehat{DS^\omega}$ have to be interpreted as the {\it correct covariant derivative on} $\Mor(\ad,\Delta_\Hor)$, in the sense that both operators covariantly differentiates elements of $\Mor(\ad,\Delta_\Hor)$, correctly combining the quantum Lie algebra structure of $\mathfrak{qg}^\#$ with its differential structure.

Let us consider the space 
\begin{equation}
    \label{realtensor}
    \Mor(\ad,\Delta_\Hor)^{\dagger}:=\{\tau\,\in\, \Mor(\ad,\Delta_\Hor) \mid \widehat{\tau}=\tau \}.
\end{equation}
Notice that (see equations (\ref{2.f24.2}), (\ref{2.f29.1})) $$\overrightarrow{\mathfrak{qpc}(\zeta)}\subseteq  \Mor(\ad,\Delta_\Hor)^{\dagger} \qquad \mbox{and}\qquad R^\omega\,\in\, \Mor(\ad,\Delta_\Hor)^{\dagger}. $$
By Proposition \ref{util} and equation (\ref{twcoder}), we obtain
    \begin{equation}
        \label{twcoder1}
        \widehat{DS^\omega}(\tau)=DS^\omega(\tau)
    \end{equation}
for all $\tau$ $\in$ $\Mor(\ad,\Delta_\Hor)^{\dagger}$. This implies that 
    \begin{equation}
        \label{twcoder2}
DS^\omega|_{\Mor(\ad,\Delta_\Hor)^{\dagger}}=\widehat{DS^\omega}|_{\Mor(\ad,\Delta_\Hor)^{\dagger}}:\Mor(\ad,\Delta_\Hor)^{\dagger}\longrightarrow \Mor(\ad,\Delta_\Hor)^{\dagger}.
    \end{equation}

The twisted covariant derivative will be essential for the theory we aim to develop. As evidence for this, we have the {\it non--commutative geometrical} version of the second Bianchi identity: 
\begin{equation}
\label{a.5}
\widehat{DS}^{\omega}(R^{\omega})=DS^{\omega}(R^{\omega})=\langle \omega,\langle \omega,\omega\rangle\rangle-\langle\langle \omega,\omega\rangle,\omega\rangle.
\end{equation}
The proof of the last equation is a straightforward calculation that holds for every qpc; there is no need to assume any additional condition on $\omega$, as the reader can verify in Proposition 4.9 of reference \cite{micho2}.  When $\omega$ is regular (\cite{micho2}) $$D^\omega=\widehat{D}^\omega\qquad \mbox{ and }\qquad S^{\omega}=\widehat{S}^\omega=0;$$ and if $\omega$  is multiplicative (\cite{micho2}) $$\langle \omega,\langle \omega,\omega\rangle\rangle-\langle\langle \omega,\omega\rangle,\omega\rangle=0.$$ So, if $\omega$ is both regular and multiplicative (for example, for a qpc of the form $\omega^\#_{\mathrm{class}}$), we get the ({\it dualization via the pull--back} of the) second Bianchi identity in differential geometry: $$\widehat{D}^\omega(R^\omega)=D^{\omega}R^{\omega}=0.$$

To conclude this subsection, we make the following assumption. 

\begin{Remark}
\label{rema}
From this point onward until the end of the paper, we shall restrict our attention exclusively to qpb's for which the quantum base space $B$ is stable under holomorphic calculus. According to Appendix B of reference \cite{micho3}, in this case, for every $\delta^V$ $\in$ $\T$ there exists a set  $$\{T^\l_k \}^{d_{V}}_{k=1} \subseteq \Mor(\delta^V,\Delta_P)$$ for some $d_{V}$ $\in$ $\N$ such that
\begin{equation}
    \label{generators}
\sum^{d_{V}}_{k=1}x^{V\,\ast}_{ki}x^{V}_{kj}=\delta_{ij}\mathbbm{1},
\end{equation}
with $x^{V}_{ki}:=T^\l_k(e_i)$. Here, $\T$  is a complete set of mutually non--equivalent irreducible (necessarily finite--dimensional) $H$--corepresentations with $\delta^\C_\triv$ $\in$ $\T$ (the trivial corepresentation on $\C$), and $\{e_i\}^{n_{V}}_{i=1}$ is the orthonormal basis of $V$ given in Theorem \ref{rep}.
\end{Remark}

The introduction of the maps $\{T^\l_k\}$ provides a valuable technical tool for carrying out explicit calculations, as the reader can see, for example, in Appendix~A and Proposition \ref{gaugeym}; and as we have commented, these maps exist if $B$ is stable under holomorphic calculus \cite{micho3}.

It is worth mentioning that this assumption on $B$ is easily satisfied and constitutes a common assumption in Durdevich's formulation of qpb's (see, for example, references \cite{micho3,micho7}) and in non--commutative geometry. For instance, the Yang--Mills theory formulated by A. Connes in reference \cite{con} holds solely for quantum spaces $B$ that are stable under holomorphic calculus (among other conditions on $B$). In Proposition 2.9 of reference \cite{sald2} we determine the explicit form of the operators $\{T^\l_k\}$ in differential geometry for a \emph{classical} principal bundle. Moreover, in Section 5 of reference \cite{sald2} we determine the explicit form of these operators for trivial qpb's and homogeneous qpb's.

\subsection{Associated Quantum Vector Bundles}

In this section, we present the \emph{non commutative geometrical} counterpart of the theory of associated vector bundles within Durdevich’s formulation of quantum principal bundles. For further details, we refer the reader to reference \cite{sald2}.

Let $\pi:GM\longrightarrow M$ be a  principal $G$--bundle in differential geometry and consider a finite--dimensional unitary representation $\alpha^V:G\longrightarrow GL(V)$. The associated vector bundle with respect to $\alpha^V$ is defined as $$\pi_{\alpha^V}:E^V\longrightarrow M, \qquad [x,v]\longmapsto \pi(x),$$ where $$E^V:=GM\times_G V:=(GM\times V)/G, $$ where the $G$--action on $GM\times V$ is given by $$(x,v,A)\longmapsto (x\cdot A,\alpha^V(A^{-1})(v)),$$ for all $x$ $\in$ $GM$, $v$ $\in$ $V$ and $A$ $\in$ $G$. In light of the Serre--Swan theorem, the associated vector bundle is equivalent to the finitely generated projective $C^\infty_\C(M)$--bimodule of its global smooth sections $$\Gamma(E^V).$$ 

It is well--known that there exists a $C^\infty_\C(M)$--bimodule isomorphism (\cite{nodg}) 
\begin{equation}
    \label{dif1}
    GP^{-1}_0:\Gamma(E^V)\longrightarrow C^\infty_\C(GM,V)^G,
\end{equation}
where $$C^\infty_\C(GM,V)^G $$ is the space of $G$--equivariant smooth maps between $GM$ and $V$; so the associated vector bundle is also equivalent to $C^\infty_\C(GM,V)^G $. Identifying the dual space $V^\#$ of $V$ with $V$,  the pull--back induces a $C^\infty_\C(M)$--bimodule isomorphism 
\begin{equation}
    \label{dif2}
    \#: C^\infty_\C(GM,V)^G\longrightarrow \Mor(\delta^V,\Delta_P),
\end{equation}
where $\delta^V$ is the corepresentation given by the pull--back of $\alpha^V$, and $\Delta_P$ is the pull--back of the  complex--extension of the canonical $G$--action on $GM$. Hence, the associated vector bundle is also equivalent to $\Mor(\delta^V,\Delta_P)$.

On the other hand, by the Serre--Swan theorem, in non--commutative geometry, quantum vector bundles are defined as left or right finitely generated projective $B$--modules, for a quantum space $B$ \cite{con, dv}. In addition, it can be proven that for a given quantum principal $\G$--bundle over $B$ and a finite--dimensional $H$--corepresentation $\delta^V$, the space $\Mor(\delta^V,\Delta_P)$ is always a left/right finitely generated projective $B$--module, where the left/right $B$--module structure is given by the left/right multiplication by elements of $B$ (\cite{sald2,micho2}). Motivated by all these facts, in  Durdevich's formulation of qpb's we have the following definition.
\begin{Definition}
\label{qvb's}
    Let $\zeta=(P,B,\Delta_P)$ be a quantum principal $\G$--bundle and let $\delta^V$ $\in$ $\Obj(\Rep_{H})$. We define the {\it associated left quantum vector bundle} (abbreviated ``associated left qvb") as the  left $B$--module $$E^V_\l:=\Mor(\delta^V,\Delta_P),$$ and we define the {\it associated right quantum vector bundle} (abbreviated ``associated right qvb") as the right $B$--module $$E^V_\r:=\Mor(\delta^V,\Delta_P).$$
\end{Definition}
From now on, we will write $\Mor(\delta^V, \Delta_P)$ to refer to the $B$--bimodule structure or the $\C$--vector space structure of this set (while the notation $E^V_\l/E^V_\r$ is to refer to the left/right $B$--module structure of this set).

In the Brzezi\'nski--Majid formulation of qpb’s, associated quantum vector bundles are defined using the cotensor product (\cite{libro}) rather than intertwiner maps, as in Durdevich’s formulation. However, according to Section~6 of reference \cite{br2}, the two definitions are isomorphic.

Let $\delta^V$ be a finite--dimensional $H$--corepresentation. Then, it is well--known that there exist $\delta^{V_1}$,..., $\delta^{V_m}$ $\in$ $\T$ such that (\cite{woro1}) $$\delta^V\cong \bigoplus^m_{j=1} \delta^{V_j}.$$ According to Remark \ref{rema}, for each $\delta^{V_j}$ there exist the operators $\{T^\l_k\}$ and hence, we can consider the union of all such operators $\{T^\l_k\}$. In this way, in accordance with  Section 3 of reference \cite{sald2}, the following isomorphisms holds:
\begin{equation}
    \label{new.5}
    E^V_\l\cong B^{d}\cdot \varrho^V(\mathbbm{1})\qquad \mbox{ with }\qquad \varrho^V(\mathbbm{1})=(\sum_i x^{V}_{ki}\,\,x^{V\,\ast}_{li})_{kl}\;\in\; M_{d}(B)
\end{equation}
as left $B$--modules for some $d$ $\in$ $\N$, where $x^{V}_{ki}:=T^\l_k(e_i)$ (see Remark \ref{rema}) and $M_{d}(B)$ denotes the space of $d\times d$ matrices with entries in $B$. In particular,  for every $T$ $\in$ $E^V_\l$, we get (\cite{sald2})
\begin{equation}
    \label{leftB}
    T=\sum_k b^{_{T}}_k\,T^\l_k\quad \mbox{ with } \quad b^{_{T}}_k=\sum_{i}T(e_i)\,x^{V\,\ast}_{ki}\;\in\;B,
\end{equation}
Notice that the superscript $\l$ is to indicate that $\{T^\l_k\}$ are left $B$--generators of $E^V_\l$.

Similarly, we obtain (see reference \cite{sald2})
\begin{equation}
    \label{new.6}
    E^V_\r\cong \varrho^{\overline{V}}(\mathbbm{1}) \cdot B^{s} 
\end{equation}
as right $B$--modules for some $s$ $\in$ $\N$, where $\varrho^{\overline{V}}(\mathbbm{1})$ is the corresponding matrix of equation (\ref{leftB}) for the complex conjugate $H$--corepresentation $\delta^{\overline{V}}$ of $\delta^V$.

If  $T$ $\in$ $\Mor(\delta^V,\Delta_P)$, then $T^\ast$ $\in$ $\Mor(\delta^{\overline{V}},\Delta_P)$ and therefore $$T^\ast=\displaystyle \sum_k b^{T^\ast}_kT^\l_k.$$ Here, the maps $\{T^\l_k\}$ are the corresponding left $B$--generators of $\Mor(\delta^{\overline{V}},\Delta_\Hor)$ and the linear map $T^\ast:\overline{V}\longrightarrow P
$ is given by $T^\ast(\overline{v}):=T(v)^\ast$ for all $\overline{v}$ $\in$ $\overline{V}$. Hence (\cite{sald2})
\begin{equation}
\label{3.f5.2}
T=\displaystyle \sum_k T^\r_k (\mu^{T^\ast}_k)^\ast
\end{equation} 
with $T^\r_k:=T^{\l\,\ast}_k$ $\in$ $\Mor(\delta^V,\Delta_P)$. The superscript $\r$ is to indicate that $\{T^\r_k\}$  right $B$--generators of  $E^V_\r$. 

As we have commented before, the operators $\{T^\l_k\}$ exist in differential geometry, and the reader can check their explicit form in Proposition 2.9 of reference \cite{sald2}. In addition, it is easy to verify that equations (\ref{leftB}), (\ref{3.f5.2}) are also satisfied in the \emph{classical} case.

Let $\pi:GM\longrightarrow M$ be a principal $G$--bundle in differential geometry and consider the associated vector bundle $\pi_{\alpha^V}:E^V\longrightarrow M$ for a finite--dimensional unitary representation $\alpha^V$. Then, according to \cite{nodg}, the \emph{Gauge Principle} holds: there exists an  isomorphism $GP^{-1}$ between $E^V$--valued differential forms of $M$ and basic forms of type $\alpha^V$ of $GM$. This isomorphism in degree $k$
\begin{equation}
    \label{dif3}
GP^{-1}_k:\Omega^k_\C(M)\otimes_{C^\infty_\C(M)}\Gamma(E^V)\longrightarrow \Omega^k_{\C}(GM,V)^G 
\end{equation}
is given by $$ GP^{-1}_k(\mu\otimes s):=\mu\,GP^{-1}_0(s),$$  where $$(\mu\,GP^{-1}_0(s))(X_1,...,X_k):=(\pi^\#\mu)_x(X_1,...,X_k)\,GP^{-1}_0(s)(x) $$ for all $X_1,...,X_k$ $\in$ $T_x(GM)$, with $\pi^\#\mu$ the pull--back  of $\mu$ by $\pi$. Moreover, equation (\ref{dif1}) induces the isomorphism
\begin{equation}
    \label{dif4}
GP^{'-1}:\Omega^\bullet_\C(M)\otimes_{C^\infty_\C(M)}C^\infty_\C(GM,V)^G\longrightarrow \Omega^\bullet_{\C}(GM,V)^G 
\end{equation}
given by $$GP'(\mu\otimes f)=\mu\,f.$$ Extending equation (\ref{dif2}) to differential forms, we have an isomorphism
\begin{equation}
    \label{dif5}
    \#: \Omega^\bullet_{\C}(GM,V)^G\longrightarrow \Mor(\delta^V,\Delta_\Hor),
\end{equation}
where $\Delta_\Hor$ is the pull--back of the canonical $G$--action on the complexification of the horizontal bundle of $GM$; and we obtain the isomorphism
\begin{equation}
    \label{dif6}
    \Upsilon^{-1}_V: \Omega^\bullet_\C(M)\otimes_{C^\infty(M)} \Mor(\delta^V,\Delta_P)\longrightarrow \Mor(\delta^V,\Delta_\Hor)
\end{equation}
given by $$\Upsilon^{-1}_V(\mu\otimes T)=\mu\,T.$$

Moreover, any principal connection $\omega_\mathrm{class}$ of    $\pi:GM\longrightarrow M$ induces a \emph{canonical} linear connection on $E^V$ by means of (\cite{nodg})
\begin{equation}
    \label{dif7}
    \nabla^{\omega_\mathrm{class}}_V:= GP\circ D^{\omega_\mathrm{class}} \circ GP^{-1}: \Gamma(E^V)\longrightarrow \Omega^1_\C(M)\otimes_{C^\infty_\C(M)}\Gamma(E^V)
\end{equation}
and the last equation extends to the exterior derivative of $\nabla^{\omega_\mathrm{class}}_V$:
\begin{equation}
    \label{dif8}
    d^{\nabla^{\omega_\mathrm{class}}_V}:= GP\circ D^{\omega_\mathrm{class}} \circ GP^{-1}: \Omega^\bullet_\C(M)\otimes_{C^\infty_\C(M)}\Gamma(E^V)\longrightarrow \Omega^\bullet_\C(M)\otimes_{C^\infty_\C(M)}\Gamma(E^V),
\end{equation}
where $D^{\omega_\mathrm{class}}$ is the covariant derivative of $\omega_\mathrm{class}$ (see equation (\ref{dif0})). In physics, the map  $\nabla^{\omega_\mathrm{class}}_{V}$ receives the name of gauge covariant derivative.

By using  equations (\ref{dif5}), (\ref{dif6}), we can induce operators equivalent on $\Mor(\delta^V,\Delta_P)$ and $\Mor(\delta^V,\Delta_\Hor)$:
\begin{equation}
    \label{dif9}
    \nabla^{\omega^\#_\mathrm{class}}_V:= \Upsilon_V\circ \#\circ  D^{\omega_\mathrm{class}} \circ \#^{-1}, \qquad d^{\nabla^{\omega^\#_\mathrm{class}}_V}:=\Upsilon_V\circ \#\circ  D^{\omega_\mathrm{class}} \circ \#^{-1}\circ \Upsilon^{-1}_V.
\end{equation}

Equation (\ref{dif6}) motivates the following result in Durdevich's formulation of qpb. The reader can find a proof of the next proposition in Section 3 of reference \cite{sald2}

\begin{Proposition}
\label{qgaugeprinciple}
    Let $\zeta=(P,B,\Delta_P)$ be a quantum $\G$--bundle  with a differential calculus and consider $\delta^V$ $\in$ $\Obj(\Rep_H)$. Then the map $$\Upsilon_V: \Mor(\delta^V,\Delta_\Hor)\longrightarrow \Omega^\bullet(B)\otimes_B E^V_\l $$ given by
\begin{equation}
\label{3.f7}
\Upsilon_{V}(\tau)=\sum_{k}\mu^\tau_k\otimes_B T^\l_k \qquad \mbox{ with }\qquad \mu^{\tau}_k=\sum_{i}\tau(e_i)\,x^{V\,\ast}_{ki}\;\in\;\Omega^\bullet(B)
\end{equation}
is the inverse of the left $B$--module morphism  
\begin{equation}
\label{3.f5.1.2.3}
\Upsilon^{-1}_{V}:\Omega^\bullet(B)\otimes_{B}E^V_\l\longrightarrow\Mor(\delta^V,\Delta_\Hor) \qquad \mbox{ given by }\qquad \Upsilon^{-1}_{V}(\mu\otimes_{B} T)=\mu\, T.
\end{equation}

Similarly, the map 
\begin{equation}
\label{3.f7.5}
\widehat{\Upsilon}_{V}: \Mor(\delta^V,\Delta_\Hor)\longrightarrow E^V_\r\otimes_B \Omega^\bullet(B),\qquad  \tau\longmapsto\sum_k T^\r_k\otimes_B (\mu^{\tau\ast}_k)^\ast  
\end{equation}
is the inverse of the right  $B$--module morphism 
\begin{equation}
    \label{3.f5.1.2.3.1}
    \widehat{\Upsilon}^{-1}_{V}:  E^V_\r\otimes_B \Omega^\bullet(B)\longrightarrow \Mor(\delta^V,\Delta_\Hor)\qquad \mbox{ given by} \qquad \widehat{\Upsilon}^{-1}_{V}(T\otimes_B\mu)=T\,\mu.
\end{equation}
\end{Proposition}

\begin{Remark}
\label{remarkinverse}
Apparently, $\Upsilon_{V}$ and $\widehat{\Upsilon}_{V}$ depend on the set of generators $\{ T^\l_k\}$, $\{ T^\r_k\}$ of $\Mor(\delta^{V},\Delta_P)$, respectively. However,  the uniqueness of the inverse map ensures that both operators are independent of the choice of these sets.
\end{Remark}

\begin{Remark}
    \label{interpretations}
    In accordance with the classical case, elements of $\Omega^\bullet(B)\otimes_{B}E^V_\l$ can be interpreted as $E^V_\l$--valued differential forms of $B$; while elements of $E^V_\r\otimes_{B}\Omega^\bullet(B)$ can be interpreted as $E^V_\r$--valued differential forms of $B$. Finally, elements of $\Mor(\delta^V,\Delta_\Hor)$ can be interpreted as quantum basic differential forms of type $\delta^V$.

    In the classical case, the curvature of a principal connection is a differential $2$--form of type $\ad_\mathrm{class}$; so from a physical interpretation, the curvature is a tensor field.  By equation (\ref{2.f29}), we get that $R^\omega$  is a basic  quantum differential $2$--form of type $\ad$ for every qpc $\omega$ and thus, in terms of a  physical interpretation, we can consider $R^\omega$ as a  non--commutative geometrical tensor field.
\end{Remark}

Under these interpretations and by the first part of equation (\ref{dif9}) (taking into account that in non--commutative geometry, the map $\#$ is not necessary), we have the following definition in Durdevich's formulation of qpb's \cite{sald2}.

\begin{Definition}
    \label{inducedqvb}
   Let $\zeta=(P,B,\Delta_P)$ be a quantum $\G$--bundle  with a differential calculus. Let $\omega$ be a qpc and consider $\delta^V$ $\in$ $\Obj(\Rep_{H})$ and equation (\ref{new1}). Then,  the linear map
\begin{equation}
\label{3.f8}
\nabla^{\omega}_{V}:E^V_\l \longrightarrow \Omega^{1}(B)\otimes_B E^V_\l,\qquad
T  \longmapsto \Upsilon_{V}(D^{\omega}(T)),
\end{equation}
is a {\it quantum linear connection} on $E^V_\l$, in the sense of reference \cite{dv}, i.e., $\nabla^{\omega}_{V}$ satisfies the left Leibniz rule. Similarly, considering equation (\ref{new2}),  the linear map
\begin{equation}
\label{3.f8.1}
\widehat{\nabla}^{\omega}_{V}:E^V_\r \longrightarrow E^V_\r\otimes_{B}\Omega^{1}(B),\qquad
T  \longmapsto \widehat{\Upsilon}_{V}(\widehat{D}^{\omega}(T)),
\end{equation}
is a quantum linear connection on $E^V_\r$, i.e., $\widehat{\nabla}^{\omega}_{V}$ satisfies the right Leibniz rule. The maps  $\nabla^{\omega}_{V}$ and $\widehat{\nabla}^{\omega}_{V}$ receive the name of  induced quantum linear connections of $\omega$ (abbreviated ``induced qlc's"). 
\end{Definition}

It worth remarking that our formulation holds for every qpc $\omega$: it is not necessary to impose any condition on $\omega$ to define the induced qlc's.

\begin{Remark}
    \label{remacon}
    By defining $\sigma_{V}:=\widehat{\Upsilon}_{V}\circ \Upsilon^{-1}_{V}$, we obtain that
\begin{equation}
\label{3.f10}
\sigma_{V}\circ \nabla^{\omega}_{V}=\widehat{\nabla}^{\omega}_{V}
\end{equation} 
when $\omega$ is regular \cite{sald1}. This is the main reason for not using the $B$--bimodule structure of $\Mor(\delta^V,\Delta_P)$, choosing instead to handle the left and right structures separately:  only for regular qpc's the induced qlc's can be combined into a $B$--bimodule quantum linear connection on $\Mor(\delta^V,\Delta_P)$.
\end{Remark}

By definition and equation (\ref{2.f30.1.1}), we obtain
\begin{equation}
    \label{nedded1}
\nabla^{\omega+\lambda}_V(T)=\nabla^\omega_V(T)+\Upsilon_V(K^\lambda(T)),\qquad \widehat{\nabla}^{\omega+\lambda}_V(T)=\widehat{\nabla}^\omega_V(T)+\widehat{\Upsilon}_V(\widehat{K}^\lambda(T))
\end{equation}
for all $\lambda$ $\in$ $\overrightarrow{\mathfrak{qpc}(\zeta)}$ (see equation (\ref{2.f24.2})).

According to Section 3 of reference \cite{sald2}, the exterior covariant derivative  $$d^{\nabla^{\omega}_{V}}: \Omega^\bullet(B)\otimes_B E^V_\l\longrightarrow \Omega^\bullet(B)\otimes_B E^V_\l \quad \mbox{such that}\quad  d^{\nabla^{\omega}_{V}}(\mu\otimes_B T)= \mu\otimes_B T +(-1)^k\mu \nabla^{\omega}_{V}(T)$$  for all $\mu$ $\in$ $\Omega^k(B)$, satisfies  
\begin{equation}
\label{3.f10.3}
d^{\nabla^{\omega}_{V}}= \Upsilon_{V}\circ D^{\omega}\circ \Upsilon^{-1}_{V}.
\end{equation}
Similarly,  the exterior covariant derivative  $$d^{\widehat{\nabla}^{\omega}_{V}}: E^V_\r \otimes_B \Omega^\bullet(B)\longrightarrow E^V_\r \otimes_B \Omega^\bullet(B) \quad \mbox{such that}\quad d^{\widehat{\nabla}^{\omega}_{V}}(T\otimes_B \mu)= \widehat{\nabla}^{\omega}_{V}(T) \mu +T \otimes_B d\mu$$ satisfies
\begin{equation}
\label{3.f10.6}
d^{\widehat{\nabla}^{\omega}_{V}}= \widehat{\Upsilon}_{V}\circ \widehat{D}^{\omega}\circ \widehat{\Upsilon}^{-1}_{V}.
\end{equation}
Equations (\ref{3.f10.3}), (\ref{3.f10.6}) are the \emph{non--commutative geometrical} generalization of the second part of equation (\ref{dif9}).

Let $\pi:GM\longrightarrow M$ be a principal $G$--bundle and consider $\alpha^V:G\longrightarrow GL(V)$ a unitary finite--dimensional representation. Then, there exists a \emph{canonical} Hermitian structure
\begin{equation}
    \label{dif10}
    (-,-):\Gamma(E^{V})\times \Gamma(E^{V})\longrightarrow C^\infty_\C(M)
\end{equation}
on the associated vector bundle. It is worth mentioning that for every principal connection $\omega_\mathrm{class}$, the induced linear connection (see equation (\ref{dif7})) $$\nabla^{\omega_\mathrm{class}}_V$$  is Hermitian with respect to $(-,-)$ \cite{gtvp}. This motivates the following construction and Theorem \ref{fgs} in Durdevich's formulation of qpb's.

Let $\zeta=(P,B,\Delta_P)$ be a qpb with a differential calculus. Then, in light of Section 3.2 of reference \cite{sald2}, for every $\delta^V$ $\in$ $\Obj(\Rep_{H})$, the canonical 
non--degenerated Hermitian structure on the free left $B$--module $B^{d}$ (see equation (\ref{canoleft2}) in Appendix B) induces a non--degenerated Hermitian structure on $E^V_\l$
\begin{equation}
    \label{canoleft0}
    (-,-)_\l: E^V_\l\times E^V_\l\longrightarrow B,\qquad (T_1,T_2)\longmapsto \sum_k T_1(e_k)\,T_2(e_k)^\ast 
\end{equation}
that is actually a $B$--valued inner product for $E^V_\l$, where $\{e_k \}$ is any orthonormal linear basis of $V$ with respect of the inner product that makes $\delta^V$ unitary. Moreover, $(-,-)_\l$ can be extended to left qvb--valued differential forms of $B$ 
\begin{equation}
    \label{canoleft1}
    (-,-)^\bullet_\l: (\Omega^\bullet(B)\otimes_B E^V_\l)\times (\Omega^\bullet(B)\otimes_B E^V_\l)\longrightarrow \Omega^\bullet(B)
\end{equation}
by means of $$ (\mu_1\otimes_B T_1,\mu_2\otimes_B T_2)^\bullet_\l=\mu_1 \,(T_1, T_2)_\l\,\mu^\ast_2.$$

Similarly, we have a non--degenerated Hermitian structure on $E^V_\r$
\begin{equation}
    \label{canoright0}
    (-,-)_\r: E^V_\r\times E^V_\r\longrightarrow B,\qquad (T_1,T_2)\longmapsto \sum_k T_1(e_k)^\ast\,T_2(e_k) 
\end{equation}
that is actually a $B$--valued inner product for $E^V_\r$. Furthermore, $(-,-)_\r$ can be extended to right qvb--valued differential forms of $B$ 
\begin{equation}
    \label{canoright1}
    (-,-)^\bullet_\r: (E^V_\r \otimes_B\Omega^\bullet(B) )\times (E^V_\r\otimes_B\Omega^\bullet(B))\longrightarrow \Omega^\bullet(B)
\end{equation}
by means of $$ ( T_1\otimes_B \mu_1,T_2\otimes_B \mu_2 )^\bullet_\r=\mu^\ast_1 \,(T_1, T_2)_\r\,\mu_2.$$

The proof of the following theorem can be found in Theorem
3.10 of reference \cite{sald2}.

\begin{Theorem}
\label{fgs}
Let $\zeta$ be a qpb with a qpc $\omega$ and consider $\delta^V$ $\in$ $\Obj(\Rep_{H})$. Then $\nabla^\omega_V$ and $\widehat{\nabla}^\omega_V$ are Hermitian, i.e., for every $T_1$, $T_2$ $\in$ $E^V_\l$ we have
$$( \nabla^\omega_V T_1,T_2 )^\bullet_\l+ (T_1,\nabla^\omega_V T_2)^\bullet_\l=d(T_1,T_2)_\l$$ and for every $T_1$, $T_2$ $\in$ $E^V_\r$ we get $$ (\widehat{\nabla}^\omega_V T_1,T_2)^\bullet_\r+( T_1,\widehat{\nabla}^\omega_V T_2)^\bullet_\r=d(T_1,T_2)_\r.$$
\end{Theorem}

It is worth mentioning that in this paper, we have defined a qpc as a linear map $$\omega:\mathfrak{qg}^\#\longrightarrow \Omega^1(P)$$ that satisfies 
$\Delta_{\Omega^\bullet(P)}(\omega(\theta))=(\omega\otimes \id_H)\ad(\theta)+\mathbbm{1}\otimes\theta$ and $\omega(\theta^\ast)=\omega(\theta)^\ast$; while in reference \cite{sald2}, a qpc is a linear map $$\omega:\mathfrak{qg}^\#\longrightarrow \Omega^1(P)$$ that only satisfies $\Delta_{\Omega^\bullet(P)}(\omega(\theta))=(\omega\otimes \id_H)\ad(\theta)+\mathbbm{1}\otimes\theta$, i.e., the condition $\omega(\theta^\ast)=\omega(\theta)^\ast$ is not necessary. In reference \cite{sald2}, when a qpc fulfills $\omega(\theta^\ast)=\omega(\theta)^\ast$ is called {\it real}; so the last theorem in reference \cite{sald2} is written in terms of real qpc's. 

For this paper, we have decided to add the condition $\omega(\theta^\ast)=\omega(\theta)^\ast$ in the definition of qpc's because
\begin{enumerate}
    \item The {\it standard} definition of qpc's in Durdevich's formulation incorporates this condition, as the reader can verify in references \cite{micho2,micho1,micho3,stheve}.
    \item With this condition, Theorem \ref{fgs} holds for every qpc.
    \item In the Section 4, we will relate qpc's to gauge boson fields, as in differential geometry. One of the {\it physical requirements} for gauge boson fields is that they must be real maps; so, the condition $\omega(\theta^\ast)=\omega(\theta)^\ast$ ensures this.
\end{enumerate}

\section{Formal Adjointability}

The previous section was devoted to presenting a summary of the motivations and essential aspects of Durdevich's formulation of qpb’s needed for the purposes of this paper. In this sense, the paper is reasonably self--contained. 

In this section, the novel part of the paper begins. We start by examining the adjointability of some of the operators introduced in Section 2.

\subsection{Formal Adjointability of the Differential}

Following the {\it classical} case,  we are going to use the Hodge star operator to define the formal adjoint operator of the differential.

\begin{Definition}
\label{a.2.1}
Let $(B,\cdot,\mathbbm{1},\ast)$ be a quantum space and let $(\Omega^\bullet(B),d,\ast)$ be a graded differential $\ast$--algebra generated by its $0$--degree  elements $\Omega^0(B)=B$ (quantum differential forms on $B$). We say that:
\begin{enumerate}
\item $B$ is orientable if there exists $n\in\N$ such that $\Omega^k(B)=0$ for all $k>n$ and
\[
\Omega^n(B)=B\,\dvol,
\]
where $0\neq\dvol\in\Omega^n(B)$ satisfies $b\,\dvol=0$ if and only if $b=0$. The element $\dvol$ is called a quantum $n$--volume form, and once such an element is fixed, we say that $B$ is oriented.

\item If $B$ is oriented, a left quantum Riemannian metric on $B$ is a family of $B$--valued inner products (antilinear in the second variable)
\[
\{\langle-,-\rangle^k_\l:\Omega^k(B)\times\Omega^k(B)\longrightarrow B\}
\]
such that for $k=0$,
\[
\begin{aligned}
\langle-,-\rangle^0_\l:\; B\times B\longrightarrow B,\qquad
(b_1,b_2)\longmapsto b_1\,b_2^\ast,
\end{aligned}
\]
and for $k=n$,
\[
\begin{aligned}
\langle-,-\rangle^n_\l:\Omega^n(B)\times \Omega^n(B)\longrightarrow B,\qquad
(b_1\,\dvol,\; b_2\,\dvol)\longmapsto b_1\,b_2^\ast,
\end{aligned}
\]
and such that
\[
\langle \mu_1\,b,\mu_2\rangle^k_\l=\langle \mu_1,\mu_2\, b^\ast\rangle^k_\l
\]
for all $\mu_1,\mu_2\in\Omega^k(B)$, $b\in B$, and $1\leq k<n$.
For our purposes, given a left quantum Riemannian metric on $B$, we define a right quantum Riemannian metric on $B$ by
\[
\langle \mu_1,\mu_2\rangle^k_\r:=\langle \mu_1^\ast,\mu_2^\ast\rangle^k_\l
\]
for all $k$. Notice that $\langle-,-\rangle^k_\r$ is now antilinear in the first variable.

\item If $B$ is oriented and $s$ is a faithful state on $B$, we define a quantum integral on $B$ by
\[
\begin{aligned}
\int_{B}: \Omega^{n}(B)\longrightarrow \C,\qquad
b\,\dvol\longmapsto s(b).
\end{aligned}
\]
We interpret this quantum integral as satisfying Stokes’ theorem by defining
\[
\begin{aligned}
\int_{\partial B}: \Omega^{n-1}(B)\longrightarrow \C,\qquad
\mu\longmapsto \int_B d\mu.
\end{aligned}
\]
If $\displaystyle\Im(d)\subseteq \Ker\left(\int_B\right)$, we say that $(B,\cdot,\mathbbm{1},\ast)$ is a quantum space without boundary (with respect to the given quantum integral).
\end{enumerate}
\end{Definition}

\noindent
Better yet, it is easy to see that
\begin{equation}
\label{a.f2.0}
\dvol\,b=E(b)\,\dvol
\end{equation}
for all $b\in B$, where $E$ is a unital multiplicative linear map and the composition $E\circ\ast$ is an involution.  
Notice that if the quantum integral is a closed graded trace, then it is possible to establish a link with cyclic cohomology (\cite{con}).  
Furthermore, by postulating orthogonality between quantum differential forms of different degrees, we can induce a quantum Riemannian structure on the entire graded space $\Omega^\bullet(B)$; hence, we will no longer use degree superscripts.

Given a quantum space $(B,\cdot,\mathbbm{1},\ast)$ equipped with a quantum integral, the maps
\begin{equation}
\label{a.f2.1}
\begin{aligned}
\langle-|- \rangle_\l := \int_B \langle-,-\rangle_\l\,\dvol,\qquad
\langle-|- \rangle_\r := \int_B \langle-,-\rangle_\r\,\dvol
\end{aligned}
\end{equation}
define inner products on $B$, called the \emph{left and right quantum Hodge inner products}, respectively.  
Of course, if the state $s$ is multiplicative, we obtain pre--$C^\ast$--algebras. However, these induced structures are not necessarily equal to the original one on $B$.

\begin{Remark}
    \label{remariemann}
    Following point~(2) of Definition~\ref{a.2.1} and using non--degenerate $B$--valued sesquilinear maps, it should be clear how to define left/right quantum pseudo--Riemannian metrics on $B$. In this paper, we focus on developing the theory for left/right quantum Riemannian metrics; however, mutatis mutandis, the theory developed can be generalized to the quantum pseudo--Riemannian setting.
\end{Remark}

\begin{Definition}
\label{a.2.4}
Assume that $B$ is endowed with a fixed orientation $\dvol$ and a quantum integral such that $B$ has no boundary. Furthermore, we assume that for a given $\mu\in\Omega^{n-k}(B)$, the left $B$--module map
\[
\begin{aligned}
F_\mu:\Omega^k(B)\longrightarrow B,\qquad
\eta\longmapsto F_\mu(\eta),
\end{aligned}
\]
defined by the relation
\[
\eta\,\mu = F_\mu(\eta)\,\dvol\qquad \mbox{ satisfies }\qquad F_\mu(-)=\langle-,\star_\l^{-1}\mu\rangle_\l
\]
for a unique element $\star_\l^{-1}\mu\in\Omega^k(B)$. Finally, we also assume that this identification induces an antilinear isomorphism  \[
\begin{aligned}
\star_\l:\Omega^k(B)\longrightarrow \Omega^{n-k}(B),\qquad
\mu&\longmapsto \star_\l\mu.
\end{aligned}
\]
We define the left quantum Hodge star operator as the operator $\star_\l$.

In the same way, we define the right quantum Hodge star operator by
\[
\star_\r := \star_\l\circ\ast.
\]
\end{Definition}

The left quantum Hodge operator satisfies several properties. For instance, by construction,
\begin{equation}
\label{a.f2.4}
\eta\,\mu=\langle\eta,\star_\l^{-1}\mu\rangle_\l\,\dvol,
\end{equation}
for all $\eta\in\Omega^k(B)$ and $\mu\in\Omega^{n-k}(B)$, and the operator $\star_\l^{-1}$ is uniquely determined by this relation.

The following result follows straightforwardly, so we omit its proof.

\begin{Theorem}
\label{a.2.5}
We have
\begin{enumerate}
\item For all $\eta$, $\mu$ $\in$ $\Omega^k(B)$, the following equality holds $$\eta\,(\star_\l \mu)=\langle\eta, \mu\rangle_\l\,\dvol.$$  
\item For all $b$ $\in$ $B$ and $\mu$ $\in$ $\Omega^\bullet(B)$, we get $$\star^{-1}_\l(b\,\mu)=(\star^{-1}_\l\mu)\,b^\ast, \qquad \star^{-1}_\l(\mu\, b)=E(b)^\ast (\star_\l\mu),$$ $$ \star_\l(E(b)^{\ast}\,\mu)=(\star_\l\mu)\,b, \qquad \star_\l(\mu\,b)=b^\ast(\star_\l\mu).$$ 
\item We get that $$\star_\l\mathbbm{1}=\dvol, \quad \star_\l\dvol=\mathbbm{1}.$$   
\item For  $\nu$ $\in$ $\Omega^m(B)$, $\eta$ $\in$ $\Omega^l(B)$, $\mu$ $\in$ $\Omega^k(B)$ such that $m+l+k=n$, we obtain  
$$ \langle\eta,\star^{-1}_\l(\nu\,\mu)\rangle_\l=\langle\eta\,\nu,\star^{-1}_\l\mu\rangle_\l.$$
\item The following formula holds $$\displaystyle \langle \eta\,|\,\mu\rangle_\l=\int_B \eta\,(\star_\l \mu)$$ for all $\eta$, $\mu$ $\in$ $\Omega^\bullet(B)$.
\end{enumerate}
\end{Theorem}

One of the main purposes of introducing the Hodge operator is the construction of the codifferential and the Laplace operator.

 \begin{Definition}
\label{a.2.6} 
Let $(B,\cdot,\mathbbm{1},\ast)$ be a quantum space. By using the left quantum Hodge star operator $\star_\l$, we define the left quantum codifferential as the linear operator
\begin{equation*}
\begin{aligned}
d^{\star_\l} := (-1)^{k+1}\,\star_\l^{-1}\circ\, d \circ \star_\l 
:\Omega^{k+1}(B)\longrightarrow \Omega^k(B),\qquad
\mu \longmapsto d^{\star_\l}\mu.
\end{aligned}
\end{equation*}
For $k+1=0$, we set $d^{\star_\l}=0$.  

In the same way, the right quantum codifferential is defined by
\[
d^{\star_\r}:= (-1)^{k+1}\,\star_\r^{-1}\circ \,d\, \circ \star_\r
= \ast \circ d^{\star_\l} \circ \ast .
\]
\end{Definition}

As in the \emph{classical} case, we obtain

\begin{Theorem}
    \label{codif}
    The map $d^{\star_\l}/d^{\star_\r}$ is the formal adjoint operator of $d$ with respect to the left/right quantum Hodge inner product defined in equation (\ref{a.f2.1}).
\end{Theorem}
\begin{proof}
    Let $\mu\in\Omega^{k+1}(B)$ and $\eta\in\Omega^{k}(B)$. Then $\star_\l\mu\in\Omega^{n-k-1}(B)$ and $\eta\,(\star_\l\mu)\in\Omega^{n-1}(B)$. Hence, by Theorem \ref{a.2.5} point (1) and since $B$ is a quantum space without boundary, we get
\begin{eqnarray*}
0 = \int_B  d(\eta\,(\star_\l\mu))
&= &  \int_B  (d\eta)\star_\l\mu + (-1)^k  \int_B  \eta (d\star_\l\mu) 
  \\
 &= & 
\int_B (d\eta)\star_\l\mu
-
 (-1)^{k+1} \int_B  \eta (\star_\l\star^{-1}_\l d\star_\l\mu)
  \\
&= & 
\int_B \langle d\eta,\mu\rangle_\l\,\dvol  -  \int_B  \eta(\star_\l\, d^{\star_\l}\mu)
  \\
&= & 
\int_B \langle d\eta,\mu\rangle_\l\,\dvol  - \int_B  \langle \eta, d^{\star_\l}\mu\rangle_\l\,\dvol
\end{eqnarray*}
and thus 
 $$\langle d\eta\,|\,\mu\rangle_\l =  \langle \eta\,|\,d^{\star_\l}\mu\rangle_\l.$$  Similarly, it can be proven that 
 $$\langle d\eta\,|\,\mu\rangle_\r =  \langle \eta\,|\,d^{\star_\r}\mu\rangle_\r.$$ 
\end{proof}

Finally, we have

\begin{Definition}
\label{a.2.8} 
Given a quantum space $(B,\cdot,\mathbbm{1},\ast)$ and the left quantum Hodge star operator $\star_\l$, the left quantum Laplace--de Rham operator is defined as 
\begin{equation*}
\vartriangle_\l:=d\circ d^{\star_\l}+d^{\star_\l}\circ d: \Omega^\bullet(B)\longrightarrow \Omega^\bullet(B);
\end{equation*}
while the right quantum Laplace--de Rham operator is given by 
\begin{equation*}
\vartriangle_\r:=d\circ d^{\star_\r}+d^{\star_\r}\circ d: \Omega^\bullet(B)\longrightarrow \Omega^\bullet(B).
\end{equation*}
\end{Definition}

It is worth mentioning that both Laplace--de Rham operators are symmetric and non--negative. The reader can go deeper into the study of all these operators in reference \cite{obauchalla}. 

\subsection{Formal Adjointability of Quantum Linear Connections}

Let $\zeta=(P,B,\Delta_P)$ be a quantum principal bundle endowed with a differential calculus such that the space of base forms satisfies Definition \ref{a.2.1}, and suppose that a left quantum Hodge star operator exists. Let $\delta^V$ be a finite--dimensional $H$--corepresentation. By considering the associated left quantum vector bundle together with its canonical Hermitian structure (see Definition~\ref{qvb's} and equation~(\ref{canoleft0})), we define
\begin{equation}
\label{ashcas}
\begin{aligned}
\langle-,-\rangle^\bullet_\l : \Omega^\bullet(B)\otimes_B E^V_\l\times \Omega^\bullet(B)\otimes_B E^V_\l  &\longrightarrow B
\end{aligned}
\end{equation}
by $$\langle \mu_1\otimes_B T_1,\mu_2\otimes_B T_2\rangle^\bullet_\l= \langle\mu_1 (T_1,T_2)_\l,\mu_2\rangle_\l.$$  By using the quantum integral, we can define the map
\begin{equation}
\label{4.f2.23}
\langle-|-\rangle^\bullet_\l :\Omega^\bullet(B)\otimes_B E^V_\l\times \Omega^\bullet(B)\otimes_B E^V_\l \longrightarrow \C
\end{equation}
as $$\displaystyle \langle\mu_1\otimes_B T_1\,|\,\mu_2\otimes_B T_2\rangle^\bullet_\l = \int_B \langle\mu_1\otimes_B T_1\,,\,\mu_2\otimes_B T_2\rangle^\bullet_\l\,\dvol.$$

On the other hand, by considering the associated right quantum vector bundle and its canonical Hermitian structure (see Definition \ref{qvb's} and equation (\ref{canoright0})), we define 
\begin{equation}
\label{ashcas1}
\begin{aligned}
\langle-,-\rangle^\bullet_\r : E^V_\r\otimes_B\Omega^\bullet(B) \times E^V_\r\otimes_B  \Omega^\bullet(B) &\longrightarrow B
\end{aligned}
\end{equation}
by 
\begin{eqnarray*}
    \langle T_1\otimes_B\mu_1 ,T_2\otimes_B\mu_2 \rangle^\bullet_\r= \langle\mu_1 ,(T_1,T_2)_\r\,\mu_2\rangle_\r
    =\langle\mu^\ast_1 ,\mu^\ast_2\,(T_1,T_2)^\ast_{\r}\rangle_\l=
    \langle\mu^\ast_1\,(T_1,T_2)_{\r} ,\mu^\ast_2\,\rangle_\l,
\end{eqnarray*}
and we can also define the map
\begin{equation}
\label{4.f2.27}
\langle-|-\rangle^\bullet_\r :E^V_\r\otimes_B\Omega^\bullet(B) \times E^V_\r\otimes_B\Omega^\bullet(B)  \longrightarrow \C
\end{equation}
by  
$$\langle T_1\otimes_B\mu_1 \mid T_2\otimes_B \mu_2\rangle^\bullet_\r= \int_B \langle T_1\otimes_B\mu_1 ,T_2\otimes_B\mu_2 \rangle^\bullet_\r\,\dvol. $$

It is worth remarking that $\langle-|-\rangle^\bullet_\l$ and $\langle-|-\rangle^\bullet_\r$ are actually inner products for their respective spaces. This is because the Hermitian structures and quantum Riemannian structures are $B$--valued inner products; so in accordance with reference \cite{lance}, their tensor products are positive--definite. Nevertheless, in Appendix A we show an explicit proof of the last statement by doing the calculations.

\begin{Definition}
 \label{4.2.14} 
Consider the exterior covariant derivative $d^{\nabla^{\omega}_{V}}$ associated with the induced quantum linear connection $\nabla^{\omega}_{V}$ (see equation (\ref{3.f10.3})) and the left quantum Hodge star operator $\star_\l$. We define 
$$d^{\nabla^{\omega}_{V}\star} :\Omega^{k+1}(B)\otimes_B E^V_\l\longrightarrow \Omega^k(B)\otimes_B E^V_\l$$
by $$d^{\nabla^{\omega}_{V}\star}:=(-1)^{k+1} ((\star^{-1}_\l\circ\ast)\otimes_B \id_{E^V_\l}) \circ\; d^{\nabla^{\omega}_{V}}\circ ((\ast\circ\star_\l)\otimes_B \id_{E^V_\l}).$$ For $k+1=0$ we take $d^{\nabla^{\omega}_{V}\star}=0$ and for $k+1=1$ we are going to write $\nabla^{\omega\,\star}_{V}:=d^{\nabla^{\omega}_{V}\star}.$

In the same way,  consider the exterior covariant derivative $d^{\widehat{\nabla}^{\omega}_{V}}$ associated with the induced quantum linear connection $\widehat{\nabla}^{\omega}_{V}$ (see equation (\ref{3.f10.6})) and the right quantum Hodge star operator $\star_\r$. We define
$$d^{\widehat{\nabla}^{\omega}_{V}\star} :E^V_\r\otimes_B \Omega^{k+1}(B)\longrightarrow E^V_\r\otimes_B \Omega^k(B)$$ by
\begin{equation*}
d^{\widehat{\nabla}^{\omega}_{V}\star}:=(-1)^{k+1} (\id_{E^V_\r}\otimes_B \star^{-1}_\r) \circ\; d^{\widehat{\nabla}^{\omega}_V}\circ (\id_{E^V_\r}\otimes_B \star_\r).
\end{equation*}
 For $k+1=0$ we take $d^{\widehat{\nabla}^{\omega}_{\alpha}\star}=0$ and for $k+1=1$ we are going to  write $\widehat{\nabla}^{\omega\,\star}_{\alpha}:=d^{\widehat{\nabla}^{\omega}_{\alpha}\star}$. 
\end{Definition}

Now, we get

\begin{Theorem}
\label{4.2.15}
Let $\omega$ be a qpc and let $\delta^V$ $\in$ $\Obj(\Rep_{H})$.  Then, the operators $d^{\nabla^{\omega}_{V}\star}$, $d^{\widehat{\nabla}^{\omega}_{V}\star}$  are the formal adjoint operators of $d^{\nabla^{\omega}_{V}}$, $d^{\widehat{\nabla}^{\omega}_{V}}$ with respect to the inner products of equations (\ref{4.f2.23}), (\ref{4.f2.27}), respectively.
\end{Theorem}

\begin{proof}
This proof consists of a large calculation. In fact, by definition $$\nabla^{\omega}_{V}(T_2)=\displaystyle \sum_i \mu^{_{D^{\omega}(T_2)}}_i\otimes_B T^\l_i \;\in\; \Omega^1(B)\otimes_B E^V_\l,$$ and one obtains $$d^{\nabla^{\omega}_{V}\star}(\mu_2\otimes_B T_2)=  d^{\star_\l}\mu_2\otimes_B x_2 +(-1)^{k+1}\sum_i \star^{-1}_\l(\mu^{_{D^{\omega}(T_2)}\,\ast}_i(\star_\l\mu_2))\otimes_B T^\l_i $$  for all $\mu_2$ $\in$ $\Omega^{k+1}(B)$, $T_2$ $\in$ $E^V_\l$. Now for $\mu_1$ $\in$ $\Omega^k(B)$ and $T_1$ $\in$ $E^V_\l$, we get

\begin{eqnarray*}
\langle d\mu_1\otimes_B T_1\mid \mu_2\otimes_B T_2\rangle^\bullet_\l&=&  \langle (d\mu_1)\, (T_1,T_2)_\l\mid\mu_2 \rangle_\l \\
 &= & 
\langle d(\mu_1 (T_1,T_2)_\l)\mid\mu_2 \rangle_\l+ (-1)^{k+1}\langle \mu_1 \,d ( T_1,T_2)_\l\mid \mu_2\rangle_\l
  \\
&= & 
\langle \mu_1 (T_1,T_2)_\l\mid d^{\star_\l}\mu_2 \rangle_\l +  (-1)^{k+1}\langle \mu_1 ( \nabla^{\omega}_{V}(T_1),T_2)^\bullet_\l\mid \mu_2\rangle_\l
  \\
&+ & 
(-1)^{k+1}\,\langle \mu_1 (T_1,\nabla^{\omega}_{V}(T_2))^\bullet_\l\mid \mu_2\rangle_\l,
\end{eqnarray*}

\noindent since  $\nabla^{\omega}_V$ is Hermitian (see Theorem \ref{fgs}). By definition, we have $$\langle \mu_1 (T_1,T_2)_\l\mid d^{\star_\l}\mu_2 \rangle_\l=\langle \mu_1 \otimes T_1 \mid d^{\star_\l}\mu_2\otimes T_2 \rangle^\bullet_\l$$ and $$ \langle \mu_1 ( \nabla^{\omega}_{V}( T_1),T_2)^\bullet_\l\mid\mu_2\rangle_\l=\langle \mu_1 \,\nabla^{\omega}_{V}(T_1)\mid\mu_2\otimes_B T_2\rangle^\bullet_\l.$$ Furthermore,   $$\langle \mu_1 (  T_1,\nabla^{\omega}_{V}(T_2))^\bullet_\l\mid\mu_2\rangle_\l= \sum_i \langle\mu_1\otimes_B T_1\mid \star^{-1}_\l(\mu^{_{D^{\omega}(T_2)}\,\ast}_i(\star_\l\mu_2))\otimes_B T^\l_i\rangle^\bullet_\l.$$ Indeed,  $$\langle \mu_1 ( T_1,\nabla^{\omega}_{V}(T_2))^\bullet_\l\mid \mu_2\rangle_\l=\sum_i\langle \mu_1 ( T_1,T^\l_i)_\l \,\mu^{_{D^{\omega}(T_2)}\,\ast}_i\mid\mu_2\rangle_\l;$$ while by Theorem \ref{a.2.5} point $4$ we obtain
\begin{eqnarray*}
    \sum_i \langle\mu_1\otimes_B T_1\mid \star^{-1}_\l(\mu^{_{D^{\omega}(T_2)}\,\ast}_i(\star_\l\mu_2))\otimes_B T^\l_i\rangle^\bullet_\l&=&\sum_i \langle\mu_1 (T_1,T^\l_i)_\l\mid \star^{-1}_\l(\mu^{_{D^{\omega}(T_2)}\,\ast}_i(\star_\l\mu_2))\rangle_\l
    \\
    &=&
    \sum_i\langle \mu_1 ( T_1,T^\l_i)_\l\, \mu^{D^{\omega}(T_2)\,\ast}_i\mid\star^{-1}_\l\star_\l\mu_2\rangle_\l
    \\
    &=&
    \sum_i\langle \mu_1 ( T_1,T^\l_i)_\l\, \mu^{D^{\omega}(T_2)\,\ast}_i\mid\mu_2\rangle_\l;
\end{eqnarray*}
thus the last assertion holds. Now, taking into account these equalities we find 

\begin{eqnarray*}
    \langle d^{\nabla^{\omega}_{V}}(\mu_1\otimes_B T_1)\,|\,\mu_2\otimes_B T_2\rangle^\bullet_\l&=&  \langle d\mu_1\otimes_B T_1\,|\,\mu_2\otimes_B T_2\rangle^\bullet_\l+  
(-1)^k\langle \mu_1 \,\nabla^{\omega}_{V}(T_1)\,|\,\mu_2\otimes_B T_2\rangle^\bullet_\l
\\
&=&
\langle \mu_1 (T_1,T_2)_\l),d^{\star_\l}\mu_2 \rangle_\l+ 
(-1)^{k+1}  \langle \mu_1 (\nabla^{\omega}_{V}(T_1),T_2)^\bullet_\l\mid \mu_2\rangle_\l
\\
&+&
(-1)^{k+1} \langle \mu_1 (T_1,\nabla^{\omega}_{V}(T_2))^\bullet_\l\mid\mu_2\rangle_\l
\\
&+&
(-1)^k\langle \mu_1 \,\nabla^{\omega}_{V}(T_1)\mid\mu_2\otimes_B T_1\rangle^\bullet_\l
\\
&=&
\langle \mu_1 (T_1,T_2)_\l),d^{\star_\l}\mu_2 \rangle_\l
+
(-1)^{k+1} \langle \mu_1 ( T_1,\nabla^{\omega}_{V}(T_2))^\bullet_\l\mid\mu_2\rangle_\l
\\
&=&
\langle \mu_1 \otimes T_1 \mid d^{\star_\l}\mu_2\otimes T_2 \rangle^\bullet_\l
\\
&+&(-1)^{k+1}\sum_i \langle\mu_1\otimes_B T_1\mid \star^{-1}_\l(\mu^{_{D^{\omega}(T_2)}\,\ast}_i(\star_\l\mu_2))\otimes_B T^\l_i\rangle^\bullet_\l
\\
&=&
\langle \mu_1\otimes_B x_1\mid d^{\nabla^{\omega}_{V}\star}(\mu_2\otimes_B T_2)\rangle^\bullet_\l.
\end{eqnarray*}

\noindent By linearity, we conclude that $d^{\nabla^{\omega}_{V}\star}$ is the formal adjoint operator of $d^{\nabla^{\omega}_{V}}$. The proof of the corresponding statement for $d^{\widehat{\nabla}^{\omega}_{V}\star}$ is completely analogous.
\end{proof}

It is worth mentioning that a similar result appears in reference \cite{obauchalla2} in the context of the Chern connection for quantum homogeneous spaces. Of course, there is a natural generalization of the left/right quantum Laplace--de Rham operator to left/right quantum vector bundle--valued forms by means of
\begin{equation}
\label{4.f2.25}
\square^{\omega}_{V}:=d^{\nabla^{\omega}_{V}}\circ d^{\nabla^{\omega}_{V}\star}+d^{\nabla^{\omega}_{V}\star}\circ d^{\nabla^{\omega}_{V}},\qquad \widehat{\square}^{\omega}_{V}:=d^{\widehat{\nabla}^{\omega}_{V}}\circ d^{\widehat{\nabla}^{\omega}_{V}\star}+d^{\widehat{\nabla}^{\omega}_{V}\star}\circ d^{\widehat{\nabla}^{\omega}_{V}}.
\end{equation}
Like in the previous subsection, these Laplacians are symmetric and non--negative.

\section{Yang–Mills Scalar Matter Fields in
Non–Commutative Geometry}

By using the theory developed in the previous sections, we can now achieve our goal: a \emph{non--commutative geometrical} version of the \emph{classical} Yang--Mills theory and scalar matter fields formulated in terms of principal bundles. We begin with the theory of \emph{free (non--interacting) gauge boson fields}.

\subsection{Yang--Mills Theory}

The next definition closely follows the \emph{classical} axiomatic formulation \cite{gtvp}.

\begin{Definition}
\label{6.1.1}
A non--commutative geometrical Yang--Mills model consists of
\begin{enumerate}
\item A quantum $\G$--bundle $\zeta=(P,B,\Delta_P)$, where $B$ is stable under holomorphic calculus.
\item A bicovariant $\ast$--FODC $(\Gamma,d)$ over $H$ such that $\mathfrak{qg}^\#$ is a finite--dimensional $\C$--vector space.
\item A differential calculus on $\zeta$ induced by the previous bicovariant $\ast$--FODC, for which $\Omega^\bullet(B)$ satisfies Definition \ref{a.2.1} and for which a left quantum Hodge star operator exists.
\item For every qpc $\omega$ of $\zeta$, the operator $$d^{DS^\omega}:=\Upsilon_{\mathfrak{qg}^\#}\circ DS^\omega  \circ \Upsilon^{-1}_{\mathfrak{qg}^\#}\qquad \mbox{ with }\qquad DS^\omega=D^\omega-S^\omega$$ is adjointable or formally adjointable in (see equation (\ref{realtensor})) $$ \Upsilon_{\mathfrak{qg}^\#}(\Mor(\ad,\Delta_\H)^\dagger)$$  with respect to $\langle-|-\rangle^\bullet_\l$; while the operator $$ d^{\widehat{DS^\omega}}:=\widehat{\Upsilon}_{\mathfrak{qg}^\#}\circ \widehat{DS^\omega} \circ \widehat{\Upsilon}^{-1}_{\mathfrak{qg}^\#} \qquad \mbox{ with }\qquad \widehat{DS^\omega}=\wedge \circ DS^\omega\circ \wedge$$ is adjointable or formally adjointable in $$ \widehat{\Upsilon}_{\mathfrak{qg}^\#}(\Mor(\ad,\Delta_\Hor)^\dagger)$$ with respect to $\langle-|-\rangle^\bullet_\r$. Of course, we have considered that in the adjointable--case, both inner products induce Hilbert space structures. 

The adjoint (or formally adjoint) operators will be denoted by $$d^{DS^\omega\star}  \qquad \mbox{ and }\qquad d^{\widehat{DS^\omega}\star} ,$$ respectively.
\end{enumerate}
\end{Definition} 

It is worth mentioning that the fourth point of the previous definition is well--defined. In fact, by the second point, we obtain $\ad$ $\in$ $\Obj(\Rep_{H})$ and hence, we can consider their associated qvb's $E^{\mathfrak{qg}^\#}_\l$, $E^{\mathfrak{qg}^\#}_\r$. Consequently, the operators $d^{\nabla^\omega_{\mathfrak{qg}^\#}}$, $d^{\widehat{\nabla}^\omega_{\mathfrak{qg}^\#}}$ make sense. On the other hand, since for every qpc $\omega$ we have $$S^\omega:\Mor(\ad,\Delta_\Hor)\longrightarrow \Mor(\ad,\Delta_\Hor),\qquad \widehat{S}^\omega:\Mor(\ad,\Delta_\Hor)\longrightarrow \Mor(\ad,\Delta_\Hor),$$ the operators 
\begin{equation}
    \label{ec.soperators}
    d^{S^{\omega}}:=\Upsilon_{\mathfrak{qg}^\#}\circ S^{\omega} \circ \Upsilon^{-1}_{\mathfrak{qg}^\#},\qquad d^{\widehat{S}^{\omega}}=\widehat{\Upsilon}_{\mathfrak{qg}^\#}\circ \widehat{S}^{\omega}\circ \widehat{\Upsilon}^{-1}_{\mathfrak{qg}^\#},
\end{equation}
\begin{equation}
    \label{ec.important}
    d^{DS^\omega}=d^{\nabla^\omega_{\mathfrak{qg}^\#}}-d^{S^{\omega}},\qquad d^{\widehat{DS^\omega}}=d^{\widehat{\nabla}^\omega_{\mathfrak{qg}^\#}}-d^{\widehat{S}^{\omega}}
\end{equation}
make sense as well. 
\begin{Remark}
    \label{Sadjoint}
    By Theorem \ref{4.2.15}, the fourth point of Definition \ref{6.1.1} is satisfied if and only if $d^{S^{\omega}}$ and $d^{\widehat{S}^{\omega}}$ are adjointable (or formally adjointable). These  adjoint (or formally adjoint) operators will be denoted by
\begin{equation}
    \label{adjointsoperator}
    d^{S^{\omega}\star},\qquad d^{\widehat{S}^{\omega}\star}.
\end{equation}
\end{Remark}

As the reader will notice in Theorem \ref{6.1.5}, there is no need to require adjointability of the operators on the entire spaces $\Upsilon_{\mathfrak{qg}^\#}(\Mor(\ad,\Delta_\Hor))$ and $\widehat{\Upsilon}_{\mathfrak{qg}^\#}(\Mor(\ad,\Delta_\Hor))$; it is sufficient to require adjointability on the $\R$--linear subspaces $ \Upsilon_{\mathfrak{qg}^\#}(\Mor(\ad,\Delta_\H)^\dagger)$ and $\widehat{\Upsilon}_{\mathfrak{qg}^\#}(\Mor(\ad,\Delta_\Hor)^\dagger)$.

The first two points of Definition \ref{6.1.1} are straightforward to satisfy. However, the third point as well as the adjointability (or formal adjointability) condition, is far less obvious. In the next section we will present two classes of qpb's for which non--commutative geometrical Yang--Mills models always exits. 

Consider a non--commutative geometrical Yang--Mills model and let $\omega$ be a qpc. We know that $R^\omega$ $\in$ $\Mor(\ad,\Delta_\Hor)$  and therefore, we define (see equations (\ref{4.f2.23}), (\ref{4.f2.27}))
\begin{equation}
\label{curnormdef}
    ||R^\omega||_\l^\bullet:=||\Upsilon_{\mathfrak{qg}^\#}(R^\omega)||_\l^\bullet,\qquad  ||R^\omega||_\r^\bullet:=||\widehat{\Upsilon}_{\mathfrak{qg}^\#}(R^\omega)||_\r^\bullet.
\end{equation}
It is worth mentioning that by Remark \ref{remarkinverse}, we obtain that $||R^\omega||_\l^\bullet$ and $||R^\omega||_\r^\bullet$ do not depend on the choice of the generators $\{ T^\l_k\},$ $\{T^\r_k\}$.

In differential geometry, the Yang--Mills action is defined by measuring the squared norm of the curvature of a principal connection (see reference \cite{gtvp} or Appendix C for a brief summary). This observation motivates the following definition.

\begin{Definition}
\label{6.1.2}
Considering Definition \ref{6.1.1}, we define the non--commutative geometrical Yang--Mills action as the association 
\begin{equation*}
\begin{aligned}
\qS_\YM:\mathfrak{qpc}(\zeta) \longrightarrow \R,\qquad
\omega &\longmapsto -\dfrac{1}{4}\left(||R^\omega||_\l^{\bullet\,2}+||R^\omega||_\r^{\bullet\,2}\right).
\end{aligned}
\end{equation*}
\end{Definition}

Now, let us consider the quantum gauge group (see Section 4 of reference \cite{sald2})
\begin{equation*}
        \begin{aligned}
            \qGG=&\{\F:\Omega^\bullet(P)\longrightarrow \Omega^\bullet(P)\mid \F  \mbox{ is a } \mbox{graded left } \Omega^\bullet(B)-\mbox{module isomorphism}  \\ &\mbox{ such that } \F(\mathbbm{1})=\mathbbm{1},\; \Delta_{\Omega^\bullet(P)}\circ \F=(\F\otimes \id_{\Gamma^\wedge})\circ \Delta_{\Omega^\bullet(P)}\mbox{ and } \F(\Im(\omega)^\ast)=\F(\Im(\omega))^\ast \\ 
     & \mbox{ for all }\; \omega \in \mathfrak{qpc}(\zeta) \}.
        \end{aligned}
    \end{equation*}
Elements of $\qGG$ are called \emph{quantum gauge transformations} (abbreviated ``qgt"). It is worth mentioning that in reference \cite{sald2}, the group $\qGG$ is defined without the condition $\F(\Im(\omega)^\ast)=\F(\Im(\omega))^\ast$ for all $\omega \in \mathfrak{qpc}(\zeta)$. In this paper, we have added this condition in the definition of $\qGG$ because in this paper, qpc's $\omega$ satisfy $\omega(\theta^\ast)=\omega(\theta)^\ast$.

The group $\qGG$ is a  generalization at the level of differential calculus of the quantum gauge group presented in reference \cite{br1}. In particular, it is isomorphic to a subgroup of the group of all convolution--invertible maps $$\f:\Gamma^\wedge\longrightarrow \Omega^\bullet(P)$$ such that $\f(\mathbbm{1})=\mathbbm{1}$ and $(\f\otimes \id_{\Gamma^\wedge})\circ \Ad=\Delta_{\Omega^\bullet(P)}\circ \f$, where the map $\Ad$ is given in equation (\ref{2.f11}), as the reader can verify in Proposition 4.2 of reference \cite{appendix}. 

In Durdevich's framework, the action of $\qGG$ on $\mathfrak{qpc}(\zeta)$ given by $$ \omega \longmapsto \f\ast \omega \ast \f^{-1}+ \f\ast (d\circ \f^{-1})$$ is not well--defined. Furthermore, even if we extend the domain of $\omega$ to $G$, the induced action on the curvature $$\f\ast R^\omega \ast \f^{-1} $$ remains ill-defined. However, in accordance with Theorem 4.7 points 1 and 2 of reference \cite{sald2}, $\qGG$ has a well--defined group action on  $\mathfrak{qpc}(\zeta)$ by $$\F^\circledast \omega:=\F\circ \omega,$$ and this formula induces a well--defined action on the curvature (see Proposition 4.8 of reference \cite{sald2}). The reader is encouraged to consult the reference \cite{sald2} for more details. Notice that $\F^\circledast \omega$ is only the {\it dualization} of the action of the gauge group on principal connections via the pull--back in the {\it classical} case.

In general, the quantum gauge group is quite large, so it is very difficult to work with it in full generality.  For this reason, it is common to restrict attention to \emph{ad hoc} subgroups in each situation. For example, in reference \cite{landi04}, the authors study when it is \emph{natural} to work only with the subgroup of $\qGG$ consisting of all degree--zero qgt's that are algebra morphisms. 

For the theory we wish to develop, the quantum gauge group is so large that in general, not every quantum gauge transformation leaves the action $\qS_\YM$ invariant. Therefore, we need to restrict our attention to the following subgroup.

\begin{Definition}
\label{6.1.3}
We define the quantum gauge group of the Yang--Mills model as the subgroup of $\qGG$ consisting of all symmetries of $\qS_\YM$, i.e.,  $$\qGG_\YM:=\{\F \in \qGG\mid \qS_\YM(\omega)=\qS_\YM(\F^{\circledast} \omega)\;\mbox{ for all }\; \omega \in \mathfrak{qpc}(\zeta) \}.$$  
\end{Definition}

In the \emph{classical} case, every principal bundle automorphism, also called a gauge transformation, is a symmetry of the Yang--Mills action. Consequently, in the \emph{dualization} of the \emph{classical} case, we have $$\qGG_\YM\cong \mathfrak{GG}$$ (in this case, $\F$ is the pull--back a principal bundle automorphism acting on $\C$--valued differential forms), where $\mathfrak{GG}$ is the \emph{classical} gauge group. In this sense, Definition~\ref{6.1.3} is a proper generalization of the \emph{classical} gauge group, addressing the fact that in non--commutative geometry, in general, not every element of $\qGG$ is a symmetry of $\qS_\YM$.

\begin{Proposition}
    \label{gaugeym}
    Let $\F:\Omega^\bullet(P)\longrightarrow \Omega^\bullet(P)$ be a quantum gauge transformation such that $\F$ is a graded differential $\ast$--algebra morphism. Then $\F$ $\in$ $\qGG_\YM$.
\end{Proposition}

\begin{proof}
   Let $\omega$ $\in$ $\mathfrak{qpc}(\zeta)$. By Proposition 4.8 of reference \cite{sald2}, we know that $R^{\F^\circledast \omega}=\F \circ R^{\omega}$. 
   
   On the other hand, by equation (\ref{3.f7}) we have $$R^\omega=\sum_k \mu^{_{R^\omega}}_kT^\l_k,$$ where  $\mu^{_{R^{\omega}}}_k=\displaystyle\sum_{i}R^\omega(\theta_i)\,x^{\mathfrak{qg}^\#\,\ast}_{ki}$. Here, $\{\theta_i \}$ is an orthonormal basis (with respect to the inner product that makes $\ad$ a unitary $H$--corepresentation) and $T^\l_k(\theta_i)=x^{\mathfrak{qg}^\#}_{ki}$, where $\{ T^\l_k\}$ is the set of left $B$--generators for $\ad$. In this way, 
   $$R^{\F^\circledast \omega}=\F\circ R^\omega= \sum_k \mu^{_{R^\omega}}_k\, \F \circ T^\l_k;$$ so, for all $\theta$ $\in$ $\mathfrak{qg}^\#$ we obtain $R^{\F^\circledast \omega}(\theta)=\displaystyle \sum_k \mu^{_{R^\omega}}_k\,\F(T^\l_k(\theta))$. 

   By equation (\ref{3.f7}) we get $$R^{\F^\circledast \omega}=\sum_k \mu^{_{R^{\F^\circledast \omega}}}_kT^\l_k,$$ where 
   \begin{eqnarray*}
  \mu^{_{R^{\F^\circledast \omega}}}_k=\displaystyle\sum_{i}R^{\F^\circledast \omega}(\theta_i)\,x^{\mathfrak{qg}^\#\,\ast}_{ki} = \sum_{i,j} \mu^{_{R^\omega}}_j \F(x^{\mathfrak{qg}^\#}_{ji}) \,x^{\mathfrak{qg}^\#\,\ast}_{ki}
  &= &
 \sum_{j} \mu^{_{R^\omega}}_j \left(\sum_{i}\F(x^{\mathfrak{qg}^\#}_{ji}) \,x^{\mathfrak{qg}^\#\,\ast}_{ki}\right)
  \\
  &= &
  \sum_{j} \mu^{_{R^\omega}}_j \,b_{jk}.
\end{eqnarray*}
 By equation (\ref{generators}), it follows that  $b_{jk}=\displaystyle\sum_{i}\F(x^{\mathfrak{qg}^\#}_{ji}) \,x^{\mathfrak{qg}^\#\,\ast}_{ki}$ $\in$ $B$ and  $$\displaystyle \sum_{k}b_{lk}\,b^\ast_{jk}=\displaystyle \sum_{i}\F(x^{\mathfrak{qg}^\#}_{li}\,x^{\mathfrak{qg}^\#\,\ast}_{ji})=\sum_{i}x^{\mathfrak{qg}^\#}_{li}\,x^{\mathfrak{qg}^\#\,\ast}_{ji},$$  where in the last equality we have used that $\displaystyle \sum_{i}x^{\mathfrak{qg}^\#}_{li}\,x^{\mathfrak{qg}^\#\,\ast}_{ji}$ $\in$ $B$ and the fact that $\F$ is a left $\Omega^\bullet(B)$--module morphism.
 
 Following the proof of Proposition \ref{a.1} in Appendix A, we know that
   $$||R^{\omega}||_\l^{\bullet\,2}=\sum_k ||\mu^{_{R^\omega}}_k ||^{\bullet\,2}_\l,\qquad ||R^{_{\F^\circledast \omega}}||_\l^{\bullet\,2}=\sum_k ||\mu^{_{R^{\F^\circledast \omega}}}_k||^{\bullet\,2}_\l.$$ Thus, by   equation (\ref{generators})
\begin{eqnarray*}
  ||R^{_{\F^\circledast \omega}}||_\l^{\bullet\,2}= \sum_k ||\mu^{_{R^{_{\F^\circledast \omega}}}}_k ||^{\bullet\,2}_\l = \sum_k \langle \mu^{_{R^{_{\F^\circledast \omega}}}}_k \mid \mu^{_{R^{_{\F^\circledast \omega}}}}_k\rangle^\bullet_\l
  &= &
 \sum_{k,j,l} \langle \mu^{_{R^\omega}}_j\,b_{jk}\mid \mu^{_{R^\omega}}_l b_{lk}\rangle^\bullet_\l
 \\
  &= &
 \sum_{k,j,l} \langle \mu^{_{R^\omega}}_j\mid \mu^{_{R^\omega}}_l b_{lk}\,b^\ast_{jk}\rangle^\bullet_\l
 \\
  &= &
  \sum_{i,j,l} \langle \mu^{_{R^\omega}}_j\mid  \mu^{_{R^\omega}}_l \,x^{\mathfrak{qg}^\#}_{li}\,x^{\mathfrak{qg}^\#\,\ast}_{ji}\rangle^\bullet_\l
  \\
  &= & 
   \sum_{i,j,l,s} \langle \mu^{_{R^\omega}}_j\mid 
    R^\omega(\theta_s)\, x^{\mathfrak{qg}^\#\,\ast}_{ls} \,x^{\mathfrak{qg}^\#}_{li}\,x^{\mathfrak{qg}^\#\,\ast}_{ji}\rangle^\bullet_\l
   \\
  &= &
  \sum_{i,j,s}\langle \mu^{_{R^\omega}}_j\mid R^\omega(\theta_s)\, \delta_{si}\,x^{\mathfrak{qg}^\#\,\ast}_{ji}\rangle^\bullet_\l
  \\
  &= &
  \sum_{i,j}\langle \mu^{_{R^\omega}}_j\mid R^\omega(\theta_i)\,x^{\mathfrak{qg}^\#\,\ast}_{ji}\rangle^\bullet_\l
   \\
  &= &
  \sum_{j}\langle \mu^{_{R^\omega}}_j\mid \mu^{_{R^\omega}}_j\rangle^\bullet_\l 
  \\
  &= &
  \sum_{j} ||\mu^{_{R^\omega}}_j ||^{\bullet\,2}_\l=||R^\omega||^{\bullet\,2}_\l.
\end{eqnarray*}

A completely analogous calculation using equation (\ref{3.f7.5}) instead of equation (\ref{3.f7}) shows that $ ||R^{_{\F^\circledast \omega}}||_\r^{\bullet\,2}=||R^\omega||^{\bullet\,2}_\r$. Therefore  $\F$ $\in$ $\qGG_\YM$.
\end{proof}

Clearly, in general, there can be elements of $\qGG_\YM$ that are not graded differential $\ast$--algebra morphisms. A concrete example is provided by Theorem~3.12 and Remark 4.10 of reference \cite{sald6}. 

Our next step is to derive the \emph{field equation} for $\omega\in\mathfrak{qpc}(\zeta)$ by postulating that the first variation of the non--commutative geometrical Yang--Mills action vanishes, in complete agreement with the \emph{classical} case.

\begin{Definition}
\label{6.1.4}
 A stationary point of $\qS_\YM$ is an element $\omega$ $\in$ $\mathfrak{qpc}(\zeta)$ such that for any $\lambda$ $\in$ $\overrightarrow{\mathfrak{qpc}(\zeta)}$ (see equation (\ref{2.f24.2}))  $$\left.\dfrac{d}{d t}\right|_{t=0}\qS_\YM(\omega + t\,\lambda)=0$$ for $t$ $\in$ $\R$. Stationary points are also called Yang--Mills qpc's or non--commutative geometrical Yang--Mills fields. In terms of a traditional physical interpretation, they should be considered as {\it free--interaction gauge boson fields possessing the symmetry} $\qGG_\YM$.
\end{Definition}

It is worth recalling that $\overrightarrow{\mathfrak{qpc}(\zeta)}$ is a $\R$--vector space, so the expression $t\,\lambda$ makes sense for $t$ $\in$ $\R$. Moreover, $\qS_\YM(\omega+t\,\lambda)$ depends polynomially on $t$, so we interpret the derivative with respect to $t$ of $\qS_\YM
$ as a formal derivative in the obvious way.

\begin{Theorem}
\label{6.1.5}
 A qpc $\omega$ is a Yang--Mills qpc if and only if
\begin{equation}
\label{6.f1.1}
\begin{aligned}
   \mathrm{Re}\left( \langle d^{DS^\omega\star}(R^{\omega})\mid  \Upsilon_{\mathfrak{qg}^\#}(\lambda)\rangle^\bullet_\l\right.+ \left. \langle d^{\widehat{DS^\omega}\star}(R^{\omega})\mid  \widehat{\Upsilon}_{\mathfrak{qg}^\#}(\lambda)\rangle^\bullet_\r\right)=0
\end{aligned}
\end{equation} 
for all $\lambda$ $\in$ $\overrightarrow{\mathfrak{qpc}(\zeta)}$, where $\mathrm{Re}(z)$ denotes the real part of $z$ $\in$ $\C$ and we have defined $$d^{DS^\omega\star}(R^{\omega}):=d^{DS^\omega\star}( \Upsilon_{\mathfrak{qg}^\#}(R^{\omega})) \qquad \mbox{ and }\qquad d^{\widehat{DS^\omega}\star} R^{\omega}:=d^{\widehat{DS^\omega}\star} (\widehat{\Upsilon}_{\mathfrak{qg}^\#}(R^{\omega})).$$
\end{Theorem}

\begin{proof}
Let $\lambda$ $\in$ $\overrightarrow{\mathfrak{qpc}(\zeta)}$ and $t$ $\in$ $\R$. Then, by equation (\ref{2.f28}) we have
\begin{eqnarray*}
R^{\omega+ t\,\lambda} = d(\omega+t\,\lambda)-\langle \omega+t\,\lambda,\omega+t\,\lambda\rangle
  &= &
 d\omega+t\,d\lambda-\langle\omega,\omega\rangle-t\,\langle \omega,\lambda\rangle-t\,\langle \lambda,\omega\rangle
- 
 t^2\,\langle\lambda,\lambda\rangle
 \\
  &= &
  R^\omega+ t\,(d\lambda-\langle \omega,\lambda\rangle-\langle\lambda,\omega\rangle)-t^2\,\langle\lambda,\lambda\rangle.
\end{eqnarray*}
According to equation (\ref{twcoder}) and remembering that $\Im(\lambda)$ $\subseteq$ $\Hor^1P$, we get
$$DS^\omega (\lambda)=d\lambda-\langle \omega,\lambda\rangle-\langle\lambda,\omega\rangle\quad \Longrightarrow \quad R^{\omega+ t\,\lambda}= R^\omega+t\,DS^\omega(\lambda)-t^2\langle\lambda,\lambda\rangle.$$  It follows that
\begin{eqnarray*}
 ||R^{\omega+t\lambda}||_\l^{\bullet \,2} =  ||\Upsilon_{\mathfrak{qg}^\#}(R^{\omega+t\lambda})||_\l^{\bullet\,2}
  &= &
  \langle\Upsilon_{\mathfrak{qg}^\#}(R^{\omega+t\lambda})\mid \Upsilon_{\mathfrak{qg}^\#}(R^{\omega+t\lambda}) \rangle_\l^\bullet
  \\
  &= &
  \langle \Upsilon_{\mathfrak{qg}^\#}(R^{\omega})\mid \Upsilon_{\mathfrak{qg}^\#}(R^{\omega})\rangle^\bullet_\l
  \\
  &+ &
  t\,\langle \Upsilon_{\mathfrak{qg}^\#}(R^{\omega})\mid \Upsilon_{\mathfrak{qg}^\#}(DS^\omega(\lambda))\rangle^\bullet_\l
  \\
  &+ &
  t^{2}\, \langle \Upsilon_{\mathfrak{qg}^\#}(R^{\omega})\mid \Upsilon_{\mathfrak{qg}^\#}(\langle\lambda,\lambda\rangle)\rangle^\bullet_\l
  \\
  &+ &
  t \,\langle \Upsilon_{\mathfrak{qg}^\#}(DS^\omega(\lambda))\mid \Upsilon_{\mathfrak{qg}^\#}(R^{\omega})\rangle^\bullet_\l
  \\
   &+ &
   \\
  & \vdots & 
  \\
  &+ & t^{4}\,\langle \Upsilon_{\mathfrak{qg}^\#}(\langle\lambda,\lambda\rangle)\mid \Upsilon_{\mathfrak{qg}^\#}(\langle\lambda,\lambda\rangle)\rangle^\bullet_\l.
\end{eqnarray*}

By taking the derivative with respect to $t$ in the last polynomical expression and evaluate it in $t=0$, we conclude
\begin{eqnarray*}
    \left.\dfrac{d}{d t}\right|_{t=0} ||R^{\omega+t\,\lambda}||_\l^{\bullet \,2} &=& \langle \Upsilon_{\mathfrak{qg}^\#}(R^{\omega})\mid \Upsilon_{\mathfrak{qg}^\#}(DS^\omega(\lambda))\rangle^\bullet_\l
     +
    \langle \Upsilon_{\mathfrak{qg}^\#}(DS^\omega(\lambda))\mid \Upsilon_{\mathfrak{qg}^\#}(R^{\omega})\rangle^\bullet_\l
    \\
    &=&
    2\, \mathrm{Re}\left( \langle \Upsilon_{\mathfrak{qg}^\#}(R^{\omega})\mid \Upsilon_{\mathfrak{qg}^\#}(DS^\omega (\lambda))\rangle^\bullet_\l \right).
\end{eqnarray*}

 Notice $$\Upsilon_{\mathfrak{qg}^\#}(DS^\omega (\lambda))=d^{DS^\omega}(\Upsilon_{\mathfrak{qg}^\#}(\lambda)).$$ Since $R^\omega$, $\lambda$ $\in$ $\Mor(\ad,\Delta_\Hor)^\dagger$ and $d^{DS^\omega}$ is adjointable (or formally adjointable) in the space $\Upsilon_{\mathfrak{qg}^\#}(\Mor(\ad,\Delta_\Hor)^\dagger)$, we obtain
\begin{eqnarray*}
\left.\dfrac{d}{d t}\right|_{t=0} ||R^{\omega+t\lambda}||_\l^{\bullet \,2} &=&  2\, \mathrm{Re}\left( \langle \Upsilon_{\mathfrak{qg}^\#}(R^{\omega})\mid \Upsilon_{\mathfrak{qg}^\#}(DS^\omega(\lambda))\rangle^\bullet_\l \right)
  \\
  &= &
  2\, \mathrm{Re}\left( \langle \Upsilon_{\mathfrak{qg}^\#}(R^{\omega})\mid  d^{DS^\omega}(\Upsilon_{\mathfrak{qg}^\#}(\lambda))\rangle^\bullet_\l \right)
  \\
  &= &
  2\, \mathrm{Re}\left( \langle d^{DS^\omega\star}(R^{\omega})\mid    \Upsilon_{\mathfrak{qg}^\#}(\lambda)\rangle^\bullet_\l\right).
\end{eqnarray*}

On the other hand,  by equation (\ref{twcoder1}) we have 
$$\widehat{DS^\omega} (\lambda)=d\lambda-\langle \omega,\lambda\rangle-\langle\lambda,\omega\rangle\quad\Longrightarrow \quad R^{\omega+ t\,\lambda}= R^\omega+t\,\widehat{DS^\omega}(\lambda)-t^2\langle\lambda,\lambda\rangle.$$ Now, a direct calculation as before proves that $$\left.\dfrac{d}{d t}\right|_{t=0} ||R^{\omega+t\lambda}||_\r^{\bullet \,2} = 2\, \mathrm{Re}\left( \langle d^{\widehat{DS^\omega}\star}(R^{\omega})\mid  \widehat{\Upsilon}_{\mathfrak{qg}^\#}(\lambda)\rangle^\bullet_\r\right)$$
and the theorem follows.
\end{proof}

We will refer to equation~(\ref{6.f1.1}) as the \emph{non--commutative geometrical Yang--Mills equation}. Notice that every flat qpc is a Yang--Mills qpc, since it trivially satisfies this equation and of course, the group $\qGG_\YM$ acts on the space of Yang--Mills qpc's.

According to Appendix C, in the \emph{classical} case, the non--commutative geometrical Yang--Mills equation coincides with the (\emph{dualization via the pull--back} of the) \emph{classical} one. 

It is important to remark that the twisted covariant derivative naturally appears both in the non--commutative geometrical Bianchi identity (see equation (\ref{a.5})) and in the non--commutative geometrical Yang--Mills equation. \emph{In the case of the Bianchi identity}, it arises naturally when computing $D^\omega R^\omega$ (which, in the non--commutative geometrical setting, is not equal to zero in general). \emph{In the case of the Yang--Mills equation}, it naturally appears when taking the variation of $R^\omega$ with respect to $\lambda$.

The operator $$DS^\omega:=D^\omega - S^\omega$$ is the \emph{correct} covariant derivative on $\Mor(\ad,\Delta_\Hor)$ for working with arbitrary qpc's and for allowing embeddings of $\mathfrak{qg}^\#$ into $\mathfrak{qg}^\#\otimes\mathfrak{qg}^\#$ beyond the classical choice $\displaystyle-{1\over 2}\,\mathrm{c}^T$. A concrete example illustrating this phenomenon is presented in references \cite{sald5,sald6}.

\subsection{Non--Commutative Geometrical Yang--Mills Scalar Matter Theory}

As in the \emph{classical} case, we will start by presenting the necessary axioms of the theory.

\begin{Definition}
\label{6.1.11}
A non--commutative geometrical Yang--Mills scalar matter model consists of
\begin{enumerate}
\item A non--commutative geometrical Yang--Mills model (see Definition \ref{6.1.1}).
\item A finite--dimensional $H$--corepresentation $\delta^V$.
\item A polynomial function ${\mathcal V}:\R \longrightarrow \R $  called {\it the potential}.
\end{enumerate}
\end{Definition} 

It is worth mentioning that the potential ${\mathcal V}$ is taken to be a polynomial function, since this is the only type of potential that appears in the \emph{classical} gauge theory setting (see, for example, \cite{gtvp}).

\begin{Definition}
\label{6.1.12} 
Considering Definition \ref{6.1.11}, we define the non--commutative geometrical Yang--Mills scalar matter action as the association 
\begin{equation*}
\qS_\YMSM: \mathfrak{qpc}(\zeta) \times E^V_\l\times E^{\overline{V}}_\r \longrightarrow \R
\end{equation*}
given by $$\qS_\YMSM(\omega,T_1,T_2)=\qS_\YM(\omega)+\qS_\SM(\omega,T_1,T_2),$$
where $\qS_\YM$ is the non--commutative geometrical Yang--Mills action and $\qS_\SM$ is the non--commutative geometrical scalar matter action, which is given by
\begin{equation*}
\begin{aligned}
    \qS_\SM(\omega,T_1,T_2)=\dfrac{1}{4}\left(|| \nabla^{\omega}_{V}T_1||^{\bullet \,2}_\l-{\mathcal V}_\l(T_1)- || \widehat{\nabla}^{\omega}_{\overline{V}}T_2||^{\bullet \,2}_\r+{\mathcal V}_\r(T_2) \right),
\end{aligned}
\end{equation*}
where ${\mathcal V}_\l(T_1):={\mathcal V}\circ \langle T_1\mid T_1\rangle^\bullet_\l$ and ${\mathcal V}_\r(T_2):={\mathcal V}\circ \langle T_2\mid T_2\rangle^\bullet_\r$. 
\end{Definition}

As in the previous subsection and for the same reasons, we can define the group of symmetries of the model.

\begin{Definition}
\label{6.1.13}
We define the quantum gauge group of the Yang-Mills--Matter model as the group $$\qGG_\YMSM:=\{\F \in \qGG\mid \qS_\YMSM(\omega,T_1,T_2)= \qS_\YMSM(\F^\circledast \omega,\mathbf{A}_\F(T_1),\widehat{\mathbf{A}}_\F(T_2))\;\; \mbox{ for all } $$ $$(\omega,T_1,T_2)\, \in \, \mathfrak{qpc}(\zeta)\times E^V_\l\times E^{\overline{V}}_\r \} \subseteq \qGG,$$ where $$\mathbf{A}_\F: E^V_\l\longrightarrow E^V_\l,\qquad \mathbf{A}_\F(T)(v)=\F(T(v)),$$ $$\widehat{\mathbf{A}}_\F: E^{\overline{V}}_\r\longrightarrow E^{\overline{V}}_\r,\qquad \widehat{\mathbf{A}}_\F(T)(v)=\F(T(v)^\ast)^\ast.$$
\end{Definition}

As in the previous subsection, we have

\begin{Proposition}
    \label{gaugesclar}
    Let $\F:\Omega^\bullet(P)\longrightarrow \Omega^\bullet(P)$ be a quantum gauge transformation such that $\F$ is a graded differential $\ast$--algebra morphism. Then $\F$ $\in$ $\qGG_\YMSM$.
\end{Proposition}

\begin{proof}
  By Proposition \ref{gaugeym}, we know that $\qS_\YM(\F^\circledast \omega)=\qS_\YM(\omega)$ for every $\omega$ $\in$ $\mathfrak{qpc}(\zeta)$. So, it is sufficient to prove that $\qS_\SM(\F^\circledast \omega,\mathbf{A}_\F(T_1),\widehat{\mathbf{A}}_\F(T_2))=\qS_\SM(\omega,T_1,T_2)$ for all $(\omega,T_1,T_2)$ $\in$ $\mathfrak{qpc}(\zeta)\times E^V_\l\times E^{\overline{V}}_\r$.
  
  By Corollary 4.12 of reference \cite{sald2}, we have $$\nabla^{\F^{\circledast}\omega}_{V}=(\id_{\Omega^\bullet(B)}\otimes_B \mathbf{A}_\F)\circ \nabla^{\omega}_{V}\circ \mathbf{A}^{-1}_{\F},$$ then, by equation (\ref{3.f7}) one gets
  \begin{eqnarray*}
 \nabla^{\F^{\circledast}\omega}_{V}(\mathbf{A}_{\F}(T_1))  = (\id_{\Omega^\bullet(B)}\otimes_B \mathbf{A}_\F)\nabla^{\omega}_{V}T_1 = 
 \sum_k \mu^{_{D^\omega\circ T_1}}\otimes_B \mathbf{A}_\F(T^\l_k).
\end{eqnarray*}
In this way, by Proposition 4.13 of reference \cite{sald2} we obtain
\begin{eqnarray*}
||\nabla^{\F^{\circledast}\omega}_{V}(\mathbf{A}_{\F}(T_1))||^{\bullet\,2}_\l&=&\langle \nabla^{\F^{\circledast}\omega}_{V}(\mathbf{A}_{\F}(T_1))\mid \nabla^{\F^{\circledast}\omega}_{V} (\mathbf{A}_{\F}(T_1))\rangle^\bullet_\l
\\
  &= &
 \sum_k \langle \mu^{_{D^\omega\circ T_1}} \langle\mathbf{A}_{\F}(T^\l_k),\mathbf{A}_{\F}(T^\l_k)\rangle_\l \mid \mu^{_{D^\omega\circ T_1}} \rangle^\bullet
  \\
  &= &
  \sum_k \langle \mu^{_{D^\omega\circ T_1}} \langle T^\l_k, T^\l_k\rangle_\l \mid \mu^{_{D^\omega\circ T_1}} \rangle^\bullet
= 
  \langle \nabla^{\omega}_{V}T_1\mid \nabla^{\omega}_{V} T_1\rangle^\bullet_\l
=
  ||\nabla^{\omega}_{V}T_1||^{\bullet\,2}_\l.
\end{eqnarray*}
  On the other hand, by Proposition 4.13 of reference \cite{sald2} we get 
  \begin{eqnarray*}
 {\mathcal V}_\l(\mathbf{A}_{\F}(T_1)) ={\mathcal V} \circ \langle \mathbf{A}_{\F}(T_1)\mid \mathbf{A}_{\F}(T_1)\rangle^\bullet_\l
&=& 
 {\mathcal V} \circ \langle \mathbbm{1} \langle \mathbf{A}_{\F}(T_1), \mathbf{A}_{\F}(T_1)\rangle_\l\mid \mathbbm{1}\rangle^\bullet_\l\\
  &= &
 {\mathcal V} \circ \langle \mathbbm{1} \langle T_1, T_1\rangle_\l\mid \mathbbm{1}\rangle^\bullet_\l
  \\
  &= &
  {\mathcal V} \circ \langle T_1\mid T_1\rangle^\bullet_\l={\mathcal V}_\l(T_1).
\end{eqnarray*}
Since the following relation holds (\cite{sald2}): $$\widehat{\nabla}^{\F^{\circledast}\omega}_{\overline{V}}=(\widehat{\mathbf{A}}_\F\otimes_B \id_{\Omega^\bullet(B)} )\circ \widehat{\nabla}^{\omega}_{\overline{V}}\circ \widehat{\mathbf{A}}^{-1}_{\F},$$ an analogous calculation shows that $$||\widehat{\nabla}^{\F^{\circledast}\omega}_{\overline{V}}(\widehat{\mathbf{A}}_{\F}(T_2))||^{\bullet\,2}_\r=||\widehat{\nabla}^{\omega}_{\overline{V}}T_2||^{\bullet\,2}_\r,\qquad {\mathcal V}_\r(\widehat{\mathbf{A}}_{\F}(T_2))= {\mathcal V}_\r(T_2)$$ and therefore, $\F$ $\in$ $\qGG_\YMSM$.
\end{proof}

Clearly, in general, there can be elements of $\qGG_\YMSM$ that are not graded differential $\ast$--algebra morphisms. A concrete example is provided by Proposition 5.8 and Remark 5.9 of reference \cite{sald6}.

As in the previous subsection, the next step is to derive the \emph{non--commutative geometrical field equations} for
$(\omega,T_1,T_2)$ $\in$ $\mathfrak{qpc}(\zeta)\times E^V_\l\times E^{\overline{V}}_\r$ by postulating that the first variation of $\qS_\YMSM$ vanishes, in complete agreement with the \emph{classical} case.

\begin{Definition}
\label{6.1.14}
A stationary point of $\qS_\YMSM$ is a triplet $(\omega,T_1,T_2)$ $\in$ $\mathfrak{qpc}(\zeta)\times E^V_\l\times E^{\overline{V}}_\r$ such that for any $(\lambda, U_1,U_2)$ $\in$ $\overrightarrow{\mathfrak{qpc}(\zeta)}  \times E^V_\l\times E^{\overline{V}}_\r$, we have  $$\left.\dfrac{d}{d t}\right|_{t=0}\qS_\YMSM(\omega+t\,\lambda ,T_1,T_2)=\left.\dfrac{\partial}{\partial z}\right|_{z=0}\qS_\YMSM(\omega,T_1+z\,U_1,T_2+z\,U_2)=0.$$ Stationary points are also called non--commutative geometrical Yang--Mills scalar matter fields and in terms of a traditional physical interpretation, they can be interpreted as {\it scalar matter and antimatter fields coupled to gauge boson fields possessing the symmetry $\qGG_\YMSM$ in the presence of the potential ${\mathcal V}$}.
\end{Definition}

Notice that $\qS_\YMSM(\omega+t\,\lambda,T_1,T_2)$ depends polynomially on $t$, and that $\qS_\YMSM(\omega,T_1+z\,U_1,T_2+z\,U_2)$ depends polynomially on $z$. Hence, we interpret both derivatives of $\qS_\YMSM$ as formal derivatives in the obvious way. We now proceed to derive the equations of motion.

\begin{Theorem}
\label{6.1.15}
The triple $(\omega,T_1,T_2)$ $\in$ $\mathfrak{qpc}(\zeta)\times E^V_\l\times E^{\overline{V}}_\r$ is a Yang--Mills scalar matter field if and only if for all $\lambda$ $\in$ $\overrightarrow{\mathfrak{qpc}(\zeta)}$ 
\begin{equation}
\label{6.f1.4}
\begin{aligned}
\mathrm{Re}( \langle \Upsilon_{V}(K^{\lambda}(T_1))\,|\,\nabla^{\omega}_{V}T_1\rangle^\bullet_\l &-  \langle \widehat{\Upsilon}_{\overline{V}}(\widehat{K}^{\lambda}(T_2))\,|\,\widehat{\nabla}^{\omega}_{\overline{V}}T_2\rangle^\bullet_\r  
\\
& = \langle d^{DS^\omega\star}(R^{\omega})\mid  \Upsilon_{\mathfrak{qg}^\#}(\lambda)\rangle^\bullet_\l+  \langle(d^{\widehat{DS}^{\omega}\star}(R^{\omega})\mid  \widehat{\Upsilon}_{\mathfrak{qg}^\#}(\lambda)\rangle^\bullet_\r),
\end{aligned}
\end{equation}
where the operators $K^{\lambda}$, $\widehat{K}^\lambda$ are defined in Definition \ref{koperator}; and 
\begin{equation}
\label{6.f1.5}
\nabla^{\omega\,\star}_{V}\left(\nabla^{\omega}_{V}\,T_1\right)-{\mathcal V}'_\l(T_1)\,T_1=0,
\end{equation}
\begin{equation}
\label{6.f1.6}
\widehat{\nabla}^{\omega\,\star}_{\overline{V}} \left(\widehat{\nabla}^{\omega}_{\overline{V}}\,T_2\right)-{\mathcal V}'_\r(T_2)\,T_2 =0,
\end{equation}
with ${\mathcal V}'_\l(T_1):={\mathcal V}'\circ \langle T_1\mid T_1\rangle^\bullet_\l$ and ${\mathcal V}'_\r(T_2):={\mathcal V}'\circ \langle T_2\mid T_2\rangle^\bullet_\r$, where ${\mathcal V}'$ is the derivative of ${\mathcal V}$.
\end{Theorem}

\begin{proof}
By the first part of equation (\ref{2.f30.1.1}) and equation (\ref{3.f8}), we have
\begin{eqnarray*}
    || \nabla^{\omega+t\lambda}_{V}T_1||^{\bullet \,2}_\l&=&\langle\nabla^{\omega+t\lambda}_{V}T_1\mid \nabla^{\omega+t\lambda}_{V}T_1 \rangle_\l^\bullet
    \\
    &=&
    \langle\nabla^{\omega}_{V}T_1+t\,\Upsilon_{V}(K^{\lambda}(T_1))\mid \nabla^{\omega}_{V}T_1+t\,\Upsilon_{V}(K^{\lambda}(T_1)) \rangle_\l^\bullet
    \\
    &=&
    \langle\nabla^{\omega}_{V}T_1\mid \nabla^{\omega}_{V}T_1 \rangle_\l^\bullet+ t\,\langle\nabla^{\omega}_{V}T_1\mid \Upsilon_{V}(K^{\lambda}(T_1)) \rangle_\l^\bullet
    \\
    &+&
    \langle \Upsilon_{V}(K^{\lambda}(T_1))\mid \nabla^{\omega}_{V}T_1\rangle_\l^\bullet+ t^2\,\langle \Upsilon_{V}(K^{\lambda}(T_1))\mid \Upsilon_{V}(K^{\lambda}(T_1))\rangle_\l^\bullet.
\end{eqnarray*}
Hence 
\begin{eqnarray*}
 \left.\dfrac{d}{d t}\right|_{t=0} ||\nabla^{\omega+t\lambda}_{V}T_1||_\l^{\bullet \,2} &=&  \langle\nabla^{\omega}_{V}T_1\mid \Upsilon_{V}(K^{\lambda}(T_1)) \rangle_\l^\bullet
+ 
  \langle \Upsilon_{V}(K^{\lambda}(T_1))\mid \nabla^{\omega}_{V}T_1\rangle_\l^\bullet
  \\
  &= &
  2\, \mathrm{Re}\left( \langle \Upsilon_{V}(K^{\lambda}(T_1))\mid \nabla^{\omega}_{V}T_1\rangle_\l^\bullet\right).
\end{eqnarray*}

Similarly, by the second part of equation (\ref{2.f30.1.1}) and equation (\ref{3.f8.1}), we obtain
$$\left.\dfrac{d}{d t}\right|_{t=0} ||\widehat{\nabla}^{\omega+t\lambda}_{\overline{V}}T_2||_\r^{\bullet \,2}=2\, \mathrm{Re}\left( \langle \widehat{\Upsilon}_{\overline{V}}(\widehat{K}^{\lambda}(T_2))\,|\,\widehat{\nabla}^{\omega}_{\overline{V}}T_2\rangle^\bullet_\r\right).$$ Furthermore, by  Theorem \ref{6.1.5}   $$\left.\dfrac{d}{d t}\right|_{t=0} \left(||R^{\omega+t\lambda}||_\l^{\bullet \,2}+ ||R^{\omega+t\lambda}||_\r^{\bullet \,2}\right)=2\,\mathrm{Re}\left( \langle d^{DS^\omega\star} (R^{\omega})\mid  \Upsilon_{\mathfrak{qg}^\#}(\lambda)\rangle^\bullet_\l+ \langle d^{\widehat{DS^\omega}\star} (R^{\omega})\mid  \widehat{\Upsilon}_{\mathfrak{qg}^\#}(\lambda)\rangle^\bullet_\r\right).$$
Therefore, equation (\ref{6.f1.4}) holds for every $\lambda$ $\in$ $\overrightarrow{\mathfrak{qpc}(\zeta)}$ if and only if $$\displaystyle\left.\dfrac{d}{d t}\right|_{t=0}\qS_\YMSM(\omega+t\,\lambda ,T_1,T_2)=0$$ for every $\lambda$ $\in$ $\overrightarrow{\mathfrak{qpc}(\zeta)}$.

On the other hand, 
\begin{eqnarray*}
  ||\nabla^{\omega}_{V}(T_1+z\,U_1)||_\l^{\bullet \,2} &=&  \langle\nabla^{\omega}_{V}(T_1+z\,U_1)\mid \langle\nabla^{\omega}_{V}(T_1+z\,U_1) \rangle_\l^\bullet
  \\
  &= &
  \langle\nabla^{\omega}_{V}T_1+z\,\nabla^{\omega}_{V}U_1\mid \langle\nabla^{\omega}_{V}T_1+z\,\nabla^{\omega}_{V}U_1 \rangle_\l^\bullet
  \\
  &= &
  \langle\nabla^{\omega}_{V}T_1 \mid \nabla^{\omega}_{V}T_1\rangle_\l^\bullet+z^\ast\,\langle\nabla^{\omega}_{V}T_1 \mid \nabla^{\omega}_{V}U_1 \rangle_\l^\bullet
  \\
  &+ &
  z\,\langle \nabla^{\omega}_{V}U_1\mid \nabla^{\omega}_{V}T_1\rangle_\l^\bullet+zz^\ast\,\langle\nabla^{\omega}_{V}U_1 \mid \nabla^{\omega}_{V}U_1 \rangle_\l^\bullet.
\end{eqnarray*}
Thus 
\begin{eqnarray*}
 \left.\dfrac{\partial}{\partial z}\right|_{z=0} ||\nabla^{\omega}_{V}(T_1+z\,U_1)||_\l^{\bullet \,2} =  \langle \nabla^{\omega}_{V}U_1\mid \nabla^{\omega}_{V}T_1\rangle_\l^\bullet
  = 
  \langle U_1\mid \nabla^{\omega\,\star}_{V}\left(\nabla^{\omega}_{V}T_1\right)\rangle_\l^\bullet.
\end{eqnarray*}
Also we have
\begin{eqnarray*}
  {\mathcal V}_\l(T_1+z\,U_1) =  {\mathcal V}_\l(\langle T_1+zU_1\mid T_1+z\,U_1\rangle^\bullet_\l )
  &= &
  {\mathcal V}_\l(\langle T_1 \mid T_1\rangle^\bullet_\l+z^\ast\,\langle T_1 \mid U_1\rangle^\bullet_\l
  \\
  &+ &
  z\,\langle U_1 \mid T_1\rangle^\bullet_\l+z\,z^\ast\,\langle U_1 \mid U_1\rangle^\bullet_\l),
\end{eqnarray*}
and since ${\mathcal V}_\l$ is a polynomial function, we get $$\left.\dfrac{\partial}{\partial z}\right|_{z=0} {\mathcal V}_\l(T_1+z\,U_1) =   {\mathcal V}'_\l(T_1)\,\langle U_1 \mid T_1\rangle^\bullet_\l=\langle U_1 \mid {\mathcal V}'_\l(T_1)\,T_1\rangle^\bullet_\l.$$ It follows that 
\begin{eqnarray*}
    \left.\dfrac{\partial}{\partial z}\right|_{z=0}  \left(||\nabla^{\omega}_{V}(T_1+z\,U_1)||_\l^{\bullet \,2}-{\mathcal V}_\l(T_1+z\,U_1)\right)&=&\langle U_1\mid \nabla^{\omega\,\star}_{V}\left(\nabla^{\omega}_{V}T_1\right)\rangle_\l^\bullet+\langle U_1 \mid {\mathcal V}'_\l(T_1)\,T_1\rangle^\bullet_\l
    \\
    &=&
    \langle U_1\mid \nabla^{\omega\,\star}_{V}\left(\nabla^{\omega}_{V}\,T_1\right)-{\mathcal V}'_\l(T_1)\,T_1 \rangle^\bullet_\l.
\end{eqnarray*}
Finally, since $\langle-|-\rangle^\bullet_\l$ is an inner product, the equation $$0=\left.\dfrac{\partial}{\partial z}\right|_{z=0}  \left(||\nabla^{\omega}_{V}(T_1+z\,U_1)||_\l^{\bullet \,2}-{\mathcal V}_\l(T_1+z\,U_1)\right)$$ holds for all $U_1$ $\in$ $E^V_\l$ if and only if equation (\ref{6.f1.5}) holds. 

A completely analogous calculation shows that the equation $$0=\left.\dfrac{\partial}{\partial z}\right|_{z=0}  \left(-||\widehat{\nabla}^{\omega}_{\overline{V}}(T_2+z\,U_2)||_\r^{\bullet \,2}+{\mathcal V}_\r(T_2+z\,U_2)\right)$$ holds for all $U_2$ $\in$ $E^{\overline{V}}_\r$ if and only if equation (\ref{6.f1.6}) holds.
\end{proof}

We shall refer to equations~(\ref{6.f1.4}), (\ref{6.f1.5}), and (\ref{6.f1.6}) as the \emph{non--commutative geometrical Yang--Mills scalar matter field equations}. According to Appendix C, in the \emph{classical} case, the non--commutative geometrical Yang--Mills scalar matter field equations coincides with the (\emph{dualization via the pull--back} of the) classical ones.

\section{Examples} 

Throughout the paper, we have made several non–trivial assumptions on the theory.

\begin{Remark}
\label{remaconditions}
We shall enumerate all the conditions:
    \begin{enumerate}
    \item A quantum principal $\G$--bundle $\zeta=(P,B,\Delta_P)$, where $B$ stable under holomorphic calculus. 
    \item A bicovariant $\ast$--FODC $(\Gamma,d)$ over $H$ for which the quantum dual Lie algebra $\mathfrak{qg}^\#$ is finite--dimensional. 
    \item A differential calculus on $\zeta$ induced by the preceding bicovariant $\ast$--FODC, such that $\Omega^\bullet(B)$ satisfies Definition~\ref{a.2.1} and admits a left quantum Hodge star operator.
\item (Using Remark \ref{Sadjoint}) For every qpc $\omega$ of $\zeta$, the operator $$d^{S^{\omega}}:=\Upsilon_{\mathfrak{qg}^\#}\circ S^{\omega} \circ \Upsilon^{-1}_{\mathfrak{qg}^\#}$$ is adjointable or formally adjointable with respect to $\langle-|-\rangle^\bullet_\l$ in (see equation (\ref{realtensor})) $$ \Upsilon_{\mathfrak{qg}^\#}(\Mor(\ad,\Delta_\H)^\dagger);$$  while the operator $$ d^{\widehat{S}^{\omega}}:=\widehat{\Upsilon}_{\mathfrak{qg}^\#}\circ \widehat{S}^{\omega}\circ \widehat{\Upsilon}^{-1}_{\mathfrak{qg}^\#}$$ is adjointable or formally adjointable with respect to $\langle-|-\rangle^\bullet_\r$ in $$ \widehat{\Upsilon}_{\mathfrak{qg}^\#}(\Mor(\ad,\Delta_\Hor)^\dagger).$$ 
\end{enumerate}
\end{Remark}

In this way, our theory can be applied in every qpb with a differential calculus for which Remark \ref{remaconditions} holds.

\begin{Remark}
    \label{oneside}
    In many works in non--commutative geometry, the theory is developed only for right modules. In this paper, however, we have developed the theory for both left and right modules. This is because, as the reader may have already noticed, the theories associated with the left and right structures are, in general, not equivalent. For example, the induced qlc's can only be combined into a $B$--bimodule qlc when $\omega$ is regular (see Remark \ref{remacon}). As another example,  in reference \cite{sald3} it is shown that the operators $$ \nabla^{\omega\star}_V\nabla^{\omega}_V, \qquad \widehat{\nabla}^{\omega\star}_V\widehat{\nabla}^{\omega}_V,$$  are different operators that do not commute with each other (in the particular case of the qpb studied in that paper). This fact is strong motivating
reason to consider the left/right structure at the same time: it
appears that in general, ignoring one of the structures leaves
to losing relevant information about the quantum spaces. Nevertheless, it is clear that the entire theory could also be developed by focusing exclusively on either the left or the right structure.
\end{Remark}

Unfortunately, in general, we do not yet know how the presence of the operator $S^{\omega}\neq 0$ in the non--commutative geometrical Bianchi identity, in the non--commutative geometrical Yang--Mills equation and in the non--commutative geometrical Yang--Mills scalar matter equations modifies the \emph{classical} field equations and the physical predictions. At present, we believe that such an analysis can only be carried out by working with a specific quantum principal bundle endowed with a specific differential calculus, since the theory developed in this paper is intentionally formulated in the greatest generality we could achieve. 

More concretely, the operator $S^{\omega}$ strongly depends on the differential calculus on $B$ and on the differential calculus on $H$, and these structures can be chosen almost arbitrarily. Once these structures have been fixed, the operator $S^{\omega}$ also depends on the embedded differential $\Theta$, which in general is not unique. All these aspects introduce complications when trying to determine general information about the non--commutative geometrical field equations. However, at the same time, this freedom in choosing the structures makes the theory richer in terms of examples to study, which is appropriate from a mathematical point of view.

The absence of a general explanation of how the non--commutative geometrical field equations modify the \emph{classical} field equations and the physical predictions is not exclusive to our theory. For example, Connes’ Yang--Mills theory does not provide such analysis either~(\cite{con}). 

In Connes’ Yang--Mills theory, whenever the operator $[\nabla,-]$ admits an adjoint operator (or a formal adjoint operator), it is possible to derive a Yang--Mills equation, i.e., an equation that characterized the critical points of the Yang--Mills--Connes functional~(\cite{conym,conym1}). However, even in this setting, there is no general explanation of how such equation modifies the \emph{classical} one and the physical predictions of the theory.

Notice that, in Connes' Yang--Mills theory, in order to obtain a field equation, it is necessary that the operator $[\nabla,-]$ admits an adjoint; while in our theory, in order to obtain a field equation, it is necessary that the operator $d^{DS^\omega}$ (or $d^{\widehat{DS^\omega}}$) admits an adjoint. In this sense, in reference \cite{saldcon}, we find that the operators  $d^{DS^\omega}$ (and  $d^{\widehat{DS^\omega}}$) and  $[\nabla,-]$ admit formal adjoint operators at the same time for quantum principal $SU(N)$--bundles over non--commutative toric manifolds and hence, this class of qpb's is suitable to study both theories. In reference \cite{saldcon} we also develop a comparison between both theories, highlighting their similarities and differences in a concrete example: quantum principal $U(1)$--bundles over the non--commutative $n$--torus.

In order to keep the length of the paper reasonable, in this section we will present two representative  classes of qpb’s satisfying the conditions of Remark~\ref{remaconditions}. We have chosen to prioritize the presentation of these general two classes of qpb's for which our theory applies rather than carrying out a complete analysis of the solutions of the non--commutative geometrical field equations, the action of the quantum gauge group on their solutions and the corresponding physical predictions in a specific example. 

This is because our intention in this paper is to promote a general theory that can be applied and studied in several interested and useful spaces in physics (as Connes' Yang--Mills theory), and developing a full specific example could give the impression that the theory applies only to that particular case. However, readers interested in such detailed analyses for concrete models are encouraged to consult references \cite{sald3,sald4,sald5,sald6}.

It is worth mentioning that the qpb’s presented in   \cite{saldcon} (i.e., quantum principal $SU(N)$--bundles over non--commutative toric manifolds) constitute another class of qpb’s to which the our theory applies, distinct from the two classes of qpb's presented in this section. Of course, these three classes of qpb's are not the only bundles for which our theory applies. For example, reference \cite{sald3} shows the application of our theory on the quantum Hopf fibration, also known as the $q$--deformed Dirac monopole bundle, and in reference \cite{sald7} (in preparation), we show a study on instantons as solutions of the non--commutative geometrical Yang--Mills equation in the \emph{quantum} version of the \emph{classical} fibration $$\mathbb{S}^3\cong SU(2)\lhook\joinrel\relbar\joinrel\rightarrow \mathbb{S}^7\longrightarrow \mathbb{S}^4,$$ comparing the obtained results with the theory presented in reference \cite{conym1} for the Connes' Yang--Mills theory.

\subsection{Quantum Principal Bundles over Smooth Manifolds}

In this subsection, we will study qpb’s over \emph{classical} spaces (manifolds). Our discussion is based on~\cite{micho1}; therefore, we strongly recommend that the reader consult this reference before proceeding, as we cannot cover all the relevant details here. These qpb’s are particularly important for the purposes of this research because, as in differential geometry, the quantum base space represents the space--time. Consequently, the qpb’s of this section provide models of \emph{non--commutative geometrical} Yang--Mills scalar matter theories on a  \emph{classical} space--time.

In reference \cite{micho1}, a quantum principal $\G$--bundle (with $H^\infty=(H,\cdot,\mathbbm{1},\ast,\Delta,\epsilon,S)$ the associated $\ast$--Hopf algebra  of the quantum group $\G$) is a triple 
\begin{equation}
    \label{clas1}
    \zeta=(P,\iota,\Delta_P)
\end{equation}
where $P$ is a $\ast$--algebra, $$\iota:B\longrightarrow P $$ is a unital linear map,  $B:=C^\infty_\C(M)$ is the algebra of $\C$--valued smooth functions of $M$ with $M$ a  $n$--dimensional oriented closed Riemannian manifold, and $$\Delta_P:P\longrightarrow P\otimes H $$ is a linear map such that  for each $x$ $\in$ $M$, there exists an open set $U\subseteq M$ containing $x$ and a $\ast$--algebra morphism $$\pi_U:P\longrightarrow C^\infty_\C(U)\otimes H $$ such that the following properties hold
\begin{enumerate}
    \item We have $$ \pi_U(\iota(f))=f|_{U}\otimes \mathbbm{1}$$ for every $f$ $\in$ $B$.
    \item If $q=\iota(\varphi)\,p$, with $\varphi$ $\in$ $C^\infty_\C(U_c)$, $p$ $\in$ $P$,  then $$\pi(q)=0\;\;\Longrightarrow \;\; q=0.$$ Here, $C^\infty_\C(U_c)$ denotes the set of $\C$--valued smooth functions of $U$ with compact support.
    \item We have $$(\id_{C^\infty_\C(U)}\otimes \Delta)\circ \pi_U=(\pi_U\otimes \id_H)\circ \Delta_P\quad \mbox{ and }\quad  C^\infty_\C(U_c)\otimes H\subseteq \pi_U(P).$$
\end{enumerate}

In Section 2 of reference \cite{micho1} it is proven that triples as in equation (\ref{clas1}) satisfies the definition of qpb given in Section 2.2. In particular, $\iota$ is a $\ast$--algebra monomorphism and hence, we can consider that $B\subseteq P$. Furthermore, according to Lemma 2.2 of reference \cite{micho1}, $B$ is in the centre of $P$. 

It is worth mentioning that  in Section 2 of reference \cite{micho1}, the author presents a construction of this class of qpb's by means of \emph{classical} cocycles; so it is easy to find concrete examples of this type of qpb's. Furthermore, it is well--known that $B$ is stable under holomorphic calculus (\cite{con}); so we get the first point of Remark \ref{remaconditions} holds.  

Let $x$ $\in$ $M$ and consider $U\subseteq M$ the open set containing $x$ for which the map $\pi_U$ exists. Let $V\subseteq U$ be a nonempty open set. Then, according to Lemma 2.5 of reference \cite{micho1}, we get that $$\pi_U(I_V)\subseteq C^\infty_\C(V_c)\otimes H$$ and the restriction $$\pi_{U}|_{I_V}:I_V\longrightarrow C^\infty_\C(V_c)\otimes H $$ is a $\ast$--algebra isomorphism. In this way, we define the $\ast$--algebra monomorphism $$\psi_U:=(\pi_U|_{I_V})^{-1}: C^\infty_\C(U_c)\otimes H\longrightarrow P.$$

In Section 3 of reference \cite{micho1}, a differential calculus on $\zeta$ is defined by:
\begin{enumerate}
    \item A graded differential $\ast$--algebra $(\Omega^\bullet(P),d,\ast)$ generated by $\Omega^0(P)=P$ ({\it quantum differential forms of} $P$).
    \item A bicovariant $\ast$--FODC $(\Gamma,d)$ over $H$.
    \item The maps $\pi_U$, $\psi_U$ are extending to graded differential $\ast$--algebra morphisms  $$\pi_U:\Omega^\bullet(P)\longrightarrow \Omega^\bullet_\C(U)\otimes \Gamma^\wedge,$$ $$\psi_U:\Omega^\bullet_\C(U_c)\otimes \Gamma^\wedge \longrightarrow \Omega^\bullet(P).$$ Here, we have considered that $\otimes$ is the tensor product of graded differential $\ast$--algebras, $\Omega^\bullet_\C(U)$ is the space of $\C$--valued differential forms of $U$, and $\Omega^\bullet_\C(U_c)$ denotes the space of $\C$--valued differential forms of $U$ with compact support.
\end{enumerate}

In Proposition 3.16 of reference \cite{micho1}, the author proves that up isomorphisms, there exists only one differential calculus on a given $\zeta$. Moreover, the map $\Delta_P$ can be extended to a graded differential $\ast$--algebra morphism $$\Delta_{\Omega^\bullet(P)}:\Omega^\bullet(P)\longrightarrow \Omega^\bullet(P)\otimes \Gamma^\wedge$$ and hence, we have a differential calculus in the sense of Section 2.2. Moreover, in this same proposition, it is proven that the space of base forms is exactly  the de Rham complex
\begin{equation}
    \label{smooth.f1}
    (\Omega^\bullet(B):=\Omega^\bullet_{\C}(M),d,\ast)
\end{equation}
of $\C$--valued differential forms of $M$.

Obviously,  $(\Omega^\bullet(B),d,\ast)$ fulfills the conditions mentioned in Definition \ref{a.2.1}  for any $B$--valued inner product induced by a complex--extension of a Riemannian metric on $M$  and clearly, there exists a Hodge star operator (\cite{nodg,gtvp}). Therefore, to show that this class of qpb's is suitable for our theory, it is sufficient to prove the adjointability (or formal adjointability) of the operators $d^{S^{\omega}_\l}$, $d^{\widehat{S}^{\omega}_\r}$, when the quantum dual Lie algebra is finite--dimensional.

Consider a qpb as in equation (\ref{clas1}) such that $\mathfrak{qg}^\#$ is a finite--dimensional $\C$--vector space. Let $\omega$ be a qpc and consider the operator $$S^{\omega}: \Mor(\ad,\Delta_\Hor)\longrightarrow \Mor(\ad,\Delta_\Hor) $$ given by $$S^{\omega}(\tau):=\langle \omega,\tau\rangle-(-1)^k\langle\tau,\omega\rangle-(-1)^k[\tau,\omega]$$ for every $\tau$ $\in$ $\Mor(\ad,\Delta_\H)$ with $\Im(\tau)$ $\subseteq$ $\Hor^k P$. 

\begin{Proposition}
    \label{ds1}
    The operators $d^{S^\omega}$, $d^{\widehat{S}^\omega}$ are formally adjointable.
\end{Proposition}
\begin{proof}
 The space $\Omega^k(B)$ is a finitely generated projective $B$--bimodule (by the Serre--Swan theorem) with a non--degenerated $B$--valued inner product and hence, in accordance with reference \cite{con} Chapter 6, there exists $s$ $\in$ $\N$ and a self--adjoint matrix $\varsigma_k$ $\in$ $M_s(B)$ such that $$\Omega^k(B)\cong B^s\cdot \varsigma_k$$ as left $B$--modules, and $$\Omega^k(B)\cong \varsigma^T_k\cdot B^s$$ as right $B$--modules, where $\varsigma^T_k$ is the transpose of $\varsigma_k$.

Since $E^{\mathfrak{qg}^\#}_\l\cong B^d \cdot \rho^{\mathfrak{qg}^\#}(\mathbbm{1})$ for some $d$ $\in$ $\N$ (as left $B$--modules, see equation (\ref{new.5})), we obtain that $$\Omega^k(B)\otimes_B E^{\mathfrak{qg}^\#}_\l\cong B^{sd}\cdot (\varsigma_k\otimes \rho^{\mathfrak{qg}^\#}(\mathbbm{1}))$$ and hence $\Omega^k(B)\otimes_B E^{\mathfrak{qg}^\#}_\l$  is also a finitely generated projective left $B$--module. Here, the element $\varsigma_k\otimes \rho^{\mathfrak{qg}^\#}(\mathbbm{1})$ is the Kronecker product of $\varsigma_k$ and $\rho^{\mathfrak{qg}^\#}(\mathbbm{1})$. Notice $$(\varsigma_k\otimes \rho^{\mathfrak{qg}^\#}(\mathbbm{1}))^\dagger=\varsigma_k\otimes \rho^{\mathfrak{qg}^\#}(\mathbbm{1}),$$ where $\dagger$ denotes the composition of the usual matrix transposition and the $\ast$ operator of $B$.

Therefore, by Proposition \ref{a.6} of Appendix B, every left $B$--module endomorphism of the space $\Omega^\bullet(B)\otimes_B E^{\mathfrak{qg}^\#}_\l$ is $B$--adjointable with respect to the Hermitian structure induced by $(-,-)^{sd}_\l$ (see Appendix B); which is, by construction, exactly the $B$--valued inner product 
\begin{equation*}
\begin{aligned}
\langle-,-\rangle^\bullet_\l : \Omega^\bullet(B)\otimes_B E^{\mathfrak{qg}^\#}_\l \times \Omega^\bullet(B)\otimes_B E^{\mathfrak{qg}^\#}_\l  &\longrightarrow B
\end{aligned}
\end{equation*}
of equation (\ref{ashcas}).

Since $B$ is in the centre of $\Omega^\bullet(P)$, it follows that 
\begin{eqnarray*}
    d^{S^{\omega}}=\Upsilon_{\mathfrak{qg}^\#} \circ S^\omega \circ \Upsilon^{-1}_{\mathfrak{qg}^\#}
    = \Upsilon_{\mathfrak{qg}^\#} \circ (\langle \omega,-\rangle-(-1)^k\langle-,\omega\rangle-(-1)^k[-,\omega]) \circ \Upsilon^{-1}_{\mathfrak{qg}^\#}
\end{eqnarray*}
is a left $B$--linear endomorphism of $\Omega^\bullet(B)\otimes_B E^{\mathfrak{qg}^\#}_\l,$ and  consequently, it is formally adjointable with respect to $\langle-|-\rangle^\bullet_\l.$

A similar argument using equation (\ref{new.6})  and Proposition \ref{a.7} of Appendix B proves that 
\begin{eqnarray*}
d^{\widehat{S}^{\omega}}=\widehat{\Upsilon}_{\mathfrak{qg}^\#} \circ \widehat{S}^\omega \circ \widehat{\Upsilon}^{-1}_{\mathfrak{qg}^\#}
= 
\widehat{\Upsilon}_{\mathfrak{qg}^\#} \circ \wedge \circ (\langle \omega,-\rangle-(-1)^k\langle-,\omega\rangle-(-1)^k[-,\omega]) \circ \wedge \circ \widehat{\Upsilon}^{-1}_{\mathfrak{qg}^\#}
\end{eqnarray*}
is a right $B$--module endomorphism of  $E^{\mathfrak{qg}^\#}_\r \otimes_B \Omega^\bullet(B)$ and therefore, it is formally adjointable with respect to $\langle-|-\rangle^\bullet_\r.$
\end{proof}

In this way, we conclude that our theory can be applied on this class of quantum principal $\G$--bundles, and it is worth noting that in principle, any quantum group $\G$ with any differential calculus on $H$ can be used, provided that $\mathfrak{qg}^\#$ is finite--dimensional.

In~\cite{sald6}, the reader can find a concrete example of this class of qpb’s for the quantum group associated with the Lie group $SU(2)$; but with a $4$--dimensional differential calculus. 

In concrete, by using a quantum principal $SU(2)$--bundle equipped with a $4$--dimensional differential calculus on $SU(2)$, the non--commutative geometrical Yang--Mills equation coincide with the \emph{classical} Yang--Mills equation of the electroweak theory in physics, while the non--commutative geometrical Yang--Mills--scalar--matter equations coincide with the corresponding \emph{classical} Yang--Mills--scalar--matter equations of the electroweak theory. In particular, we show that the \emph{classical} $(SU(2)\times U(1))$--symmetries of the electroweak theory are also symmetries of the non--commutative geometrical model, that is, if $\mathfrak{GG}$ denotes the gauge group of the electroweak theory, then $$\mathfrak{GG}\subseteq \qGG_\YM\qquad \mbox{ and }\qquad \mathfrak{GG}\subseteq \qGG_\YMSM.$$ Moreover, we show that the Higgs mechanism can also be described within this framework.

The study of $SU(N)$--gauge theories for $N\geq 3$ over \emph{classical} spaces by replacing the standard differential calculus on $SU(N)$ with a non--commutative one, and its relation with sectors of the Standard Model of Elementary Particles will be explored in forthcoming publications. 

\subsection{Trivial Quantum Principal Bundles over the Moyal--Weyl Algebra}

 In this subsection we are going to illustrate our theory in trivial quantum principal bundles, in the sense of references \cite{micho2,stheve}, in a concrete non--commutative $\ast$--algebra $B$. From a physical point of view, the qpb’s of this section provide models of \emph{non--commutative geometrical} Yang--Mills scalar matter theories on the \emph{quantum} space--time $B$.

Before proceeding we have to talk about the first point of Remark \ref{remaconditions}. We have assumed that the quantum base space $B$ is stable under holomorphic calculus  only to guarantee that equation (\ref{generators}) holds (\cite{micho3}), and to guarantee that $$b\,b^\ast=0\;\;\Longrightarrow \;\; b=0$$ for all $b$ $\in$ $B$. The last condition implies the positive--definiteness of the $B$--valued inner products $(-,-)_\l$, $(-,-)_\r$. However, for trivial qpb's, it is not necessary to make any assumption on $B$ to ensure the existence of the generators in equation (\ref{generators}), as the reader can check in the proof of Proposition $5.1$ of reference \cite{sald2}. Therefore, to apply our theory in trivial qpb's, it is sufficient to work with a quantum base space $B$ for which $b\,b^\ast=0\;\;\Longrightarrow \;\; b=0$ for all $b$ $\in$ $B$. This is the case, for example, of the Moyal--Weyl algebra~(\cite{05}).

According to references \cite{micho2,stheve}, a quantum principal $\G$--bundle (with $H^\infty=(H,\cdot,\mathbbm{1},\ast,\Delta,\epsilon,S)$ the associated $\ast$--Hopf algebra  of the quantum group $\G$) of the form $$\zeta_\triv=(P:=B\otimes H,B,\Delta_P:=\id_B\otimes \Delta)$$
is called {\it trivial}. If a differential calculus on $\zeta_\triv$ is given by (using the corresponding tensor products) $$\Omega^\bullet(P):=\Omega^\bullet(B)\otimes \Gamma^\wedge\quad  \mbox{ and }\quad \Delta_{\Omega^\bullet(P)}:=\id_{\Omega^\bullet(B)}\otimes \Delta, $$ where $\Omega^\bullet(B)$ is some graded differential $\ast$--algebra generated by $\Omega^0(B)=B$ and $\Gamma^\wedge$ is the universal differential envelope $\ast$--calculus of some bicovariant $\ast$--FODC of $H$, then the differential calculus is commonly referred to as \emph{trivial}.

Let us take any trivial qpb $$\zeta_\triv=(P,B,\Delta)$$ with a trivial differential calculus. Then, the linear map
\begin{equation}
\label{5.f1.1}
\omega^\triv:\mathfrak{qg}^\# \longrightarrow \Omega^1(P),\qquad
\theta \longmapsto \mathbbm{1}\otimes \theta
\end{equation}
is a regular qpc (see Lemma $6.11$ of reference \cite{micho2}). According to Lemma $6.10$ of reference \cite{micho2}, there is a bijection between 
$$\mathrm{Hom}(\mathfrak{qg}^\#,\Omega^\bullet(B))=\{ L^\tau:\mathfrak{qg}^\#\longrightarrow \Omega^\bullet(B)\mid  L^\tau \mbox{ is linear} \} $$ and $\Mor(\ad,\Delta_\Hor)$. This bijection is given by 
\begin{equation}
    \label{trv.f1}
    \tau=(L^\tau\otimes\id_H)\circ \ad
\end{equation}
for every $\tau$ $\in$ $\Mor(\ad,\Delta_\Hor)$. Also, by Lemma $6.11$ of reference \cite{micho2}, there is a bijection between
\begin{eqnarray*}
    \mathrm{Hom}(\mathfrak{qg}^\#,\Omega^1(B))^\dagger=\{ A^\omega:\mathfrak{qg}^\#\longrightarrow \Omega^1(B)\mid   A^\omega \mbox{ is linear and} &\ast\circ A^\omega=A^\omega\circ \ast\}
\end{eqnarray*}
and the set of all qpc's of $\zeta_\triv$. This bijection is given by 
\begin{equation}
\label{trv.f2}
\omega=(A^\omega\otimes \id_H)\circ \ad+\omega^\triv.
\end{equation}

The linear map $A^\omega:\mathfrak{qg}^\#\longrightarrow \Omega^1(B)$ can be interpreted as the {\it non--commutative gauge potential} of $\omega$. The bijection of equation (\ref{trv.f2}) extends naturally to the curvature by  (see Lemma $6.12$ of reference \cite{micho2}) 
\begin{equation}
    \label{trv.f3}
    R^{\omega}=(F^\omega\otimes \id_H)\circ \ad,
\end{equation}
where $F^\omega:\mathfrak{qg}^\#\longrightarrow \Omega^2(B)$ is the linear map defined as 
\begin{equation}
    \label{trv.f3.1}
    F^\omega:=dA^\omega-\langle A^\omega,A^\omega\rangle.
\end{equation}
The map $F^\omega$ can be interpreted as the {\it non--commutative field strength} of $\omega$. Notice that $F^{\omega^\triv}=0$ (because $A^{\omega^\triv}=0$) and it follows that $\omega^\triv$ is flat. 

Furthermore, Lemma $6.12$ of reference \cite{micho2} shows that 
\begin{equation}
    \label{trv.f4.1}
    D^\omega(\tau) = (D^\omega_B(L^\tau)\otimes \id_H)\circ \ad \qquad \mbox{ with }\qquad D^\omega_B(L^\tau):=dL^\tau-(-1)^{k}[L^\tau,A^\omega],
\end{equation}
for every $\tau$ $\in$ $\Mor(\ad,\Delta_\Hor)$ such that $\Im(\tau)\subseteq \Hor^k P$, where $[-,-]$ is defined in equation (\ref{trans}). In other words, $D^\omega_B(L^\tau)$ is the element of $\mathrm{Hom}(\mathfrak{qg}^\#,\Omega^1(B))$ associated with $D^\omega(\tau)$ $\in$ $\Mor(\ad,\Delta_\Hor)$ under the bijection of equation (\ref{trv.f1}). In the same way, we get that
\begin{equation}
    \label{lonecesito23}
     \widehat{D}^\omega(\tau) = (\widehat{D}^\omega_B(L^\tau)\otimes \id_H)\circ \ad \qquad \mbox{ with }\qquad \widehat{D}^\omega_B(L^\tau):=(\wedge\circ D^\omega_B\circ \wedge)(L^\tau),
\end{equation}
for every $\tau$ $\in$ $\Mor(\ad,\Delta_\Hor)$ such that $\Im(\tau)\subseteq \Hor^k P$.

\begin{Proposition}
    \label{prop.trv}
    For every $\tau$ $\in$ $\Mor(\ad,\Delta_\Hor)$ such that $\Im(\tau)\subseteq \Hor^k P$, the map 
    $$S^\omega_B(L^\tau):=\langle A^\omega,L^\tau\rangle-(-1)^k\langle L^\tau,A^\omega\rangle-(-1)^k[L^\tau,A^\omega]$$ is the element of $\mathrm{Hom}(\mathfrak{qg}^\#,\Omega^1(B))$ associated with $S^\omega(\tau)$ $\in$ $\Mor(\ad,\Delta_\Hor)$ under the bijection of equation (\ref{trv.f1}). In other words, we have 
     $$S^\omega(\tau)=(S^\omega_B(L^\tau)\otimes\id_H)\circ \ad.$$ 
\end{Proposition}

\begin{proof}
  Let $\theta$ $\in$ $\mathfrak{qg}^\#$ and  $\Theta(\theta)=\displaystyle \sum^m_{i,j=1}\theta_i\otimes \theta'_j$ for some $\theta_i$, $\theta'_j$ $\in$ $\mathfrak{qg}^\#$. Then 
  \begin{eqnarray*}
 \langle \omega,\tau\rangle (\theta)=
      \sum^m_{i,j=1} \omega(\theta_i)\,\tau(\theta'_j)
      &=&\sum^m_{i,j=1} (A^\omega(\theta^{(0)}_i)\otimes \theta^{(1)}_i + \omega^\triv(\theta_i))\cdot (L^\tau(\theta'^{(0)}_j)\otimes \theta'^{(1)}_j) 
      \\
      &=&
\sum^m_{i,j=1} A^\omega(\theta^{(0)}_i)L^\tau(\theta'^{(0)}_j)\otimes \theta^{(1)}_i\theta'^{(1)}_j+\omega^\triv(\theta_i)\, \tau(\theta'_j).     
  \end{eqnarray*} 
 Notice that $\displaystyle \sum^m_{i,j=1}\omega^\triv(\theta_i)\, \tau(\theta'_j)=\langle \omega^\triv,\tau\rangle(\theta).$ Since $\Theta$ $\in$ $\Mor(\ad,\ad\otimes \ad)$, it follows that
\begin{eqnarray*}
    \sum^m_{i,j=1} A^\omega(\theta^{(0)}_i)L^\tau(\theta'^{(0)}_j)\otimes \theta^{(1)}_i\theta'^{(1)}_j&=& ((m\circ (A^\omega\otimes L^\tau))\otimes \id_H) \ad^{\otimes 2}(\Theta(\theta))
    \\
    &=&
    ((m\circ (A^\omega\otimes L^\tau)\circ \Theta)\otimes \id_H)\ad(\theta)
    \\
    &=&
    (\langle A^\omega,L^\tau\rangle \otimes \id_H)\ad(\theta),
\end{eqnarray*}
 where $m:\Omega^\bullet(B)\otimes \Omega^\bullet(B)\longrightarrow \Omega^\bullet(B)$ is the product map; so $$\langle \omega,\tau\rangle=(\langle A^\omega,L^\tau\rangle \otimes \id_H)\circ \ad + \langle \omega^\triv,\tau\rangle.$$
 Similarly, we have
 $$\langle \tau,\omega\rangle=(\langle L^\tau,A^\omega\rangle \otimes \id_H)\circ \ad + \langle \tau,\omega^\triv\rangle \quad \mbox{ and }\quad [\tau,\omega ]=([L^\tau,A^\omega] \otimes \id_H)\circ \ad + [\tau,\omega^\triv].$$ Thus
 \begin{eqnarray*}
S^\omega(\tau) &=&((\langle A^\omega,L^\tau\rangle-(-1)^k\langle L^\tau,A^\omega\rangle-(-1)^k[L^\tau,A^\omega])\otimes \id_H)\circ \ad +
     S^{\omega^\triv}(\tau)
     \\
     &=&
     (S^\omega_B(L^\tau)\otimes \id_H)\circ \ad+S^{\omega^\triv}(\tau)
 \end{eqnarray*}
 However, $S^{\omega^\triv}(\tau)=0$ because $\omega^\triv$ is regular.
\end{proof}

Of course, the following proposition can be proven in a similar way
\begin{Proposition}
    \label{prop.trv.1}
    For every $\tau$ $\in$ $\Mor(\ad,\Delta_\Hor)$ such that $\Im(\tau)\subseteq \Hor^k P$, the map 
    $$\widehat{S}^\omega_B(L^\tau):= (\wedge \circ S^{\omega}_B \circ \wedge)(L^\tau)$$ is the element of $\mathrm{Hom}(\mathfrak{qg}^\#,\Omega^1(B))$ associated with $\widehat{S}^\omega(\tau)$ $\in$ $\Mor(\ad,\Delta_\Hor)$ under the bijection of equation (\ref{trv.f1}). In other words, we have 
     $$\widehat{S}^\omega(\tau)=(\widehat{S}^\omega_B(L^\omega)\otimes\id_H)\circ \ad.$$ 
\end{Proposition}

It follows that
 
\begin{Corollary}
\label{coro1}
    For every $\tau$ $\in$ $\Mor(\ad,\Delta_\Hor)$ such that $\Im(\tau)\subseteq \Hor^k P$, the twisted covariant derivatives are expressed as
    $$DS^\omega(\tau)=(DS^\omega_B(L^\tau)\otimes \id_H)\circ \ad \qquad \mbox{ with }\qquad DS^\omega_B(L^\tau):=D^\omega_B(L^\tau)-S^\omega_B(L^\tau)$$ and
    $$\widehat{DS^\omega}(\tau)=(\widehat{DS^\omega}_B(L^\tau)\otimes \id_H)\circ \ad \qquad \mbox{ with }\qquad \widehat{DS^\omega}_B(L^\tau):=\widehat{D}^\omega_B(L^\tau)-\widehat{S}^\omega_B(L^\tau).$$
\end{Corollary}

Consider $\R^4$ and its space of $\C$--valued smooth functions $C^\infty_\C(\R^4)$. By choosing a $4\times 4$ antisymmetric matrix $(\Lambda^{\mu\nu})$ $\in$ $M_4(\R)$, we define a new product: for every $f$, $h$ $\in$ $C^\infty_\C(\R^4)$ 
\begin{equation}
    \label{trv.f4}
    f\cdot h:=m'\circ \mathrm{exp}\left( \frac{i\,\Lambda^{\mu\nu}}{2} \frac{\partial}{\partial x^\mu}\otimes \frac{\partial}{\partial x^\nu} \right)(f\otimes h),
\end{equation}
where $m'$ denotes the usual product on $C^\infty_\C(\R^4)$ and we have used Einstein summation convention with $i=\sqrt{-1}$. With this new product, we get a  non--commutative unital $\ast$--algebra $B$ (\cite{05}), where the unit is $\mathbbm{1}(x)=1$ for all $x$ $\in$ $\R^4$ and the $\ast$ operation is the complex conjugate. The $\ast$--algebra $B$ is given the name of Moyal--Weyl algebra (\cite{05}).

 Equation (\ref{trv.f4}) is easily extended to $\Omega^\bullet_\C(\R^4)$, the space of $\C$--valued differential forms of $\R^4$. Indeed, we define the operator $$\frac{\partial}{\partial x^\gamma}: \Omega^\bullet_\C(\R^4)\longrightarrow \Omega^\bullet_\C(\R^4)$$ given by  $$\frac{\partial}{\partial x^\gamma}(\mu):=\mathcal{L}_{ \frac{\partial}{\partial x^\gamma}}(\mu),$$ where $\mathcal{L}_{ \frac{\partial}{\partial x^\gamma}}$ is the Lie derivative with respect to $\displaystyle \frac{\partial}{\partial x^\gamma}$. In this way, equation (\ref{trv.f4}) can be extended to $\C$--valued differential forms of $\R^4$. For example, since  $$\displaystyle \mathcal{L}_{\frac{\partial }{\partial x^\mu}\frac{\partial }{\partial x^\gamma}}:=\mathcal{L}_{\frac{\partial }{\partial x^\mu}}\circ \mathcal{L}_{\frac{\partial }{\partial x^\gamma}}\qquad  \mbox{ and  }\qquad \mathcal{L}_{ \frac{\partial}{\partial x^\gamma}} dx^{\mu}=0,$$ it follows that 
\begin{equation}
\label{trv.f5}
dx^\mu\cdot f=f\cdot dx^\mu, \qquad dx^\mu\wedge dx^\nu=-dx^\nu \wedge  dx^\mu 
\end{equation}
for every $f$ $\in$ $B$ and $\mu$, $\nu=0,1,2,3$. With this new product of differential forms, we get a non--commutative graded differential $\ast$--algebra (\cite{05}), which we will denote by $\Omega^\bullet(B)$. 

For a given quantum group $\G$ with its associated $\ast$--Hopf algebra $H^\infty=(H,\cdot,\mathbbm{1},\Delta,S,\epsilon,\ast)$, let $$\zeta_{\triv}$$ be a trivial quantum principal $\G$--bundle over the Moyal--Weyl algebra $B$. Let $(\Gamma,d)$ be a bicovariant $\ast$--FODC over $H$ such that $\mathfrak{qg}^\#$ is a finite--dimensional $\C$--vector space, and consider the corresponding trivial differential calculus $$\Omega^\bullet(P)=\Omega^\bullet(B)\otimes \Gamma^\wedge\quad  \mbox{ and }\quad \Delta_{\Omega^\bullet(P)}=\id_{\Omega^\bullet(B)}\otimes \Delta $$ on $\zeta_{\triv}$.

Let us define the following left quantum Riemmanian metric on $B$
\begin{equation}
\label{trv.f6}
\{\langle-,-\rangle^k_\l: \Omega^k(B)\times \Omega^k(B) \longrightarrow B\mid k=0,1,2,3,4\}
\end{equation}
such that $\langle f ,h\rangle^0_\l= f\cdot h^\ast$ and for $k=1,2,3,4$ it is the usual euclidean metric of the de Rham differential algebra of $\R^4$. For example $$\displaystyle \langle \sum^3_{\mu=0}f_\mu\,dx^\mu ,\sum^3_{\nu=0}h_\nu\,dx^\nu\rangle^1_\l= \sum^3_{\mu,\nu=0} f_\mu\cdot h^\ast_\nu$$ and $$\langle f\,\dvol,h\,\dvol\rangle^4_\l=f\cdot h^\ast,$$ where 
\begin{equation}
    \label{trv.f7}
    \dvol:=dx^0\wedge dx^1\wedge dx^2\wedge dx^3
\end{equation}
is the quantum $4$--volume form. Furthermore, there is a left quantum Hodge operator which is defined as in differential geometry for $\R^4$ with the Euclidean metric (\cite{gtvp}). For example
\begin{equation*}
    \label{trv.f8}
    \begin{aligned}
    \star_\l dx^0=dx^1\wedge dx^2\wedge dx^3, \quad \star_\l dx^1=-dx^0\wedge dx^2\wedge dx^3,\\
    \star_\l dx^2=dx^0\wedge dx^1\wedge dx^3,\quad \star_\l dx^3=-dx^0\wedge dx^1\wedge dx^2
    \end{aligned}
\end{equation*}
and
\begin{equation*}
    \begin{aligned}
    \star_\l dx^0 \wedge dx^1=dx^2\wedge dx^3, \quad \star_\l dx^0 \wedge dx^2=-dx^1\wedge dx^3, \\ \star_\l dx^0 \wedge dx^3=dx^1\wedge dx^2,\quad
    \star_\l dx^1 \wedge dx^2=dx^0\wedge dx^3,\\ \quad \star_\l dx^1 \wedge dx^3=-dx^0\wedge dx^2, \quad \star_\l dx^2 \wedge dx^3=dx^0\wedge dx^1.
    \end{aligned}
\end{equation*}
For the quantum integral, we are going to take the usual integration on $\Omega^4(B)$. As in differential geometry, the quantum integral is defined only for measurable elements.

From now on and as in differential geometry, we will restrict ourselves to $\Omega^\bullet(B)'$, the space of differential forms with coefficient--functions in the complex Schwartz space of $\R^4$. This assumption is to guarantee that $d\mu$, $d^{\star_\l}\mu$, $\mu\cdot \eta$ are integrable provide that $\mu$ and $\eta$ are integrable. 

In this sense, we restrict our attention to quantum principal connections $\omega$ for which $
\Im(A^\omega)\subseteq \Omega^1(B)'.$  This condition is easily satisfied, since $A^\omega$ is a linear map with finite--dimensional domain. Then, by equation~(\ref{trv.f3.1}), it follows that $\Im(F^\omega)\subseteq \Omega^2(B)'.$  As in differential geometry, this restriction is necessary because not every element of $\Omega^4(B)$ is integrable and because of the physical requirement that all fields must vanish at infinity. Consequently, to show that this class of quantum principal bundles is suitable for our theory, it suffices to prove the adjointability of the operators $d^{S^{\omega}}$ and $d^{\widehat{S}^{\omega}}$.

Consider the left quantum Hodge inner product $\langle-|-\rangle_\l$. Let $$\eta=f\, dx^\mu\; \in\; \Omega^1(B)',\quad \sigma= h\,dx^{j_1}\wedge\cdots \wedge dx^{j_k}\; \in\; \Omega^k(B)'$$ and  $$0\not=\kappa=l\,dx^\mu \wedge dx^{j_1}\wedge\cdots \wedge dx^{j_k}\; \in\; \Omega^{k+1}(B)'$$ for a fixed $\mu$ $\in$ $\{0,1,2,3\}$ and  $0\leq j_1<\cdots <j_k\leq 3$. Then $\eta\cdot\sigma=f\cdot h\;dx^\mu\wedge dx^{j_1}\wedge\cdots \wedge dx^{j_k}$; so we have $$\langle \eta\cdot \sigma\mid \kappa\rangle_\l = \int_{\R^4}  f\,h\, l^\ast\,\dvol= \int_{\R^4}  h\,(f^\ast\,l)^\ast\,\dvol=\langle \sigma\mid \alpha_\eta(\kappa)\rangle_\l,$$ where $\alpha_\eta(\kappa)= f^\ast\cdot l\,dx^{j_1}\wedge\cdots \wedge dx^{j_k}$. Similarly, we have $$\langle \sigma\cdot\eta\mid \kappa\rangle_\l = \int_{\R^4} h\,f\,l^\ast\,\dvol=\int_{\R^4} h\,(l^\ast\,f)^\ast\,\dvol=\langle \sigma\mid \beta_\eta(\kappa)\rangle_\l,$$ where $\beta_\eta(\kappa)=l^\ast\cdot f\,dx^{j_1}\wedge\cdots \wedge dx^{j_k}$. By linearity, we can conclude that for every $\eta$ $\in$ $\Omega^1(B)'$, $\sigma$ $\in$ $\Omega^k(B)'$, $\kappa$ $\in$ $\Omega^{k+1}(B)'$, there exist $$\alpha_\eta(\kappa),\;  \beta_\eta(\kappa)  \;\in\; \Omega^{k}(B)'$$ such that
\begin{equation}
    \label{conmut}
    \langle\eta\cdot \sigma\mid \kappa\rangle_\l=\langle \sigma\mid \alpha_\eta(\kappa)\rangle_\l,\qquad \langle\sigma\cdot \eta\mid\kappa\rangle_\l=\langle \sigma\mid \beta_\eta(\kappa)\rangle_\l.
\end{equation}

\begin{Proposition}
        \label{ds2}
    The operators $d^{S^\omega}$, $d^{\widehat{S}^\omega}$ are formally adjointable.
\end{Proposition}
\begin{proof}
Consider the decomposition $$\ad\cong \bigoplus^u_{m=1}\delta^{V_m}$$ for some $u$ $\in$ $\N$, with $\delta^{V_m}$ $\in$ $\T$. Since this decomposition is also orthonormal (\cite{woro1}), if $\beta_m=\{ \theta_{m_i}\}^{n_{m}}_{m_i=1}$ is an orthonormal linear basis of $V_m$, then $$\beta=\beta_{m_1}+\cdots + \beta_{m_u}=\{\theta_1,...,\theta_n \}$$ is an orthonormal linear basis of $\mathfrak{qg}^\#$ and by Proposition $5.1$ of reference \cite{sald2}, the linear maps $\{ T^\l_k:\mathfrak{qg}^\# \longrightarrow B\otimes H \}$ given by
$$T^\l_k(\theta_{m_i})=\mathbbm{1}\otimes g^{V_m}_{k\,m_i} $$ (where the elements $\{ g^{V_m}_{k\,m_i}\}^{n_{m}}_{k,m_i=1}$ are the ones of Theorem \ref{rep} for $\delta^{V_m}$) form an orthonormal $B$--basis of the associated qvb $E^{\mathfrak{qg}^\#}_\l$, with respect to $(-,-)_\l$. Thus, by Proposition \ref{prop.trv} it follows that $$d^{S^\omega}\varphi=\sum_{j} S^\omega_B(L^\varphi)(\theta_j)\otimes_B T^\l_j$$ for every $\varphi=\displaystyle\sum_{j}\sigma_j\otimes_B T^\l_j$  $\in$ $\Omega^k(B)'\otimes_B E^{\mathfrak{qg}^\#}_\l$, where the linear map $$L^\varphi:\mathfrak{qg}^\#\longrightarrow \Omega^k(B)'$$ is the one associated with the element  $\displaystyle\sum_{j}\sigma_j\, T^\l_j$ $\in$ $\Mor(\ad,\Delta_\Hor)$ under the bijection of equation (\ref{trv.f1}). Since 
$$\Theta(\theta_j)=\sum_{i,m}a^{im}_j\,\theta_i\otimes \theta_m,\qquad \mathrm{c}^T(\theta_j)=\sum_{i,m}b^{im}_j\,\theta_i\otimes \theta_m $$ with $a^{im}_j,$ $b^{im}_j$ $\in$ $\C$, we obtain 
\begin{eqnarray*}
    S^\omega_B(L^\varphi)(\theta_j)&=& (\langle A^\omega,L^\varphi\rangle-(-1)^k\langle L^\varphi,A^\omega\rangle-(-1)^k[L^\varphi,A^\omega])(\theta_j) 
    \\
    &=&
    \sum_{i,m}a^{im}_j A^\omega(\theta_i)\cdot \sigma_m-(-1)^{k}(a^{mi}_j+b^{mi}_j)\sigma_i\cdot A^\omega(\theta_m)
    \\
    &=&
    \sum_m r^m_j\cdot \sigma_m-(-1)^{k}\sigma_m \cdot s^m_j,
\end{eqnarray*}
where $$r^m_j=\displaystyle \sum_i a^{im}_j A^\omega(\theta_i),\;\; s^m_j=\displaystyle \sum_i(a^{mi}_j+b^{mi}_j)A^\omega(\theta_i)\;\in\;\Omega^1(B)'.$$ So $$d^{S^\omega}\varphi=\sum_{j}\left( \sum_m r^m_j\cdot \sigma_m-(-1)^{k}\sigma_m \cdot s^m_j\right)\otimes_B T^\l_j.$$

In this way, for every $\psi=\displaystyle\sum_{j}\kappa_j\otimes_B T^\l_j$ $\in$ $\Omega^{k+1}(B)'\otimes_B E^{\mathfrak{qg}^\#}_\l$, we get

\begin{eqnarray*}
\langle  d^{S^\omega}\varphi  \mid  \psi \rangle^\bullet_\l=\sum_{j}\langle S^\omega_B(L^\varphi)(\theta_j)\mid \kappa_j\rangle_\l
&=&\sum_{m,j} \langle r^m_j\cdot \sigma_m-(-1)^{k}\sigma_m \cdot s^m_j\mid \kappa_j\rangle_\l
    \\
    &=& 
    \sum_{m,j} \langle \sigma_m\mid \alpha_{r^m_j}(\kappa_j)-(-1)^{k}\beta_{s^m_j}(\kappa_j) \rangle_\l
    \\
    &=&
    \sum_{m}\langle \sigma_m\mid S^{\omega\star}_B(L^\psi)(\theta_m)  \rangle_\l
    \\
    &=& 
    \langle  \varphi  \mid  d^{S^\omega\star}\psi \rangle^\bullet_\l,
\end{eqnarray*}
where the linear map $$L^\psi:\mathfrak{qg}^\#\longrightarrow \Omega^{k+1}(B)'$$ is the one associated with the element  $\displaystyle\sum_{j}\kappa_j\, T^\l_j$ $\in$ $\Mor(\ad,\Delta_\Hor)$ under the bijection of equation (\ref{trv.f1}); and
$S^{\omega\star}_B(L^\psi)$ is the element of $\mathrm{Hom}(\mathfrak{qg}^\#,\Omega^1(B))$ given by $$S^{\omega\star}_B(L^\psi)(\theta_m):= \sum_j \alpha_{r^m_j}(\kappa_j)-(-1)^{k}\beta_{s^m_j}(\kappa_j).$$ Finally, the operator $d^{S^\omega\star}$ is the formal adjoint operator of $d^{S^\omega}$ and is defined by 
\begin{equation}
    \label{sop1}
    d^{S^\omega\star}\psi=\sum_m  S^{\omega\star}_B(L^\psi)(\theta_m) \otimes_B T^\l_m.
\end{equation}
Of course, the linear map $S^{\omega\star}_B(L^\psi)$ $\in$ $\mathrm{Hom}(\mathfrak{qg}^\#,\Omega^1(B))$ is the one associated with the element $\Upsilon^{-1}_{\mathfrak{qg}^\#}(d^{S^\omega\star}(\psi))$ $\in$ $\Mor(\ad, \Delta_\Hor)$ under the bijection of equation (\ref{trv.f1}).

In a completely analogous way, it can be proven that $d^{\widehat{S}^\omega}$ is formally adjointable with respect to $\langle-|-\rangle^\bullet_\r$. We will denote this operator as
\begin{equation}
    \label{sop2}
    d^{\widehat{S}^\omega\star}\phi=\sum_m T^\r_m   \otimes_B \widehat{S}^{\omega\star}_B(L^\phi)(\theta_m),
\end{equation}
with $\phi=\displaystyle \sum_j T^\r_j\otimes_B \upsilon_j$ $\in$ $E^{\mathfrak{qg}^\#}_\r\otimes_B \Omega^{k+1}(B)'$ and $\{T^\r_m\}$ the right $B$--generators of $\Mor(\ad,\Delta_P)$.
\end{proof}

According to Theorem \ref{4.2.15} and equations (\ref{trv.f4.1}), (\ref{lonecesito23}), it follows that
\begin{equation}
    \label{diff1}
    d^{\nabla^\omega_{\mathfrak{qg^\#}}\star}\psi= \sum_m D^{\omega\star}_B(L^\psi)(\theta_m)\otimes_B T^\l_m, \qquad d^{\widehat{\nabla}^\omega_{\mathfrak{qg^\#}}\star}\phi=\sum_m T^\r_m\otimes_B \widehat{D}^{\omega\star}_B(L^\phi)(\theta_m),
\end{equation}
with $$D^{\omega\star}_B(L^\psi):=(-1)^{k+1}(\star^{-1}_\l\circ \ast \circ D^\omega_B\circ \ast \circ \star_\l)(L^\psi), \qquad \widehat{D}^{\omega\star}_B(L^\phi):=(-1)^{k+1}(\star^{-1}_\r\circ \widehat{D}^\omega_B\circ \star_\r)(L^\phi).$$

\begin{Corollary}
    \label{coro2}
    The operators $d^{DS^\omega\star}$, $d^{\widehat{DS^\omega}\star}$ are given by $$d^{DS^\omega\star}\psi= \sum_m DS^{\omega\star}_B(L^\psi)(\theta_m)\otimes_B T^\l_m, \qquad d^{\widehat{DS^\omega}\star}\phi=\sum_mT^\r_m\otimes_B \widehat{DS}^{\omega \star}_B(L^\phi)(\theta_m),$$ where $$ DS^{\omega\star}_B(L^\psi):=D^{\omega\star}_B(L^\psi)-S^{\omega\star}_B(L^\psi),\qquad \widehat{DS}^{\omega\star}_B(L^\phi):=\widehat{D}^{\omega\star}_B(L^\phi)-\widehat{S}^{\omega\star}_B(L^\phi).$$
\end{Corollary}

In summary, Proposition \ref{ds2} (and hence, Corollary \ref{coro2}) shows that our theory can be applied on this class of quantum principal $\G$--bundles and it is worth noting that any quantum group $\G$ can be used, and any $\ast$--FODC on $H$ can also be used, provided that $\mathfrak{qg}^{\#}$ is finite--dimensional. 

In reference \cite{sald5}, the reader can find a concrete example of this class of qpb's for the canonical quantum group associated with the Lie group $U(1)$. In other words, the study presented in reference \cite{sald5} is a study of the electromagnetic theory on the Moyal--Weyl algebra.

An interesting result of the theory is the fact that, for an electromagnetic field in vacuum (without any other fields), the presence of the operator $S^\omega$ in both the non--commutative geometrical Bianchi identity and the non--commutative geometrical Yang--Mills equation implies that {\it the non--commutative Maxwell equations in vacuum are not equal to zero}. In other words, the operator $S^\omega$ produces {\it electric and magnetic charges and currents in vacuum}, as the reader can verify in Theorems 3.2, 3.6 of reference \cite{sald5} or equivalently, equations (3.23), (3.25) of \cite{sald5}.

The study of the Yang--Mills--scalar matter theory within the framework developed in this paper, together with its comparison with the existing results in the literature for other sectors of the Standard Model (such as the electroweak sector and the strong sector) or the full Standard Model or Grand Unification Theories over the Moyal--Weyl algebra, will be the subject of forthcoming publications.

It is worth mentioning that we have chosen the Moyal--Weyl algebra $B$ as the quantum base space of the trivial qpb’s used to illustrate our theory in this section solely because this algebra is one of the most widely used models of non--commutative space--time in quantum gravity~\cite{05}. In principle, any other non--commutative algebra $B'$ may be used, provided that $B'$ is equipped with a differential calculus satisfying Definition~\ref{a.2.1}, together with a quantum Hodge operator and sufficiently well--understood commutation relations. For example, one could consider the $\kappa$--Minkowski space, which is also one of the most widely used models of non--commutative space--time in quantum gravity \cite{06}.

\section{Concluding Comments}

This work was developed under Durdevich's theory of qpb's and as we have mentioned in Section 2.2, there are mathematical reason to consider this approach and not Brzezi\'nski--Majid formulation \cite{libro}. In order of importance, these reasons are the following.
\begin{enumerate}
    \item The curvature $R^\omega$ of every qpc $\omega$ is a linear map with domain $\mathfrak{qg}^\#$, as the {\it dualization} of the classical case indicates.
    \item The curvature is a basic quantum differential form of type $\ad$, as in the {\it classical} case, i.e.,  $R^\omega$ $\in$ $\Mor(\ad,\Delta_\Hor)$. In terms of a {\it physical interpretation}, we can consider $R^\omega$ as a {\it non--commutative geometrical tensor field}.
    \item When $\mathfrak{qg}^\#$ is a finite--dimensional $\C$--vector space, we can apply all the theory of associated qvb's for the $H$--corepresentation $\ad$. Particularly, we can use the inner products of equations (\ref{4.f2.23}), (\ref{4.f2.27}) and take the square norm of $R^\omega$, as in differential geometry. 
    \item The existence of the formal adjoint operators of $d^{\nabla^\omega_V}$ and $d^{\widehat{\nabla}^\omega_V}$ (Theorem~\ref{4.2.15}) is a consequence of Theorem~\ref{fgs}. As far as the author knows, with exception of homogeneous qpb's and the Chern connection (\cite{obauchalla2}),  there is no analog of Theorem~\ref{fgs} in the Brzezi\'nski--Majid formulation of quantum principal bundles. 
    
    Although, in principle, it should be possible to formulate such an analog with the same generality than Theorem~\ref{fgs}, the absence of this result complicates the corresponding procedure within the Brzezi\'nski--Majid framework.
    \item The operator $S^\omega$, which has no counterpart in Brzezi\'nski--Majid   formulation of quantum principal bundles, allows one to obtain the correct field equations for $R^\omega$, since, as shown in Theorem \ref{6.1.5}, this operator appears naturally when one considers the variation $R^{\omega+t\,\lambda}$.
\end{enumerate}

In contrast, in Brzezi\'nski--Majid formulation of qpb's
\begin{enumerate}
    \item The curvature $r^\omega$ of a qpc $\omega$ is defined on $H$, or equivalently on $\Ker(\epsilon)$, by $$g\longmapsto r^\omega(g)=d\omega(\pi(g))+\omega(\pi(g^{(1)}))\,\omega(\pi(g^{(2)})).$$ 
    \item The curvature $r^\omega$ is not a basic quantum form of type $\ad$, but rather a basic quantum form of type $\Ad$~(\cite{libro}). Therefore, from a \emph{physical perspective}, $r^\omega$ is not exactly a \emph{non--commutative geometrical tensor field}. The map $R^\omega$ appears to be a better choice for use in non--commutative geometrical gauge theory.
    \item  In general, the curvature $r^\omega$ is defined on an infinite--dimensional $\C$--vector space. Consequently, in order to define its squared norm, it is necessary to introduce an appropriate and general notion of convergence and ensure that the square norm of $r^\omega$ converges. If one wishes to avoid this complication, one may restrict attention to the cases in which $\Ker(\epsilon)$ is finite--dimensional; however, this would considerably limit the range of examples to which the theory applies.
    \item Since, in general, the Brzezi\'nski--Majid formulation of qpb’s does not deal with elements of degree~$2$ or higher, there is no general theory for the variation $r^{\omega+t\,\lambda}$, nor for the operators $d^{\nabla^\omega_V}$, $d^{\widehat{\nabla}^\omega_V}$. This complicates very much the procedure in this formulation.
\end{enumerate}

Of course, there are other important aspects of Durdevich's theory that enabled the formulation of our theory. For instance, the use of universal differential envelope $\ast$--calculus as quantum differential forms of $H$. This space not only allows us to extend the $\ast$--Hopf algebra structure of $H$ to $\Gamma^{\wedge}$, but it is also maximal with this property (\cite{micho1}). Moreover, it generalizes the graded differential  $\ast$--algebra of $\C$--valued differential forms of a compact matrix Lie group (\cite{appendix}). The universal differential envelope $\ast$--calculus and the extension of the coproduct $\Delta$ play a central role in defining the notion of differential calculus on qpb's, and in the definition of the $\ast$--algebras $\Hor^\bullet P$, $\Vert^\bullet P$, $\Omega^\bullet(M)$ in all degrees (not only for degree $1$), as well as in the definition of the $H$--corepresentation $\Delta_\H$ .

It is worth noticing that only missing ingredient for addressing the full geometry of particle physics is a formulation of spinorial matter. This is not straightforward. However, an interesting starting point is to follow the line of research of A. Connes in reference \cite{con}, namely, to consider a qpb with quantum base $B$ given by a spectral triple and $\Omega^\bullet(B)$ given by Connes' space of differential forms, and then to repeat the construction developed in this paper. A first step in this direction is presented in reference  \cite{saldcon}, in which we \emph{unify} Connes' formulation of fermionic matter with the field matter framework of this paper and we also compare Connes' formulation of Yang--Mills theory with the one presented here in a concrete example.

Now, let us talk about the operator $S^\omega$. First of all, it is worth emphasizing that an auxiliary map is required (the embedded differential $\Theta$) to define  $R^\omega$ and this implies the existence of the operator $S^\omega$. As we have seen, this operator  naturally appears in the non--commutative geometrical Bianchi identity and in the non--commutative geometrical Yang--Mills equation. In the \emph{dualization of the classical case via the pull-back}, every $\omega^\#_\mathrm{class}$ is regular and hence  $S^{\omega^\#_\mathrm{class}}=0$ (\cite{micho2}). So, the operator $S^\omega$ is identically zero in differential geometry.

Moreover, notice that we have asked for the existence of the adjoint operators of $d^{S^{\omega}}$ and $d^{\widehat{S}^{\omega}}$, not for an specific form of them. This is because, for the qpb and the differential calculus used in Sections 2, 3 of reference \cite{sald5}, we show that $$ d^{S^{\omega}\star}=(-1)^{k+1} \star^{-1}_\l\circ\, d^{S^{\omega}} \circ \star_\l,\qquad d^{\widehat{S}^{\omega}\star}=\ast \circ d^{S^{\omega}\star} \circ \ast;$$ however for the quantum bundle of  Section 4 we get that  (\cite{sald5}) 
\begin{equation}
\label{61}
    d^{S^{\omega}\star}\not=(-1)^{k+1} \star^{-1}_\l\circ\, d^{S^{\omega}} \circ \star_\l.
\end{equation}
Thus, we conclude that even for the same qpb with different differential calculus, there is not a general formula for the adjoint operators of $d^{S^{\omega}}$ and $d^{\widehat{S}^{\omega}}$ .

To conclude this section, we now discuss the quantum gauge group and the models presented in Section 4.  This paper is formulated under the \emph{philosophy} that, the actions that measure the squared norm of the curvature and the squared norm of the induced linear connection are the central mathematical objects in Yang--Mills scalar matter theory. The group of symmetries of these actions is a secondary object, determined by the actions. 

This is a non--standard way to proceed. In physics, it is common to first choose a Lie group $G$ and then construct an action (by determining a principal $G$--bundle) in such a way that $G$ becomes the symmetry group (or a subgroup of it) of the action. In this paper, however, we first choose a quantum principal $\G$--bundle, then consider the corresponding actions, and finally characterize the symmetry groups of these actions, since not every quantum gauge transformation is a symmetry.

 Nevertheless, as discussed in Section 4.1, the Definitions~\ref{6.1.3} and~\ref{6.1.13} provide proper generalizations of the symmetry group appearing in differential geometry. It is worth noticing that in reference \cite{saldcon} we define the quantum gauge group of the fermionic models in the same way as here, for the same reasons.

\begin{appendix}
\section{}
In this appendix we are going to show the explicit proof of the following statement.
\begin{Proposition}
\label{a.1}
Let $\delta^V$ $\in$ $\Obj(\Rep_{H})$. Then the maps $\langle-|-\rangle^ \bullet_\l$, $\langle-|-\rangle^ \bullet_\r$ (see equations (\ref{4.f2.23}), (\ref{4.f2.27})) are well--defined inner products.
\end{Proposition}
\begin{proof}
The only part of the statement that it is not trivial is
the positive--definiteness; so let us proceed to prove it.  Let $\psi=\displaystyle\sum^m_{k=1} \mu_k\otimes_B T_k$ $\in$ $\Omega^\bullet(B)\otimes_B E^V_\l$ such that $\langle\psi\mid\psi\rangle^\bullet_\l=0$. Thus, $\tau:=\Upsilon^{-1}_{V}(\psi)=\displaystyle \sum^m_{k=1} \mu_k\,T_k$ $\in$ $\Mor(\delta^V,\Delta_\H)$ and by equation (\ref{3.f7}) we have $$\psi=\displaystyle\sum^m_{k=1} \mu_k\otimes_B T_k=\sum^{d_{V}}_{k=1} \mu^\tau_k\otimes_B T^\l_k,$$ where $\mu^\tau_k=\displaystyle \sum^{n_{V}}_{i=1}\tau(e_i)\,x^{V\,\ast}_{ki}$. Hence, by equations (\ref{ashcas}), (\ref{generators}) we obtain
 \begin{eqnarray*}
\langle\psi,\psi\rangle^\bullet_\l =  \sum^{d_{V}}_{k,j=1} \langle\mu^\tau_k\otimes_B T^\l_k , \mu^\tau_j\otimes_B T^\l_j \rangle^\bullet_\l
  &= &
  \sum^{d_{V}}_{k,j=1}
  \langle   \mu^\tau_k\, (T^\l_k, T^\l_j)_\l,  \mu^\tau_j\rangle_\l
  \\
  &= &
  \sum^{d_{V},n_{V}}_{k,j,l=1}\langle \mu^\tau_k\,x^{V}_{kl}\,x^{V\,\ast}_{jl},  \mu^\tau_j\rangle_\l 
  \\
  &= &
\sum^{d_{V},n_{\alpha}}_{k,j,l,i=1}\langle \tau(e_i)\,x^{V\,\ast}_{ki}\,x^{V}_{kl}\,x^{V\,\ast}_{jl},  \mu^\tau_j\rangle_\l 
  \\
  &= &
\sum^{d_{V},n_{V}}_{j,l,i=1}\langle \tau(e_i)\,\delta_{il}\,x^{V\,\ast}_{jl},  \mu^\tau_j\rangle_\l
\\
  &= &
 \sum^{d_{V},n_{V}}_{j,i=1}\langle \tau(e_i)\,\,x^{V\,\ast}_{ji},  \mu^\tau_j\rangle_\l 
 \\
  &= & 
 \sum^{d_{V}}_{j=1}\langle \mu^\tau_j,  \mu^\tau_j\rangle_\l.
\end{eqnarray*}
In this way $$0=\langle\psi\mid\psi\rangle^\bullet_\l=\int_B \langle\psi,\psi\rangle^\bullet_\l \,\dvol= \sum^{d_{V}}_{j=1} \int_B\langle \mu^\tau_j,  \mu^\tau_j\rangle_\l\,\dvol,$$ which implies that $\langle \mu^\tau_j,  \mu^\tau_j\rangle_\l=0$ for all $j$ $\in$ $\{1,...,d_V \}$; so $\mu^\tau_j=0$ for all $j$ and therefore, $\psi=0$.

Let $\phi=\displaystyle\sum^m_{k=1} T_k\otimes_B \mu_k$ $\in$ $E^V_\l\otimes_B\Omega^\bullet(B)$ such that $\langle\phi\mid\phi\rangle^\bullet_\r=0$.  Thus, $\eta:=\widehat{\Upsilon}^{-1}_{V}(\phi)=\displaystyle \sum^m_{k=1} T_k\,\mu_k$ $\in$ $\Mor(\delta^V,\Delta_\H)$ and by equation (\ref{3.f7.5}) we have $$\phi=\displaystyle\sum^m_{k=1} T_k\otimes_B\mu_k =\sum_{k} T^\r_k\otimes_B (\mu^{\eta\ast}_k)^\ast,$$ where $\mu^{\eta\ast}_k=\displaystyle \sum_{i}\eta(e_i)^\ast\,x^{\overline{V}\,\ast}_{ki}$. Here, $x^{\overline{V}}_{ki}=T^\l_k(e_i)$, where $\{T^\l_k\}$ is the set of generators of equation (\ref{generators}) for the $H$--corespresentation $\delta^{\overline{V}}$ (the conjugate corepresentation of $\delta^V$), and $T^\r_k=T^{\l\,\ast}_k$. Hence, by equations (\ref{ashcas1}), (\ref{generators}) we obtain

\begin{eqnarray*}
\langle\phi,\phi\rangle^\bullet_\r =  \sum_{k,j} \langle T^\r_k\otimes_B (\mu^{\eta\ast}_k)^\ast ,  T^\r_j \otimes_B (\mu^{\eta\ast}_k)^\ast\rangle^\bullet_\r
  &=& 
  \sum_{k,j}
  \langle   \mu^{\eta\ast}_k, (( T^\r_k, T^\r_j)_\r\; (\mu^{\eta\ast}_j)^\ast)^\ast\rangle_\l 
  \\
  &= &
  \sum_{k,j}
  \langle   \mu^{\eta\ast}_k, \mu^{\eta\ast}_j\, ( T^\r_k, T^\r_j)^\ast_\r\; \rangle_\l 
  \\
  &= &
  \sum_{k,j,l}
  \langle   \mu^{\eta\ast}_k, \mu^{\eta\ast}_j\, T^\r_j(e_l)^\ast\, T^\r_k(e_l)\; \rangle_\l 
  \\
  &= &
  \sum_{k,j,l}
  \langle   \mu^{\eta\ast}_k, \mu^{\eta\ast}_j\,x^{\overline{V}}_{jl}\, x^{\overline{V}\,\ast}_{kl}\, \rangle_\l 
  \\
  &= &
 \sum_{k,j,l,i}
  \langle   \mu^{\eta\ast}_k, \eta(e_i)^\ast\,x^{\overline{V}\,\ast}_{ji}\,x^{\overline{V}}_{jl}\, x^{\overline{V}\,\ast}_{kl}\, \rangle_\l  
  \\
  &= &
\sum_{k,l,i}
  \langle   \mu^{\eta\ast}_k, \eta(e_i)^\ast\,\delta_{il}\, x^{\overline{V}\,\ast}_{kl}\, \rangle_\l 
\\
  &= &
 \sum_{k,i}
  \langle   \mu^{\eta\ast}_k, \eta(e_i)^\ast\, x^{\overline{V}\,\ast}_{ki}\, \rangle_\l 
 \\
  &= &
 \sum_{k}\langle \mu^{\eta\ast}_k ,  \mu^{\eta\ast}_k\rangle_\l.
\end{eqnarray*}
In this way $$0=\langle\phi\mid\phi\rangle^\bullet_\r=\int_B \langle\phi,\phi\rangle^\bullet_\r \,\dvol= \sum_{k} \int_B\langle \mu^{\eta\ast}_k ,  \mu^{\eta\ast}_k\rangle_\l\,\dvol,$$ which implies that   $\langle \mu^{\eta\ast}_k ,  \mu^{\eta\ast}_k\rangle_\l=0$ for all $k$; so $\mu^{\eta\ast}_k=0$ and therefore, $\phi=0$.
\end{proof}

\section{}

In this appendix we are going to deal with adjointability on finitely generated projective modules.

Let $(A,\cdot,\mathbbm{1},\ast)$ be a $\ast$--algebra. The map 
\begin{equation}
\label{canoleft2}
(-,-)^{d}_\l:A^{d}\times A^{d} \longrightarrow A,\qquad (\overline{x},\overline{y})\longmapsto ( \overline{x}, \overline{y})^{d}_\l:=\sum^{d}_{i=1} x_i\,y^\ast_i
\end{equation}
is the \emph{canonical} non--degenerate Hermitian structure on the free left $A$--module $$A^d=\underbrace{A\times A\times \cdots \times A}_{d-times} $$ Here, $\overline{x}=(x_1,...,x_d)$, $\overline{y}=(y_1,...,y_d)$ with $x_i$, $y_i$ $\in$ $A$.

If $M_d(A)$ denotes the space of $d\times d$ matrices with entries in $A$ and $\End(A^d)$ denotes the space of all left $A$--module endomorphisms of $A^d$, then $$\End(A^d)\cong M_d(A)$$ as rings. In fact, if $\beta:=\{\overline{e}_i \}^d_{i=1}$ is the canonical basis of $A^d$, then the previous ring isomorphism is given by $$C \longleftrightarrow C_\beta:=(C_{ij}),$$ where $C(\overline{e}_i)=\displaystyle \sum^d_{j=1} \overline{e}_j\,C_{ij}$. Furthermore, for every $\overline{x}$ $\in$ $A^d$ we have $$C(\overline{x})=\overline{x}\cdot C_\beta \qquad \mbox{ and }\qquad U(C(\overline{x}))=\overline{x}\cdot C_\beta \cdot U_\beta $$ for every $C$, $U$ $\in$ $\End(A^d)$. In addition,  for every $C_\beta$ $\in$ $M_d(A)$, a direct calculation shows that $$(\overline{x}\cdot C_\beta,\overline{y})^d_\l=(\overline{x},\overline{y}\cdot C^\dagger_\beta)^d_\l,$$ i.e., $C\cong C_\beta$ is $A$--adjointable with respect to $(-,-)^{d}_\l$. Here, $\dagger$ denotes the composition of the $\ast$ operation with the usual matrix transposition; and the $A$--adjoint operator of $C$ is the left $A$--module morphism $$C^{\star}:A^d\longrightarrow A^d$$ induced by the matrix $C^{\dagger}_\beta$.

Let  $\rho$ $\in$ $\End(A^d)$ be an idempotent element. Then $$\rho(A^d)=A^d \cdot \rho_\beta$$ is a finitely generated projective left $A$--module (\cite{lan}). Moreover, if $\rho_\beta=\rho^\dagger_\beta$, the map $(-,-)^{d}_\l$ induces a non--degenerate Hermitian structure (Section 4.3 Proposition 22 of reference \cite{lan})
$$(-,-)_\l:\rho(A^d) \times \rho(A^d) \longrightarrow A.$$ 

\begin{Proposition}
\label{a.2}
    Every left $A$--module endomorphism of $\rho(A^d)$ is $A$--adjointable  with respect to $(-,-)_\l$.
\end{Proposition}
\begin{proof}
    Let $C: \rho(A^d)\longrightarrow \rho(A^d)$ be a left $A$--module morphism and consider the map
    \begin{equation*}
        C^{\mathrm{ext}}: A^d\longrightarrow A^d,\qquad
        \overline{a} \longmapsto C(\rho(\overline{a})).
    \end{equation*}
    Since $C$ and $\rho$ are left $A$--linear maps, it follows that $C^{\mathrm{ext}}$ $\in$ $\End(A^d)$. We claim $$C=\rho\circ C^{\mathrm{ext}}\circ \rho.$$ Indeed, let $\overline{x}$ $\in$ $\rho(A^d)$. Then  $\overline{x}=\rho(\overline{a})$ for some $\overline{a}$ $\in$ $A^d$ and hence $ C(\overline{x})=C(\rho(\overline{a}))$. On the other hand,
    \begin{eqnarray*}
        (\rho\circ C^{\mathrm{ext}}\circ \rho)(\overline{x})=(\rho\circ C^{\mathrm{ext}}\circ \rho)(\rho(\overline{a}))=\rho(C^{\mathrm{ext}}(\rho(\rho(\overline{a}))))&=&\rho(C^{\mathrm{ext}}(\rho(\overline{a})))
        \\
        &=&
        \rho(C(\rho(\rho(\overline{a}))))
        \\
        &=&
        \rho(C(\rho(\overline{a}))).
    \end{eqnarray*}
   However, $\Im(C)$ $\subseteq$ $\rho(A)$, so $\rho(C(\rho(\overline{a})))=C(\rho(\overline{a}))=C(\overline{x})$ and therefore $$C(\overline{x})=(\rho\circ C^{\mathrm{ext}}\circ \rho)(\overline{x}).$$ This proves our claim. 

   In this way, by considering the ring isomorphism $\End(A^d)\cong M_d(A)$ we get $$C\cong \rho_\beta\cdot C^{\mathrm{ext}}_\beta\cdot \rho_\beta$$ and since $\rho_\beta=\rho^\dagger_\beta$, we obtain $$(\overline{a}\cdot\rho_\beta\cdot (\rho_\beta\cdot C^{\mathrm{ext}}_\beta \cdot \rho_\beta),\overline{b}\cdot\rho_\beta)_\l=(\overline{a}\cdot\rho_\beta ,\overline{b}\cdot\rho_\beta \cdot(\rho_\beta\cdot C^{\mathrm{ext}\,\dagger}_\beta \cdot \rho_\beta))_\l. $$ Hence, the induced map $$C^{\star}: \rho(A^d)\longrightarrow \rho(A^d)$$ of the matrix $\rho_\beta\cdot C^{\mathrm{ext}\,\dagger}_\beta \cdot \rho_\beta$ is the $A$--adjoint operator of $C$.
\end{proof}

Let $M$ be a finitely generated projective left $A$--module. According to Section 4.2 of reference \cite{lan}, there exists $d$ $\in$ $\N$ and an idempotent element $\rho$ $\in$ $\End(A^d)$ such that $$M\cong\rho(A^d)=A^d\cdot \rho_\beta$$ as left $A$--modules. By the last proposition, it follows that

\begin{Proposition}
    \label{a.3}
    If $\rho_\beta=\rho^\dagger_\beta$, then every left $A$--module endomorphism $C$ of $M$ is $A$--adjointable with respect to the Hermitian structure induced by $(-,-)^{d}_\l$. The $A$--adjoint operator of $C$ will be denoted by $C^{\star}$.
\end{Proposition}

Consider the free left $A$--module $A^d\otimes_A A^s$. This module is canonically isomorphic to $A^{ds}$ by means of $$\overline{\beta}_i\otimes_A \overline{\alpha}_j\;\;\longleftrightarrow\;\; \overline{e}_{(i-1)s+j},$$ where $\{\overline{e}_k\}^{ds}_{k=1}$, $\{\overline{\beta}_k\}^{d}_{k=1}$, $\{\overline{\alpha}_k\}^{s}_{k=1}$ are the canonical basis of $A^{ds}$, $A^{d}$, $A^{s}$, respectively. In this way, the canonical Hermitian structure $$(-,-)^{ds}_\l: A^{ds}\times A^{ds}\longrightarrow A^{ds}$$ defines the Hermitian structure $$(-,-)^{\otimes}_\l: (A^d\otimes_A A^s)\times (A^d\otimes_A A^s)\longrightarrow A $$ given by $$(\overline{x}\otimes_A \overline{y},\overline{u}\otimes_A \overline{v})^{\otimes}_\l:=(\overline{x}\,(\overline{y},\overline{v})^s_\l,\overline{u})^d_\l.$$

Let $M\cong \rho(A^d)=A^d\cdot \rho_\beta$, $N\cong \varsigma(A^s)=A^s\cdot \varsigma_\alpha$ be two finitely generated projective left $A$--modules with $\rho^\dagger_\beta=\rho_\beta$ and  $\varsigma^\dagger_\alpha=\varsigma$. Assume that $M\otimes_A N$ is a finitely generated projective left $A$--module too by embedding it in $A^d\otimes_A A^s\cong A^{ds}$ and its associated idempotent matrix is self--adjoint as well. Then by Proposition \ref{a.3} we get

\begin{Proposition}
    \label{a.6}
    Every left $A$--module morphism $$C: M\otimes_A N\longrightarrow M\otimes_A N$$ is $A$--adjointable with respect to the induced Hermitian structure of $(-,-)^{\otimes}_\l$ on $M\otimes_A N$.
\end{Proposition}

Of course, there are similar results for right $A$--modules and the \emph{canonical} right Hermitian structure on $A^d$ given by 
    \begin{equation}
    \label{canoright2}
     (-,-)^{d}_\r:A^{d}\times A^{d} \longrightarrow A,\qquad (\overline{x},\overline{y})\longmapsto (\overline{x},\overline{y})^{d}_\r:=\displaystyle  \sum^{d}_{i=1} x^\ast_i\,y_i.
    \end{equation}
 In particular, consider the free right $A$--module $A^d\otimes_A A^s$. This module is canonically isomorphic to $A^{ds}$ by means of $$\overline{\beta}_i\otimes_A \overline{\alpha}_j\;\;\longleftrightarrow\;\; \overline{e}_{(i-1)s+j}.$$ In this way, the canonical Hermitian structure $$(-,-)^{ds}_\r: A^{ds}\times A^{ds}\longrightarrow A^{ds}$$ defines the Hermitian structure $$(-,-)^{\otimes}_\r: (A^d\otimes_A A^s)\times (A^d\otimes_A A^s)\longrightarrow A $$ given by $$(\overline{x}\otimes_A \overline{y},\overline{u}\otimes_A \overline{v})^{\otimes}_\r:=(\overline{x},\,(\overline{y},\overline{v})^s_\r\,\overline{u})^d_\r.$$

Let $M\cong \rho(A^d)=\rho_\beta\cdot A^d $, $N\cong \varsigma(A^s)=\varsigma_\alpha \cdot A^s$ be two finitely generated projective right $A$--modules with $\rho^\dagger_\beta=\rho_\beta$ and $\varsigma^\dagger_\alpha=\varsigma$. Assume that $M\otimes_A N$ is a finitely generated projective right $A$--module too by embedding it in $A^d\otimes_A A^s\cong A^{ds}$ and its associated idempotent matrix is self--adjoint as well. Then we have

\begin{Proposition}
    \label{a.7}
    Every right $A$--module morphism $$C: M\otimes_A N\longrightarrow M\otimes_A N$$ is $A$--adjointable with respect to the induced Hermitian structure of $(-,-)^{\otimes}_\r$ on $M\otimes_A N$.
\end{Proposition}
 Of course, the proof of the last proposition is completely similar to one of Proposition \ref{a.6}.

\section{}

In this appendix we check that the theory develop in this paper is a proper generalization in non--commutative geometry of the \emph{classical} theory of Yang--Mills scalar matter fields (\cite{gtvp}).

Let $\pi:GM\longrightarrow M$ be a principal $G$--bundle and let $\alpha^V:G\longrightarrow GL(V)$ be a unitary finite--dimensional representation. Consider the associated vector bundle $$\pi_{\alpha^V}:E^V\longrightarrow M$$ and its \emph{canonical} Hermitian structure (\cite{gtvp}) 
$$(-,-):\Gamma(E^V)\times \Gamma(E^V)\longrightarrow C^\infty_\C(M).$$ Then, the induced linear connection $\nabla^{\omega_\mathrm{class}}_V$ (see equation (\ref{dif7})) is Hermitian with respect to $(-,-)$ \cite{gtvp}.

For a complex--extension of a Riemannian metric on $M$, there is a $C^\infty_\C(M)$--valued inner product on $\C$--valued differential forms (antilinear in the first coordinate)
$$\langle-,-\rangle:\Omega^\bullet_\C(M)\times \Omega^\bullet_\C(M)\longrightarrow C^\infty_\C(M)$$ and using the \emph{usual} integration on $\Omega^4_\C(M)$, we obtain an inner product
$$\langle-|-\rangle:=\int_M \langle-,-\rangle\,\dvol: \Omega^\bullet_\C(M)\times \Omega^\bullet_\C(M)\longrightarrow \C.$$ Notice that associated with $\langle-,-\rangle$, we have the star Hodge operator  $\star$ (\cite{gtvp}).

Now, there is a $C^\infty_\C(M)$--valued inner product on $E^V$--valued differential forms
$$\langle-,-\rangle^\bullet: (\Omega^\bullet_\C(M)\otimes_{C^\infty_\C(M)}\Gamma(E^{V})) \times (\Omega^\bullet_\C(M)\otimes_{C^\infty_\C(M)}\Gamma(E^{V}))\longrightarrow C^\infty_\C(M)$$ given by  
$$\langle \mu_1\otimes_B s_1,\mu_2\otimes_B s_2\rangle^\bullet= \langle\mu_1 (s_1,s_2),\mu_2\rangle,$$ and an inner product 
$$\langle-|-\rangle^\bullet:=\int_M \langle-,-\rangle^\bullet\,\dvol: (\Omega^\bullet_\C(M)\otimes_{C^\infty_\C(M)}\Gamma(E^{V})) \times (\Omega^\bullet_\C(M)\otimes_{C^\infty_\C(M)}\Gamma(E^{V}))\longrightarrow \C.$$

Using equations (\ref{dif1}), (\ref{dif2}), we get 
\begin{equation}
    \label{c.1}
    (-,-)_Q:\Mor(\delta^V,\Delta_P)\times \Mor(\delta^V,\Delta_P)\longrightarrow C^\infty_\C(M),
\end{equation}
a $C^\infty_\C(M)$--valued inner product on  $\Omega^\bullet_\C(M)\otimes_{C^\infty_\C(M)}\Mor(\delta^V,\Delta_P)$
\begin{equation}
    \label{c.2}
    \langle-,-\rangle^\bullet_Q,
\end{equation}
and an inner product on $\Omega^\bullet_\C(M)\otimes_{C^\infty_\C(M)}\Mor(\delta^V,\Delta_P)$
\begin{equation}
    \label{c.3}
    \langle-|-\rangle^\bullet_Q:=\int_M \langle-,-\rangle^\bullet_Q\,\dvol,
\end{equation}
where $\delta^V$ is the corepresentation given by the pull--back of $\alpha^V$ and $\Delta_P$ is the pull--back of the complex--extension of the canonical $G$--action on $GM$.

Let $\omega_\mathrm{class}$ be a principal connection of $\pi:GM\longrightarrow M$. Consider $\ad_{\mathrm{class}}$ the complex--extension of the right adjoint action of $G$ on its Lie algebra $\mathfrak{g}$. Then, the complex--extension $\Omega^{\omega_\mathrm{class}}$ of the curvature  of $\omega_\mathrm{class}$ is a basic differential $2$--form of type $\ad_{\mathrm{class}}$ (\cite{nodg}), i.e., 
\begin{equation*}
    \Omega^{\omega_\mathrm{class}}\;\in\; \Omega^\bullet_\C(GM,\mathfrak{g}_\C)^G.
\end{equation*}
Then, by equation (\ref{dif3}), we have
\begin{equation*}
GP(\Omega^{\omega_\mathrm{class}})\;\in\;\Omega^2_\C(M)\otimes_{C^\infty_\C(M)}\Gamma(E^{\mathfrak{g}_\C})
\end{equation*}
and by equation (\ref{dif5}) we obtain
\begin{equation}
    \label{c.4}
    R^{\omega^\#_\mathrm{class}}:=\#(\Omega^{\omega_\mathrm{class}})\;\in\; \Mor(\ad^\#_\mathrm{class},\Delta_\Hor)^G
\end{equation}
which is the map of equation (\ref{classcur2}); and by equation  (\ref{dif6}) we obtain
\begin{equation}
    \label{c.5}
\Upsilon_{\mathfrak{g}^\#_\C}(R^{\omega^\#_\mathrm{class}})\;\in\;\Omega^2_\C(M)\otimes_{C^\infty_\C(M)}\Mor(\ad^\#_{\mathrm{class}},\Delta_P).
\end{equation}

The \emph{classical} Yang--Mills functional is defined as
\begin{equation}
\label{c.6}
    \qS_\mathrm{YM}:\mathfrak{pc}(\pi)\longrightarrow \R,\qquad \omega_\mathrm{class}\longmapsto -{1\over 2}||\Omega^{\omega_\mathrm{class}}||^{\bullet\, 2},
\end{equation}
where $$||\Omega^{\omega_\mathrm{class}}||^{\bullet\, 2}:=||GP(\Omega^{\omega_\mathrm{class}})||^{\bullet\, 2}=\langle GP(\Omega^{\omega_\mathrm{class}})\mid GP(\Omega^{\omega_\mathrm{class}})\rangle^\bullet$$ and $\mathfrak{pc}(\pi)$ is the space of principal connections of the principal $G$--bundle $\pi:GM\longrightarrow M$ \cite{gtvp}. Moreover, critical points of this functional are characterized by the \emph{classical} Yang--Mills equation (\cite{gtvp})
\begin{equation}
\label{c.6.5}
d^{\nabla^{\omega_\mathrm{class}}_{\mathfrak{g}_\C} \star}(\Omega^{\omega_\mathrm{class}})=0\qquad \mbox{ with }\qquad   d^{\nabla^{\omega_\mathrm{class}}_{\mathfrak{g}_\C}\star}(\Omega^{\omega_\mathrm{class}}):=d^{\nabla^{\omega_\mathrm{class}}_{\mathfrak{g}_\C}\star}(GP(\Omega^{\omega_\mathrm{class}})), 
\end{equation}
where 
\begin{equation}
    \label{c.7}
    d^{\nabla^{\omega_\mathrm{class}}_{\mathfrak{g}_\C}\star}
\end{equation}
is the formal adjoint operator with respect to $\langle-|-\rangle^\bullet$ of the exterior derivative of the induced linear connection (see equations (\ref{dif7}), (\ref{dif8})) \cite{gtvp}.

However, by construction, $R^{\omega^\#_\mathrm{class}}$ is equivalent to $\Omega^{\omega_\mathrm{class}}$, $\Upsilon_{\mathfrak{g}_\C}(R^{\omega^\#_\mathrm{class}})$ is equivalent to $GP(\Omega^{\omega_\mathrm{class}})$ and $\langle-|-\rangle^\bullet_Q$ is equivalent to $\langle-|-\rangle^\bullet$; hence
\begin{equation*}
    ||R^{\omega^\#_\mathrm{class}}||^{\bullet\, 2}_Q:=||\Upsilon_{\mathfrak{g}_\C}(R^{\omega^\#_\mathrm{class}})||^{\bullet\, 2}_Q=||\Omega^{\omega_\mathrm{class}}||^{\bullet\, 2}.
\end{equation*}
This implies that the \emph{classical} Yang--Mills function is equivalent to 
\begin{equation}
    \label{c.9}
     \omega^\#_\mathrm{class}\longmapsto -{1\over 2}||R^{\omega^\#_\mathrm{class}}||^{\bullet\, 2}_Q,
\end{equation}
and therefore, critical points of the functionals of equations (\ref{c.6}), (\ref{c.9}) are the same, through the bijection 
$$\{ \omega_\mathrm{class}\}\;\;\longleftrightarrow \;\; \{ \omega^\#_\mathrm{class}\}.$$ In particular, the \emph{classical} Yang--Mills equation is equivalent to 
\begin{equation}
    \label{c.10}
    d^{\nabla^{\omega^\#_\mathrm{class}}_{\mathfrak{g}_\C}\star}(R^{\omega^\#_\mathrm{class}})=0\qquad \mbox{ with }\qquad  d^{\nabla^{\omega^\#_\mathrm{class}}_{\mathfrak{g}_\C}\star}(R^{\omega^\#_\mathrm{class}}):=d^{\nabla^{\omega^\#_\mathrm{class}}_{\mathfrak{g}_\C}\star}(\Upsilon_{\mathfrak{g}_\C}(R^{\omega^\#_\mathrm{class}})),
\end{equation}
where 
\begin{equation}
    \label{c.11}
    d^{\nabla^{\omega^\#_\mathrm{class}}_{\mathfrak{g}_\C}\star}
\end{equation}
is the formal adjoint operator with respect to $\langle-|-\rangle^\bullet_Q$ of the operator of equation (\ref{dif9}), which is, by construction, equivalent to the operator of equation (\ref{c.7}).

Equations (\ref{c.9}), (\ref{c.10}) (or equivalently, equations (\ref{c.6}), (\ref{c.6.5})) are the \emph{classical} counterparts of the non--commutative geometrical Yang--Mills action and the non--commutative geometrical Yang--Mills equation. This shows that the theory of Section 4.1 is a proper generalization in non--commutative geometry of the Yang--Mills theory in differential geometry.

In particular, we can obtain equation (\ref{c.10}) from equation (\ref{6.f1.1}) just recalling that for a given principal connection $\omega_\mathrm{class}$, the complex--extension of its pull--back $\omega^\#_\mathrm{class}$ is a regular and multiplicative quantum principal connection (\cite{micho1,micho2}),  and hence $$D^{\omega^\#_\mathrm{class}}=\widehat{D}^{\omega^\#_\mathrm{class}} \qquad \mbox{ and }\qquad S^{\omega^\#_\mathrm{class}}=\widehat{S}^{\omega^\#_\mathrm{class}}=0.$$ Moreover, the entire left structure is equivalent, through the canonical flip, to the right
structure and therefore, equation  (\ref{6.f1.1}) turns into
\begin{eqnarray}
\label{c.12}
    0= \,\mathrm{Re}\left( \langle d^{\nabla^{\omega^\#_\mathrm{class}}_{\mathfrak{g}_\C}\star}(R^{\omega^\#_\mathrm{class}})\mid  \Upsilon_{\mathfrak{g}_\C}(\lambda^\#_\mathrm{class})\rangle^\bullet_Q \right)
\end{eqnarray}
for all $\lambda^\#_\mathrm{class}$ $\in$ $\overrightarrow{\mathfrak{qpc}(\zeta)}$ (see equation (\ref{2.f24.2})). The notation $\lambda^\#_\mathrm{class}$ for elements of $\overrightarrow{\mathfrak{qpc}(\zeta)}$ is because in this case (the \emph{classical} case), every element of $\overrightarrow{\mathfrak{qpc}(\zeta)}$ comes from the pull--back of a basic differential $1$--form $\lambda_\mathrm{class}$ of type $\ad_\mathrm{class}$.

By equation (\ref{twcoder2}) and since $d^{\nabla^{\omega^\#_\mathrm{class}}_{\mathfrak{g}_\C}\star}(R^{\omega^\#_\mathrm{class}})$ is a $1$--form, we conclude that  $$d^{\nabla^{\omega^\#_\mathrm{class}}_{\mathfrak{g}_\C}\star}(R^{\omega^\#_\mathrm{class}}) \;\in\; \Upsilon_{\mathfrak{g}_\C}(\overrightarrow{\mathfrak{qpc}(\zeta)}).$$  Finally, since $\langle-|-\rangle^\bullet_Q$ is an inner product, equation (\ref{c.12}) implies equation (\ref{c.10}).

On the other hand, for a given principal connection $\omega_\mathrm{class}$, a unitary finite--dimensional representation $\alpha^V$ and a polynomical function ${\mathcal V}:\R\longrightarrow \R$, the \emph{classical} scalar matter functional is given by
\begin{equation}
\label{c.13}
\begin{aligned}
 \qS_\SM: \mathfrak{pc}(\pi)\times \Gamma(E^V)\longrightarrow \R,\qquad   (\omega_\mathrm{class},s)\longmapsto \dfrac{1}{2}\left(|| \nabla^{\omega_{\mathrm{class}}}_{V}s||^{\bullet \,2}-{\mathcal V}(s)\right),
\end{aligned}
\end{equation}
where $\nabla^{\omega_{\mathrm{class}}}_{V}$ is the induced linear connection (see equation (\ref{dif7})) and ${\mathcal V}(s):={\mathcal V}\circ \langle s\mid s\rangle^\bullet$. A variation of $$s\longmapsto s+s'$$ on this functional leads to (\cite{gtvp})
\begin{equation}
    \label{c.13.5}
\nabla^{\omega_{\mathrm{class}}\star}_{V}\left(\nabla^{\omega_{\mathrm{class}}}_{V}(s)\right)-{\mathcal V}'(s)s=0,
\end{equation}
where 
\begin{equation}
    \label{c.13.6}
    \nabla^{\omega_{\mathrm{class}}\star}_{V}
\end{equation}
is the formal adjoint operator of $\nabla^{\omega_{\mathrm{class}}}_{V}$ with respect to $\langle-|-\rangle^\bullet$ and ${\mathcal V}'$ is the derivative of ${\mathcal V}$; while a variation of $$\omega_\mathrm{class}\longmapsto \omega_\mathrm{class}+\lambda_\mathrm{class}$$ on this functional  leads to (\cite{gtvp})
\begin{equation}
\label{c.17}
0=\mathrm{Re}\left(\langle  {\alpha^{V}}'(\lambda_\mathrm{class})\,s\mid \nabla^{\omega_\mathrm{class}}_V \,s\rangle^\bullet\right),
\end{equation}
where $\lambda_\mathrm{class}$ is a basic differential $1$--form of type $\ad_\mathrm{class}$, and
\begin{equation}
    \label{c.17.1}
    \left({\alpha^{V}}'(\lambda_\mathrm{class})\,s\right)(X_p)=\left.{d\over dt}\right|_{t=0}\alpha^V(\mathrm{exp}(t\,\lambda_\mathrm{class}(Y_x)))\,s(p)
\end{equation}
for all $X_p$ $\in$ $T_p M$, with $Y_x$ $\in$ $T_x(GM)$ a lifting of $X_p$. It is worth mentioning that (\cite{gtvp})
\begin{equation}
    \label{c.17.2}
\nabla^{\omega_\mathrm{class}+\lambda_\mathrm{class}}_V(s)=\nabla^{\omega_\mathrm{class}}_V(s)+{\alpha^{V}}'(\lambda_\mathrm{class})\,s.
\end{equation}

Using equations (\ref{dif1}), (\ref{dif2}) and the fact that $\langle-|-\rangle^\bullet_Q$ is equivalent to $\langle-|-\rangle^\bullet$, we obtain 
$$|| \nabla^{\omega_{\mathrm{class}}}_{V}s||^{\bullet \,2}=|| \nabla^{\omega^\#_{\mathrm{class}}}_{V}T||^{\bullet \,2}_Q \qquad \mbox{ and }\qquad {\mathcal V}(s)={\mathcal V}(T),$$
where $$T=\#(GP^{-1}_0(s))\;\in\;\Mor(\delta^V,\Delta_P),$$ the operator $\nabla^{\omega^\#_{\mathrm{class}}}_{V}$ is the one of equation (\ref{dif9}) and ${\mathcal V}(T):={\mathcal V}\circ \langle T\mid T\rangle^\bullet_Q$. This implies that the \emph{classical} scalar matter action is equal to
\begin{equation}
\label{c.14}
\begin{aligned}
  (\omega^\#_\mathrm{class},T)\longmapsto \dfrac{1}{2}\left(|| \nabla^{\omega^\#_{\mathrm{class}}}_{V}T||^{\bullet \,2}_Q-{\mathcal V}(T)\right),
\end{aligned}
\end{equation}
and therefore, critical points of the functionals of equations (\ref{c.13}), (\ref{c.14}) are the same, through the bijection 
$$\{ (\omega_\mathrm{class},s)\}\;\;\longleftrightarrow \;\; \{ (\omega^\#_\mathrm{class},T)\}.$$ In particular, a variation of $$T+U$$ on the functional of equation (\ref{c.14}) leads to
\begin{equation}
    \label{c.15}
\nabla^{\omega^\#_{\mathrm{class}}\star}_{V}\left(\nabla^{\omega^\#_{\mathrm{class}}}_{V}(T)\right)-{\mathcal V}'(T)T=0,
\end{equation}
where 
\begin{equation}
    \label{c.16}
    \nabla^{\omega^\#_{\mathrm{class}}\star}_{V}
\end{equation}
is the formal adjoint operator of $\nabla^{\omega^\#_{\mathrm{class}}}_{V}$ with respect to $\langle-|-\rangle^\bullet_Q$, which is, by construction, equivalent to the operator of equation (\ref{c.13.6}); while a variation of $$\omega^\#_\mathrm{class}\longmapsto \omega^\#_\mathrm{class}+\lambda^\#_\mathrm{class} $$
on the functional of equation (\ref{c.14}) leads to 
\begin{equation}
\label{c.17.3}
0=\mathrm{Re}\left(\langle  \nabla^{\omega^\#_\mathrm{class}}_V \,T\mid \Upsilon_V(K^{\lambda^\#_\mathrm{class}}(T))\rangle^\bullet_Q\right).
\end{equation}
Notice that by comparing equation (\ref{c.17.2}) with the first part of equation (\ref{nedded1}), the operator equivalent to ${\alpha^{V}}'(\lambda_\mathrm{class})\,s$ in terms of the space $\Mor(\delta^V,\Delta_P)$ is $$\Upsilon_V(K^{\lambda^\#_\mathrm{class}}(T)). $$

Finally, combining all the results of this appendix, we conclude that the \emph{classical} Yang--Mills scalar matter functional
\begin{equation}
    \label{c.18}
    \qS_{\YMSM}=\qS_\YM+\qS_\SM
\end{equation}
is equivalent to the functional
\begin{equation}
    \label{c.19}
    (\omega^\#_\mathrm{class},T)\longmapsto {1\over 2}\left(-||R^{\omega^\#_\mathrm{class}}||^2_Q+|| \nabla^{\omega^\#_{\mathrm{class}}}_{V}T||^{\bullet \,2}_Q-{\mathcal V}(T) \right)
\end{equation}
and therefore, critical point of the functionals of equations (\ref{c.18}), (\ref{c.19}) are the same,  through the bijection 
$$\{ (\omega_\mathrm{class},s)\}\;\;\longleftrightarrow \;\; \{ (\omega^\#_\mathrm{class},T)\}.$$ It follows that the
\emph{classical} Yang--Mills scalar matter equations (\cite{gtvp})
\begin{equation}
    \label{c.20}
    \begin{aligned}
        0&=\mathrm{Re}\left( \langle d^{\nabla^{\omega_\mathrm{class}}_{\mathfrak{g}_\C}\star}(\Omega^{\omega_\mathrm{class}})\mid  GP(\lambda_\mathrm{class})\rangle^\bullet-\langle  {\alpha^{V}}'(\lambda_\mathrm{class})\,s\mid \nabla^{\omega_\mathrm{class}}_V \,s\rangle^\bullet \right)\\
0&=\nabla^{\omega_{\mathrm{class}}\star}_{V}\left(\nabla^{\omega_{\mathrm{class}}}_{V}(s)\right)-{\mathcal V}'(s)\,s
    \end{aligned}
\end{equation}
are equivalent to
\begin{equation}
    \label{c.21}
    \begin{aligned}
        0&=\mathrm{Re}\left( \langle d^{\nabla^{\omega^\#_\mathrm{class}}_{\mathfrak{g}_\C}\star}(R^{\omega^\#_\mathrm{class}})\mid  \Upsilon_{\mathfrak{g}_\C}(\lambda^\#_\mathrm{class})\rangle^\bullet_Q-  \langle  \nabla^{\omega^\#_\mathrm{class}}_V \,T\mid \Upsilon_V(K^{\lambda^\#_\mathrm{class}}(T))\rangle^\bullet_Q \right)\\
0&=\nabla^{\omega^\#_{\mathrm{class}}\star}_{V}\left(\nabla^{\omega^\#_{\mathrm{class}}}_{V}(T)\right)-{\mathcal V}'(T)\,T.
    \end{aligned}
\end{equation}

As in the case of the \emph{classical} Yang--Mills equation,  equation~(\ref{c.21}) can be obtained from equations~(\ref{6.f1.4}), (\ref{6.f1.5}), and~(\ref{6.f1.6}) by recalling that $\omega^\#_\mathrm{class}$ is a regular and multiplicative quantum principal connection, and that, in this particular case, the entire left structure is equivalent to the right structure via the canonical flip. 

We conclude that the theory of Section 4.2 is a proper
generalization in non--commutative geometry of the Yang–Mills scalar matter theory in differential geometry.\\
\end{appendix}

\end{document}